\documentclass[hidelinks,onefignum,onetabnum]{siamart220329}

\usepackage{amsfonts}
\usepackage{amssymb}
\usepackage{amsopn}
\usepackage{cleveref}
\usepackage{algorithm} 
\usepackage{algpseudocode}
\usepackage{xcolor}
\usepackage{subfig}
\usepackage{graphicx} %

\newtheorem{example}{Example}[section]
\usepackage[margin=1in]{geometry}
\graphicspath{{./figures/}{./figs}}
\newcommand{\floor}[1]{\left\lfloor #1 \right\rfloor}

\algrenewcommand\algorithmicrequire{\textbf{Input:}}
\algrenewcommand\algorithmicensure{\textbf{Output:}}

\headers{Preconditioning Low Rank GMRES for Matrix Differential Equations}{S. Meng, D. Appel\"{o}, and Y. Cheng}

\title{Preconditioning Low Rank Generalized Minimal Residual Method (GMRES) for Implicit Discretizations of Matrix Differential Equations\thanks{Submitted to the editors \today.
}}

\author{Shixu Meng\thanks{Department of Mathematics, Virginia Tech,
Blacksburg, VA 24061 U.S.A.
  (\email{sgl22@vt.edu}).}
\and Daniel Appel\"{o}
\thanks{Department of Mathematics, Virginia Tech,
Blacksburg, VA 24061 U.S.A. (\email{appelo@vt.edu}) Research supported by DOE Office of Advanced Scientific Computing Research under the Advanced Research in Quantum Computing program, Award Number DE-SC0025424, NSF DMS-2208164, and Virginia Tech. 
 }
\and
 Yingda Cheng
\thanks{Department of Mathematics, Virginia Tech,
Blacksburg, VA 24061 U.S.A.  (\email{yingda@vt.edu})
 {\tt yingda@vt.edu}  Research supported by   DOE grant DE-SC0023164 and Virginia Tech.}}

\begin{document}

\maketitle

\begin{abstract}
This work proposes a new class of preconditioners for the low rank Generalized Minimal Residual Method (GMRES) for multiterm matrix equations arising from implicit timestepping of linear matrix differential equations. We are interested in  computing low rank solutions to  matrix equations, e.g. arising from spatial discretization of stiff partial differential equations (PDEs). The low rank GMRES method is a particular class of Krylov subspace method where the iteration  is performed on the low rank factors of the solution. Such methods can exploit the low rank property of the solution to save on computational and storage cost.

Of critical importance for the efficiency and applicability of the low rank GMRES method is the availability of an effective low rank preconditioner that operates directly on the low rank factors of the solution and that can limit the iteration count and the maximal Krylov rank. 

The preconditioner we propose here is based on the basis update and Galerkin (BUG) method, resulting from the dynamic low rank approximation. It is a nonlinear preconditioner for the low rank GMRES scheme that naturally operates on the low rank factors. Extensive numerical tests show that this new preconditioner is highly efficient in limiting iteration count and maximal Krylov rank.  We show that the preconditioner performs well  for general diffusion equations including highly challenging problems, e.g. high contrast, anisotropic equations. Further, it compares favorably with the state of the art exponential sum preconditioner. We also propose a hybrid BUG - exponential sum preconditioner based on alternating between the two preconditioners. 
\end{abstract}

\section{Introduction}

We are interested in computing low rank solutions from  implicit discretizations of linear matrix differential equations. The model equation takes the form

\begin{equation}
\label{eqn:mode}
\frac{d}{dt}X(t)=F(X(t),t),  \quad X(t) \in \mathbb{R}^{m_1\times m_2}, \quad X(0)=X_0,
\end{equation}
where 
\[
F(X(t),t)=\sum_{j=1}^s A_j X(t) B_j^T +G(t),
\] 
and $A_j \in \mathbb{R}^{m_1\times m_1}, B_j \in \mathbb{R}^{m_2\times m_2}$ are sparse or structured matrices where a fast matrix-vector product is known. Further, $G(t) \in  \mathbb{R}^{m_1\times m_1} $ is a given function that is assumed to have a known low rank decomposition. The integer $s$ is assumed not to be too large. 
Equation \eqref{eqn:mode} arises in many applications governed by partial differential equations (PDE) after spatial discretizations have been deployed. Although the present paper only concerns matrix equations, we note that the methodology here was designed with low rank tensor approximation as a future research direction.

Using the method-of-lines approach and by implicit temporal discretizations, we arrive at a mutiterm linear matrix equation of the following form
\begin{equation}
\label{eq:lme}
    C_1 X D_1^T+ C_2 X D_2^T + \cdots C_k X D_k^T := \mathcal{A} X= b,
\end{equation}
where  $\mathcal{A}: \mathbb{R}^{m_1\times m_2} \mapsto \mathbb{R}^{m_1\times m_2}$ is the associated linear operator.
 If \eqref{eq:lme} has special structure (for example when it has Lyapunov or Sylvester form), then specialized solvers have been developed \cite{simoncini2016computational}. 
However, in the most generic form \eqref{eq:lme}, 
 the standard algorithm    consists of transforming the matrix above into vector form
\begin{equation}
\label{eq:lmev}
 (D_1 \otimes C_1+  \cdots+ D_k \otimes C_k) {\rm vec} (X) := \mathcal{A}_v {\rm vec} (X) = {\rm vec} (b)
\end{equation}
and then using numerical linear algebra tools to solve \eqref{eq:lmev}. Here, $\otimes$ denotes the standard matrix Kronecker product and ${\rm vec}(X)  \in \mathbb{R}^{m_1 m_2}$ denotes the vectorized form of the matrix $X.$ The drawback of this approach is that special structure of the matrix $X$,  which can potentially save computational storage and cost, will be lost. 

In many applications, the solution to \eqref{eqn:mode}   is of low rank or can be well approximated by a low rank matrix, and we are interested in numerical methods  that can exploit this. A prominent class of such methods is low rank Krylov methods, which use low rank truncation within a standard Krylov method for \eqref{eq:lme} or \eqref{eq:lmev}. Such ideas can also be generalized to tensor linear equations when truncation in specific tensor format is used. For matrix equations,  we mention the work of  low rank truncation by a greedy approach combined with the Galerkin method \cite{kressner2015truncated} and low rank conjugate gradient  \cite{simoncini2023analysis} and the references within. For tensor equations, there are various versions of low rank Krylov solvers, presented in \cite{kressner2010krylov, kressner2011low, tobler2012low}, 
and tensor train Generalized Minimal Residual Method (GMRES) appear in \cite{dolgov2013tt, coulaud2022robust}. In particular, \cite{rodgers2023implicit} presented the tensor train-GMRES or hierarchical Tucker-GMRES for the implicit discretization of nonlinear tensor differential equations. There are several challenges associated with low rank Krylov methods. Low rank truncation results in loss of orthogonality in the Krylov subspace \cite{palitta2021convergence, simoncini2023analysis} which makes analysis of convergence difficult. Also the key in achieving computational efficiency is to find a  preconditioner to achieve fast convergence  to avoid excessive growths of the numerical ranks during intermediate iterations \cite{kressner2011low}. Without a good preconditioner, the intermediate rank can easily explode which defeats the purpose  of using low rank approximations.
Note that preconditioners for low rank methods are only allowed to operate on the low rank factors of the solution \cite{grasedyck2013literature,bachmayr2023low}. A prominent technique is to write the approximate inverse as sum of Kronecker products, for example as in the popular exponential sum (ES) preconditioner \cite{hackbusch2006low,hackbusch2006low2}. Another popular methos is the multigrid method, which was used in  e.g. \cite{ballani2013projection, tobler2012low}.  Some other ideas include low-rank tensor diagonal preconditioners for wavelet discretizations  \cite{bachmayr2012adaptive} and references mentioned within \cite{grasedyck2013literature,bachmayr2023low}.

An important class of methods for computing low rank solutions to time-dependent problems is based on the dynamic low rank approximation (DLRA) \cite{koch}. There, by projecting the equation onto the tangent space of the low rank manifold, we can constrain the solution to the fixed rank manifold. In the work of Ceruti et al. \cite{ceruti2022unconventional}, the authors proposed the basis update and Galerkin (BUG) method to implement DLRA. This method is shown to have many good properties compared to earlier versions of DLRA schemes, and have been applied and generalized to many applications. However, it is limited to first order in time and the DLRA framework in general incurs the modeling error (the error due to the tangent projection) which cannot be estimated or bounded \cite{mBUG,lam2024randomized}.

This work is motivated by the success of the BUG method as a DLRA for time-dependent problems. We propose to use BUG as a novel preconditioner for low rank GMRES method. This preconditioner in itself can be easily defined according to the BUG procedure, resulting in a non-standard nonlinear preconditioner. We discuss how to effectively incorporate this preconditioner in various types of implicit high order time steppers and how to choose the optimal parameter in its implementations. We verify the performance of the preconditioner by comparing it to the ES preconditioner, and propose a hybrid version of them. Numerical examples including the chanllenging high contrast, anisotropic equations are considered. 

The rest of the paper is organized as follows: in Section \ref{sec:lrgmres}, we review the low rank GMRES method in its various form. Section \ref{sec:} reviews the  BUG method. In Section \ref{sec:new}, we present the new BUG  preconditioner and its variants. Section \ref{sec:numerical} contains simulation results. We conclude the paper in Section \ref{sec:conclude}.

\section{Review:  low rank GMRES}
\label{sec:lrgmres}
In this section, we will review the low rank GMRES method for solving matrix equation \eqref{eq:lme}. 
The idea is simple, to equip the well-established method GMRES method \cite{saad1986gmres} by the low rank truncation.
Note that there are several variants of low rank GMRES  method \cite{ballani2013projection, dolgov2013tt,coulaud2022robust}. Our presentation closely follows the recent work in \cite{coulaud2022robust} for tensor train-GMRES.

\subsection{GMRES}

We begin with the review of the Modified Gram-Schmidt GMRES (MGS-GMRES) for solving linear systems $Ax=b$ with $A \in \mathbb{R}^{n\times n}$, $x \in \mathbb{R}^{n}$, and $b \in \mathbb{R}^{n}$. 
Starting from the initial guess $x_0$, MGS-GMRES constructs a series of approximations $x_k = x_0 + r_k$ in Krylov subspace 
$$
r_k = \mbox{argmin}_{r \in \mathcal{K}_k(A,r_0) } \|r_0- Ar\|, 
$$
with $r_0 = b - Ax_0$ and 
$$
\mathcal{K}_k(A,r_0) = \mbox{span} \left\{ r_0, Ar_0, \cdots, A^{k-1} r_0  \right\}.
$$
To find the minimal residual, a matrix $V_k = [v_1,\cdots, v_k] \in \mathbb{R}^{n\times k}$ with orthogonal columns and an upper Hessenberg matrix $\overline{H}_k \in \mathbb{R}^{(k+1)\times k}$ are iteratively constructed using the Arnoldi procedure such that $\mbox{span} {V_k} = \mathcal{K}_k(A,r_0)$ and
$$
AV_k = V_{k+1} \overline{H}_k, \quad \mbox{with} \quad V^T_{k+1} V_{k+1}  = I_{k+1}.
$$
To proceed, one finds $x_k = x_0 + r_k$ where $r_k = V_k y_k$ and
$$
y_k = \mbox{argmin}_{y \in \mathbb{R}^k} \|\beta e_1- \overline{H}_k y \|, 
$$
where $\beta = \|r_0\|$ and $e_1 = (1,0,\cdots,0)^T \in \mathbb{R}^{k+1}$ so that 
$$
\| \beta e_1- \overline{H}_k y_k\| = \|Ar_k - r_0\| = \|b-Ax_k\|.
$$

For the stopping criteria, we follow the the norm-wise backward error for both vector and tensor  linear systems \cite{coulaud2022robust}
$$
\eta_{A,b}(x_k) = \frac{\|Ax_k - b\|}{\|A\|_2\|x_k\|+\|b\|}.
$$ 
Now we summarize the MGS-GMRES algorithm in \Cref{Algorithm: MGS-GMRES}.
\begin{algorithm} 
	\caption{MGS-GMRES}\label{Algorithm: MGS-GMRES}
\begin{algorithmic}[1]
\Require $A$, $b$, $x_0$, GMRES tolerance $\delta$,  maximal number of iterations $m$
\Ensure Approximate solution $x_m$
\State $r_0 = b - A x_0$, $\beta = \|r_0\|$, $v_1 = r_0/\beta$.
\For{$k=1,2,\cdots,m$} 
 \State $w = A v_k$
    \For{$i=1,2,\cdots,k$} 
    \State $\overline{H}_{i,k} = \langle v_i,w \rangle$
    \State $w =w - \overline{H}_{i,k} v_i$
 \EndFor
 \State $\overline{H}_{k+1,k} = \|w\|$
\State $v_{k+1} = w/\overline{H}_{k+1,k}$
\State $y_k = {\rm argmin}_{y\in \mathbb{R}^k} \| \beta e_1 - \bar{H}_k y \|$
\State $x_k = x_0+V_k y_k$
		\If {$\eta_{A,b}(x_k)\le \delta$}
        \State Break
        \EndIf
\EndFor 
\end{algorithmic}
\end{algorithm} 

\subsection{Low rank GMRES}
The idea of the low rank GMRES (lrGMRES) for matrix linear systems to equip   MGS-GMRES with the low rank truncation. In the following, we review the lrGMRES for \eqref{eq:lme}.

The main idea is to keep every low rank matrix   in its singular value decomposition (SVD) form in  all computations during lrGMRES, 
$$
X = USV^T,
$$
where $[U,S,V]$ is the SVD of $X \in \mathbb{R}^{m_1 \times m_2}$ where $U \in \mathbb{R}^{m_1 \times r}$, $S \in \mathbb{R}^{r \times r}$, and $V \in \mathbb{R}^{m_2 \times r}$. In particular, we present in  \Cref{Algorithm: lrGMRES} the computation of  $X=U_xS_xV_x^T$ to
$$
\mathcal{A}(U_xS_xV_x^T) = U_bS_bV_b^T,
$$
with a given linear operator $\mathcal{A}: \mathbb{R}^{m_1\times m_2} \mapsto \mathbb{R}^{m_1\times m_2}$, and right hand side $b \in \mathbb{R}^{m_1\times m_2}$ in SVD form $b = U_bS_bV_b^T$.  Compared to \Cref{Algorithm: MGS-GMRES},  the main feature is the use of a low rank truncation $\mathcal{T}^{sum}_\epsilon$ (which is defined in Algorithm \ref{Algorithm: trunc_sum}) to prevent   rank explosion during iterations. 
  This operation is also called ``rounding" in the tensor literature. In  \Cref{Algorithm: trunc_sum}, $\mathcal{T}^{svd}_\epsilon$ simply performs a truncated SVD: for any matrix $A=USV^T$ with $S = {\rm diag}(\sigma_1,\dots,\sigma_s)$, then $\mathcal{T}^{svd}_\epsilon (A) = U[:,1:r] {\rm diag}(\sigma_1,\dots,\sigma_r) U[:,1:r]^T$ such that $\left(\sum_{j=r+1}^s \sigma_j^2\right)^{1/2} \le \epsilon$. The rounding tolerance $\epsilon$ has to be chosen smaller than or equal to the lrGMRES tolerance $\delta$ according to \cite{coulaud2022robust}. It was shown in \cite{coulaud2022robust} that this lrGMRES is backward stable for tensor linear systems.

\begin{algorithm} 
	\caption{lrGMRES}\label{Algorithm: lrGMRES}
	\begin{algorithmic}[1] 
        \Require{linear map $\mathcal{A}$, $b$ in SVD form $b=U_bS_bV_b^T$, initial guess in SVD form $x_0 = U_{x0}S_{x0}V_{x0}^T$, rounding tolerance $\epsilon$, GMRES tolerance $\delta$, maximal number of iterations $m$}
         \Ensure{$[U_{x}, S_{x}, V_{x}, {\rm hasConverged}] = \mbox{lrGMRES}(\mathcal{A},U_{x0}, S_{x0}, V_{x0},U_{b}, S_{b}, V_{b},\epsilon,\delta,m)$}
         
        \State $[U_{r0}, S_{r0}, V_{r0}] = \mathcal{T}^{sum}_\epsilon \left(b- \mathcal{A}(U_{x0}S_{x0}V_{x0}^T) \right)$, $\beta = \|S_{r0}\|$, $U_{1}= U_{r0}$, $ S_1= S_{r0}/\beta$, $ V_1= V_{r0}$.
		\For {$k=1,\dots, m$}
		\State $[U_{w}, S_{w}, V_{w}] = \mathcal{T}^{sum}_\epsilon\left( \mathcal{A}(U_kS_kV_k^T) \right)$
		\For {$i=1,\dots,k$}
		\State $\overline{H}_{i,k} = \langle U_{i}S_{i}V^T_{i},U_{w}S_{w}V^T_{w}\rangle$
        \State $[U_{w}, S_{w}, V_{w}] = \mathcal{T}^{sum}_\epsilon\left( U_{w}S_{w}V^T_{w} -\overline{H}_{i,k} U_{i}S_{i}V^T_{i}\right)$
		\EndFor
        \State $[U_{w}, S_{w}, V_{w}] = \mathcal{T}^{sum}_\epsilon\left( U_{w}S_{w}V^T_{w}\right)$
		\State $\overline{H}_{k+1,k} = \|U_{w}S_{w}V^T_{w}\|$
        \State $U_{k+1} = U_w$, $S_{k+1} = S_w/\overline{H}_{k+1,k}$, $V_{k+1} = V_w$
        \State $y_k = {\rm argmin}_{y \in \mathbb{R}^k} \| \beta e_1 - \bar{H}_k y \| $
        \State ${[U_{x}, S_{x}, V_{x}]} = \mathcal{T}^{sum}_\epsilon\left( U_{x0}S_{x0}V_{x0}^T + \sum_{j=1}^k y_k(j) U_jS_jV_j^T \right)$
		\If {$\eta_{\mathcal{A},b}({U_{x}S_{x}V^T_{x}})\le \delta$}
        \State hasConverged = True
        \State Break 
        \EndIf
		\EndFor
	\end{algorithmic} 
\end{algorithm}

\begin{algorithm}
 \caption{Truncation sum of low rank matrices $(U_j,S_j,V_j)_{j=1}^m \mapsto U,S,V$ \label{Algorithm: trunc_sum}} 
    \begin{algorithmic}[1]
    \Require{low rank matrices in the form $U_jS_jV^T_j, j =1,\dots,m$ with each $S_j$ being a low rank diagonal matrix, rounding tolerance $\epsilon$}
    \Ensure{$USV^T = \mathcal{T}^{sum}_\epsilon(\sum_{j=1}^m U_jS_jV^T_j)$}
    
    \State Form $U=[U_1,\dots ,U_m]$, $S={\rm diag}(S_1,\dots,S_m)$, $V=[V_1,\dots,V_m]$
    \State Perform column pivoted QR: $[Q_1,R_1,\Pi_1] = {\rm qr}(U)$, $[Q_2,R_2,\Pi_2] = {\rm qr}(V)$
    \State Compute the truncated SVD: $\mathcal{T}^{svd}_\epsilon (R_1\Pi_1S\Pi_2^TR_2^T) = USV$
    \State Form $U \leftarrow Q_1 U$, $V \leftarrow Q_2 V$
 	\end{algorithmic} 
\end{algorithm}

In practice, restart is needed to control the memory of the solver, i.e. maximal Krylov rank. The restarted lrGMRES is presented in \Cref{Algorithm: restarted lrGMRES}.

\begin{algorithm}
	\caption{Restarted lrGMRES}\label{Algorithm: restarted lrGMRES}
	\begin{algorithmic}[1] 
        \Require{linear map $\mathcal{A}$, $b$ in SVD form $b=U_bS_bV_b^T$, initial guess in SVD form $x_0 = U_{x0}S_{x0}V_{x0}^T$, rounding tolerance $\epsilon$, GMRES tolerance $\delta$, maximal number of iterations $m$,   restart parameter $\mbox{maxit}$}
         \Ensure{$[U_{x}, S_{x}, V_{x}, \mbox{hasConverged}] = \mbox{rlrGMRES}(\mathcal{A},U_{x0}, S_{x0}, V_{x0},U_{b}, S_{b}, V_{b},\epsilon,\delta,m,\mbox{maxit})$}
         \State hasConverged = False, $i =1$, $[U_{x}, S_{x}, V_{x}] = [U_{x0}, S_{x0}, V_{x0}]$
         \While{hasConverged = False $\&$ $i\le\mbox{maxit}$}
            \State $[U_{x}, S_{x}, V_{x}, \mbox{hasConverged}] = \mbox{lrGMRES}(\mathcal{A},U_{x}, S_{x}, V_{x},U_{b}, S_{b}, V_{b},\epsilon,\delta,m)$
            \State $i=i+1$
         \EndWhile
	\end{algorithmic} 
\end{algorithm} 

\subsection{Preconditioner}
\label{sec:es}
To speed up the convergence of lrGMRES and limit the intermediate rank growth, we consider the preconditioned lrGMRES with a right preconditioner $\mathcal{M}$. Namely, we solve $\mathcal{A} \mathcal{M} t=b$ and let $X=\mathcal{M} t.$ The preconditioned algorithms are detailed in \Cref{Algorithm: MGlrGMRES} and \Cref{Algorithm: restarted MGlrGMRES}. 

\begin{algorithm}
	\caption{Preconditioned lrGMRES}\label{Algorithm: MGlrGMRES}
	\begin{algorithmic}[1] 
        \Require{linear map $\mathcal{A}$, $b$ in SVD form $b=U_bS_bV_b^T$, initial guess in SVD form $x_0 = U_{x0}S_{x0}V_{x0}^T$, rounding tolerance $\epsilon$, GMRES tolerance $\delta$, maximal number of iterations $m$, preconditioner $\mathcal{M}$}
         \Ensure{$[U_{x}, S_{x}, V_{x}, {\rm hasConverged}] = \mbox{plrGMRES}(\mathcal{A},\mathcal{M},U_{x0}, S_{x0}, V_{x0},U_{b}, S_{b}, V_{b},\epsilon,\delta,m)$}
         
        \State $[U_{r0}, S_{r0}, V_{r0}] = \mathcal{T}^{sum}_\epsilon \left(b- \mathcal{A}(U_{x0}S_{x0}V_{x0}^T) \right)$, $\beta = \|S_{r0}\|$, $U_{1}= U_{r0}$, $ S_1= S_{r0}/\beta$, $ V_1= V_{r0}$.
		\For {$k=1,\dots, m$}
		\State $[U_{w}, S_{w}, V_{w}] = \mathcal{T}^{sum}_\epsilon\left( \mathcal{A}\mathcal{M}(U_kS_kV_k^T) \right)$
		\For {$i=1,\dots,k$}
		\State $\overline{H}_{i,k} = \langle U_{i}S_{i}V^T_{i},U_{w}S_{w}V^T_{w}\rangle$
        \State $[U_{w}, S_{w}, V_{w}] = \mathcal{T}^{sum}_\epsilon\left( U_{w}S_{w}V^T_{w} -\overline{H}_{i,k} U_{i}S_{i}V^T_{i}\right)$
		\EndFor
        \State $[U_{w}, S_{w}, V_{w}] = \mathcal{T}^{sum}_\epsilon\left( U_{w}S_{w}V^T_{w}\right)$
		\State $\overline{H}_{k+1,k} = \|U_{w}S_{w}V^T_{w}\|$
        \State {$U_{k+1} = U_w$, $S_{k+1} = S_w/\overline{H}_{k+1,k}$, $V_{k+1} = V_w$}
        \State $y_k = \mbox{argmin}_{y\in\mathbb{R}^k} \| \beta e_1 - \overline{H}_k y \|$
        \State $[U_{e}, S_{e}, V_{e}] = \mathcal{T}^{sum}_\epsilon\left(  \sum_{j=1}^k y_k(j) U_jS_jV_j^T \right)$
        \State $[U_{Me}, S_{Me}, V_{Me}] = \mathcal{T}^{sum}_\epsilon\left(   \mathcal{M}(U_{e}S_{e}V_{e}^T)\right)$
        \State ${[U_{x}, S_{x}, V_{x}]} = \mathcal{T}^{sum}_\epsilon\left( U_{x0}S_{x0}V_{x0}^T +   U_{Me}S_{Me}V_{Me}^T\right)$
		\If {$\eta_{\mathcal{A},b}(U_{x}S_{x}V^T_{x})\le \delta$}
        \State hasConverged = True
        \State Break
        \EndIf
		\EndFor
	\end{algorithmic} 
\end{algorithm} 

\begin{algorithm}
	\caption{Restarted preconditioned lrGMRES}\label{Algorithm: restarted MGlrGMRES}
	\begin{algorithmic}[1] 
        \Require{linear map $\mathcal{A}$, $b$ in SVD form $b=U_bS_bV_b^T$, initial guess in SVD form $x_0 = U_{x0}S_{x0}V_{x0}^T$, rounding tolerance $\epsilon$, GMRES tolerance $\delta$, maximal number of iterations $m$,   restart parameter $\mbox{maxit}$, preconditioner $\mathcal{M}$}
         \Ensure{$[U_{x}, S_{x}, V_{x}, \mbox{hasConverged}] = \mbox{rplrGMRES}(\mathcal{A},\mathcal{M},U_{x0}, S_{x0}, V_{x0},U_{b}, S_{b}, V_{b},\epsilon,\delta,m,\mbox{maxit})$}

         \State hasConverged = False, $i =1$, $[U_{x}, S_{x}, V_{x}] = [U_{x0}, S_{x0}, V_{x0}]$
         \While{hasConverged = False $\&$ $i\le\mbox{maxit}$}
            \State $[U_{x}, S_{x}, V_{x}, \mbox{hasConverged}] = \mbox{plrGMRES}(\mathcal{A},\mathcal{M},U_{x}, S_{x}, V_{x},U_{b}, S_{b}, V_{b},\epsilon,\delta,m)$
            \State $i=i+1$
         \EndWhile
	\end{algorithmic} 
\end{algorithm}

The design of the preconditioner for low rank scheme is nontrivial. Besides the standard design principle for preconditioners, we require it to be applied in a dimension by dimension fashion, i.e. it should only operator on  low rank factors, with   memory and computational time proportionate to the rank of the intermediate iterates. There are a few such candidates in the literature, among which the most well-knowns are ES preconditioner and multigrid preconditioner. Due to the strong dependence on rank \cite{hackbusch2015solution} the multigrid   preconditioner for low rank methods is not as effective as its counterpart for full rank methods. In this paper, we choose to benchmark our method with the ES preconditioner. The details of the ES preconditioner is provided in the Appendix.

\section{Review:  basis update and Galerkin (BUG) method}
\label{sec:}

In this section, we review an alternative approach for time discretizations of \eqref{eqn:mode} in low rank format introduced in \cite{ceruti2022unconventional}, the BUG method.  This scheme is based on DLRA \cite{koch} or Dirac–Frenkel time-dependent variational principle. It was designed to capture a fixed rank solution by projecting the equation onto the tangent space of the fixed rank low rank manifold. In particular, the DLRA solves
\begin{equation}
\label{eq:DLRAF}
\frac{d}{dt}X(t)=\Pi_{X(t)} F(X(t), t),  
\end{equation}
where $\Pi_{X(t)}$ is the orthogonal projection onto the tangent space $T_{X(t)}\mathcal{M}_r$ of the the rank $r$ matrix manifold
$\mathcal{M}_r=\{X \in  \mathbb{R}^{m_1\times m_2}, {\sf rank}(X)=r\}$ at $X(t).$   
For completeness, we include a description of the BUG method for solving \eqref{eqn:mode} in Algorithm \ref{algo:BUG}. In Algorithm \ref{algo:BUG},  first, two subproblems that come from the projector splitting of DLRA \cite{lubich2014projector} are solved in the K- and L-step, identifying the row and column  spaces.     Then, a Galerkin evolution step is performed in the resulting space, searching for the optimal solution with the Galerkin condition. 
 This method is   first order in time and fixed in rank. Later improvements to the method includes the rank adaptive version \cite{ceruti2022rank}, and
second or higher order variations \cite{Nakao:2023aa,Ceruti2024}. The BUG method performs well in many applications. However, it suffers from the modeling error from the tangent projection \cite{mBUG} and this error can not be bounded or estimated \cite{mBUG, lam2024randomized}. 

\begin{algorithm}
\begin{algorithmic}[1] 
    \Require{numerical solution at $t^n:$  rank $r$ matrix $\hat{X}^n$ in its SVD form $U^n \Sigma^n (V^n)^T.$}
    \Ensure{numerical solution at $t^{n+1}:$  rank $r$ matrix $\hat{X}^{n+1}$ in its SVD form $U^{n+1} \Sigma^{n+1} (V^{n+1})^T.$}
    	 
	\State {   Prediction.}   K-step and L-step integrating from $t^n$ to $t^{n+1}$.
	
	Solve $$ K^{n+1} - K^{n}=\Delta t F(K^{n+1} (V^n)^T,t) V^n, \quad K^n=U^n \Sigma^n,$$
	to obtain $K^{n+1},$ and $[\tilde{U},\sim,\sim]= {\rm qr}(  K^{n+1}).$
	
	Solve $$ L^{n+1}-L^n=\Delta t F(U^n (L^{n+1})^T,t)^T U^n, \quad L^n=V^n \Sigma^n,$$
	to obtain $L^{n+1},$ and  $[\tilde{V},\sim,\sim]= {\rm qr}( L^{n+1}).$  
	
\State  {  Galerkin Evolution.}  S-step: solve for $S^{n+1} $ from  
$$ S^{n+1}-S^{n}= \Delta t \tilde{U}^T F(\tilde{U} S^{n+1} \tilde{V}^T, t) \tilde{V}, \quad S^n= \tilde{U}^T U^n \Sigma^n (V^n)^T \tilde{V},$$
to obtain $S^{n+1}.$

\State {Output.} $\hat{X}^{n+1}=\tilde{U}S^{n+1} \tilde{V}^T=U^{n+1} \Sigma^{n+1} (V^{n+1})^T.$
 
    \caption{BUG integrator using implicit Euler for $\frac{d}{dt}X=F(X,t)$ \cite{ceruti2022unconventional}.}
    \label{algo:BUG}
    \end{algorithmic}
\end{algorithm}

\section{Proposed  methods: BUG preconditioner and variants}
\label{sec:new}
Motivated by the success of BUG method for many applications, we propose to use it as a preconditioner for low rank GMRES scheme. The low rank feature of BUG method makes it suitable as a preconditioner for low rank schemes. With the GMRES framework, it no long suffers the shortcomings of convergence due to the modeling error. Below, we will describe the BUG preconditioner for various implicit time discretizations. Then we will propose a hybrid preconditioner combining BUG and the standard ES preconditioners.

\subsection{BUG preconditioner and implicit time stepping}
\label{sec:bugp}
Implicit time stepping of the matrix differential equations \eqref{eqn:mode} will yield matrix equations in the form $\mathcal{A}X=b.$ This can be preconditioned by the BUG methodology as described in Algorithm \ref{Algorithm: BUG}. Starting with an input of an initial guess {$U S V^T$} with rank $r$, the preconditioner $\mathcal{M}_{\rm BUG}^{\mathcal{A}, {U, S, V}}$ computes a low rank solution $U_{\rm new} S_{\rm new} V_{\rm new}^T$ with the same rank $r$. We note that in the K-, L- and Galerkin steps, smaller systems (of size $m_1 \times r, m_2 \times r, r \times r$ ) need to be solved respectively.  With the assumption that $r$ being small, and the projected equations being better conditioned, those will be   amenable for fast computations.
We would like to emphasize that this is an unconventional preconditioner because the algorithm depends on the choice of  $U, S, V$, and when used, e.g. in Algorithm \ref{algo:imlmbugc} it is a nonlinear mapping unlike most other preconditioners.

\begin{algorithm}
	\caption{BUG preconditioner for $\mathcal{A}X=b$} \label{Algorithm: BUG}
	\begin{algorithmic}[1]
        \Require{
        linear map $\mathcal{A}$, the right hand side $b$, initial guess in SVD form $U S V^T$}
         \Ensure{$[U_{\rm new},S_{\rm new},V_{\rm new}] = \mathcal{M}_{\rm BUG}^{ \mathcal{A}, U, S, V}(b)$} 
         \State K-step:
         set $K_0 = bV$, solve
         $$
          \mathcal{A} (K_1  V^T)  V = K_0
         $$
         to obtain $[U,\sim,\sim] = {\sf qr}(K_1).$
         \State L-step: set $L_0 = b^T U$, solve
         $$
          \mathcal{A}^T (L_1  U^T)  U = L_0
         $$
        to obtain $[V,\sim, \sim] = {\sf qr} (L_1).$ Here $\mathcal{A}^T (L_1  U^T): = (\mathcal{A} (U L_1^T))^T$).
         \State Galerkin: set $\Sigma_0 = U^T b V$, solve
         $$
         U^T\mathcal{A} (U\Sigma_1 V^T)V = \Sigma_0
         $$
         and obtain $[U_c,S_c,V_c] = {\rm svd} (\Sigma_1)$.
         \State Obtain $U_{\rm new} = UU_c$, $V_{\rm new} = VV_c$, $S_{\rm new} = S_c$.
	\end{algorithmic} 
\end{algorithm} 

In this work, we consider the following time discretizations.  
\begin{itemize}
 \item
\textbf{Implicit midpoint}: in the implicit midpoint method, we solve 
\begin{equation}
\label{eq:cnm}
  X^{n+1} - \Delta t  \left( \theta \sum_{j=1}^s A_j X^{n+1} B_j^T +1/2 \sum_{j=1}^s A_j X^{n } B_j^T + G^{n}\right)= X^n,
\end{equation}
where $   G^n = G(x,y, (t_n + t_{n+1})/2)$.
\item \textbf{$l-$step BDF (backward differentiation formula) time stepping}: the   BDF time stepping leads to 
\begin{equation}
\label{eq:bdfm}
X^{n+1} - \Delta t \, \beta \left( \sum_{j=1}^s A_j X^{n+1} B_j^T +  G^{n+1} \right) =  \sum_{j=0}^l \alpha_j X^{n-j},
\end{equation}
where $\alpha_j, \beta$ are coefficients for BDF scheme.
\item \textbf{Diagonally implicit Runge–Kutta (DIRK) time stepping}: The $l-$ stage DIRK method means we solve from $i=1, \ldots l$ 
\begin{equation}
\label{eq:dirkm}
 X^{(i)} - \Delta t    \left( a_{ii} F(X^{(i)}, t^n+c_i \Delta t) + \sum_{j=1}^{i-1} a_{ij} F(X^{(j)}, t^n+c_j \Delta t) \right) = X^n,
\end{equation}
and
$$
X^{n+1} = X^{n} + \Delta t \sum_{j=1}^l b_j F(X^{(j)}, t^n+c_j \Delta t),
$$
where   the coefficients $a_{ij}, c_j, b_i$ are given by the Butcher tableau, and $F(X ,t)=\sum_{j=1}^s A_j X  B_j^T +G(t).$

\end{itemize}

\begin{algorithm}
	\caption{Implicit Midpoint Method with lrGMRES with BUG preconditioner} 
	\begin{algorithmic}[1]
        \Require{A set of second-order finite difference matrices $\{A_j,B_j\}_{j=1}^s$, time discretization parameters $\Delta t$, $n_t$ and $\theta$, initial condition in SVD form $X_0 = U_0 S_0 V_0^T$,   source $G(t)$, rounding tolerance for lrGMRES $\epsilon$, truncation tolerance $\epsilon_2,$ lrGMRES stopping criteria $\delta$, maximal number of iterations $m$,   restart parameter $\mbox{maxit}$, preconditioner $\mathcal{M}$}
         \Ensure{Solution at final time $[U,S,V]$}
         \State 
         Set the linear map $\mathcal{A}$ by $\mathcal{A}X = X - \Delta t \theta \sum_{j=1}^s A_{j} X B^T_{j} $.
         \State Initialization: $U = U_0$, $S = S_0$, $V = V_0$
         \For {$n=1,\dots, n_t$}
         \State $t = (n-1)\Delta t$
         \State $X \leftarrow USV^T$, $   G_\theta \leftarrow G(\cdot,\cdot, (1-\theta) t + \theta (t+\Delta t))$,
         $
         U_bS_bV_b^T \leftarrow \Delta t  (1-\theta) \sum_{j=1}^s A_j X B_j^T + G^{n}_\theta + X
         $ 
         \State
         $\mathcal{M} \leftarrow \mathcal{M}_{\rm BUG}^{\mathcal{A}, U, S, V}$
         
         \State $[U, S, V, \mbox{hasConverged}] \leftarrow \mbox{rplrGMRES}(\mathcal{A},\mathcal{M},U, S, V,U_{b}, S_{b}, V_{b},\epsilon,\delta,m,\mbox{maxit})$ \label{gmresline}
         \State $[U, S, V] \leftarrow \mathcal{T}^{sum}_{\epsilon_2} ( USV^T)$  \label{truncline}
         \EndFor  
	\end{algorithmic} 
 \label{algo:imlmbugc}
\end{algorithm}

These methods can be readily combined with the low rank GMRES solver with BUG preconditioner, see for example Algorithm \ref{algo:imlmbugc}, where the implicit midpoint method is provided as an example. 
In the following, we would like to discuss several key design features of the BUG preconditioner applied to matrix equations arising from the aforementioned time discretizations. 

First, we discuss  the choice of $U, S, V$ in the BUG preconditioner, and we note that $USV^T$ is also used as the initial guess of the low rank GMRES iteration. The choice is   according to the following.

\begin{itemize}
    \item For the  implicit midpoint schemes, $U S V^T=X^n$ is the numerical solution at time step $n$.
        \item For the BDF scheme, $U S V^T = \sum_{j=0}^l a_j X^{n-j}$ is   a linear combination of computed solutions at previous steps.
    \item For the DIRK scheme, at the $j-$th inner stage, we take $U S V^T=X^{n-1, (j)}$ which is the numerical solution at the $j-$th inner stage of the previous time steps. This choice  is critical for the performance of the DIRK method and will   be studied in detail in Example \ref{ex:dirk}.
\end{itemize}

For the stopping criteria, we follow \cite{coulaud2022robust} and use 
$$
\eta_{\mathcal{A},b}(x_k) = \frac{\|\mathcal{A}x_k - b\|}{\|\mathcal{A}\|_2\|x_k\|+\|b\|} \le \delta,
$$
where $\|\mathcal{A}\|_2$ is estimated as follows. Choose a set $\mathrm{W}$ of normalized matrices $w$ generated randomly from a
normal {and uniform} distribution, we estimate
$$
\|\mathcal{A}\|_2 = \max_{w \in \mathrm{W}} \|\mathcal{A} w\|.
$$
In our algorithm, we choose {$10$  normally distributed matrices and $10$  uniformly distributed matrices}. Such a stopping criteria may ensure a backward stable solution in contrast to using $\eta_{b}(x_k)=\|\mathcal{A}x_k - b\|/\|b\|$ as was pointed out by \cite{coulaud2022robust}.  The stopping parameter of lrGMRES is chosen as $\delta = \epsilon$ which was recommended in \cite{coulaud2022robust}.

Finally, we want to discuss the choice of the rounding tolerance $\epsilon_{\rm }$ in  low rank GMRES. This is chosen according to the local truncation errors of the scheme. %
In Section \ref{sec:numerical}, we provide details for such choice for various kinds of discretizations. Here, let's focus on an example of the application of Algorithm  \ref{algo:imlmbugc} for solving \eqref{eqn:mode} arising from a diffusion type of partial differential equation with second order spatial discretizations. The details of the setup and numerical results are reported in Section \ref{sec:parameterstudy}.  We assume the time step and spatial mesh is on the same order, i.e. $\Delta t=O(h),$ where $h$ is the spatial mesh in both $x$ and $y$ direction.  In this case  the local truncation error is on the order of $\Delta t (\Delta t^2 + h^2) = O(h^3)$. 
A pessimistic estimate of the local truncation error caused by an $\epsilon$ truncation in the solution matrix $X$ is 
$ 
h \epsilon (1+ \Delta t/h^2).
$
Here, we note the $h$ factor is needed to convert   matrix norm to solution norm. This is because for a 2D function $f(x,y)$, the $L^2$ norm $\sqrt{\int (f(x,y))^2 dx dy}=C h \|F\|,$ where the matrix $F_{ij}=F(x_i, y_j).$
Recall that the rounding tolerance $\epsilon$ in the lrGMRES is to make sure that its induced error is on the same order of the local truncation error. %
From $h \epsilon (1+ \Delta t/h^2)=O(h^3),$ we obtain $\epsilon = O(h^3)$. Similarly, we take $\epsilon_2=O(h^2)$ for the truncation parameter in  Line \ref{truncline} in Algorithm \ref{algo:imlmbugc} to match the local truncation error (considering the rescaling of matrix norm and solution norm). 
The $2$-norm of the operator $\mathcal{A}$   is expected to be bounded by $C\frac{\Delta t}{h^2}$ for linear diffusion type equations which are weakly perturbations of isotropic diffusion equation with constant coefficients, since for ${\rm vec}(\mathcal{A} X) = (I\bigotimes I - \frac{1}{2} \Delta t ( I \bigotimes D +  D \bigotimes I )){\rm vec}(X)$, the $2$-norm of the operator $\mathcal{A}$ is bounded by $C\frac{\Delta t}{h^2}$; here $D$ is the second-order finite difference matrix whose condition number scales as $1/h^2$.
With those parameter choices, we have the following results on stability and convergence of the numerical method.

\begin{theorem}[Stability] 
Suppose the matrix differential equation \eqref{eqn:mode} satisfies one-sided Lipschitz condition
$$
\langle F(X,t)-F(Y,t), X-Y \rangle \le \alpha \|X-Y\|^2.
$$
The implicit midpoint method with low rank GMRES scheme as described in Algorithm \ref{algo:imlmbugc} applied to a linear diffusion type problem with the property $\|\mathcal{A}\|_2 \le \frac{C \Delta t}{h^2},$ and mesh size $\Delta t=O(h),$ tolerance $\epsilon=O(h^3), \epsilon_2=O(h^2)$ is stable if Line \ref{gmresline} converges for all time steps, i.e. we have
$$
\|X^n\| \le C \|X^0\|.
$$
Here, the matrix inner product is defined as $\langle A, B \rangle=\sum_{i,j}a_{ij}b_{ij}$ and $\|\cdot\|$ denotes the Frobenious norm. $C$ is a generic constant that does not depend on mesh size $\Delta t, h.$
\end{theorem}
\begin{proof}
Let's define the intermediate solution from Line \ref{gmresline} in Algorithm \ref{algo:imlmbugc} as $\tilde{X}^{n+1}.$ Then $X^{n+1}=\mathcal{T}^{sum}_{\epsilon_2}(\tilde{X}^{n+1}),$ and $\|X^{n+1}\| \le \|\tilde{X}^{n+1}\|$ due to the property of the truncated SVD.
We denote $r^{n}=\tilde{X}^{n+1}-X^n- \Delta t F\left( (X^n+\tilde{X}^{n+1})/2, t^{n+1/2}\right),$ which is the residual from the low rank GMRES. From the stopping criteria and the fact that it has converged, and $\delta=\epsilon$, we have \begin{equation}
\label{eq:restimate}
\|r^n\| \le \epsilon (\|\mathcal{A}\|_2\|\tilde{X}^{n+1}\|+\|b\|) \le C \Delta t^2 (\||\tilde{X}^{n+1}\|+\|X^n\|),
\end{equation}
where we have used the definition of $b$ from \eqref{eq:cnm}. Now we apply the inner product with $\tilde{X}^{n+1}+X^n$ to the residual equation, and we have
\begin{eqnarray*}
&&\|\tilde{X}^{n+1}\|^2-\|X^n\|^2=\langle \tilde{X}^{n+1}-X^n, \tilde{X}^{n+1}+X^n \rangle \\
&&=
\left \langle \Delta t F\left( (X^n+\tilde{X}^{n+1})/2, t^{n+1/2}\right) +r^n, \tilde{X}^{n+1}+X^n \right \rangle \\
&& \le \alpha \Delta t/2 \|X^n+\tilde{X}^{n+1}\|^2 +\|r^n\| \|X^n+\tilde{X}^{n+1}\| \\
&& \le C \Delta t (\|X^n\|^2+\|\tilde{X}^{n+1}\|^2)
\end{eqnarray*}
where we used \eqref{eq:restimate} in the last line.
Then it is straightforward to see $\|X^{n+1}\| \le\|\tilde{X}^{n+1}\| \le (1+C\Delta t) \|X^n\| $ and the theorem follows.
\end{proof}

\begin{theorem}[Convergence] 
Under the same assumptions as in the previous  theorem,
the implicit midpoint method with low rank GMRES scheme as described in Algorithm \ref{algo:imlmbugc} is convergent of second order, i.e.
\begin{equation}
\label{eqn:errorest}
\|X^n-X(t^n)\| \le C \Delta t.
\end{equation}
with an initial condition satisfying $\|X^n-X(t^n)\| \le C h.$
Here, $X(t^n)$ denotes the exact solution at $t^n.$ $C$ is a generic constant that does not depend on mesh size $\Delta t, h.$ We notice that \eqref{eqn:errorest} implies second order because of the matrix and function norm conversion $\sqrt{\int (f(x,y))^2 dx dy}=C h \|F\|.$
\end{theorem}
\begin{proof}
We use the same notation as in the proof of the previous theorem. \eqref{eq:restimate} now gives $\|r^n\| \le C \Delta t^2$ using the stability result.

We define the error $e^n=X^n-X(t^n),$ and $\tilde{e}^n=\tilde{X}^n-X(t^n).$ 
For the exact solution, we have $$
X(t^{n+1})-X(t^n)=\Delta t F\left( (X(t^n)+X(t^{n+1}))/2, t^{n+1/2}\right)+G^n,
$$
where $G^n =C \Delta t^2$ from local truncation analysis, $\Delta t=O(h)$ and matrix and solution norm rescaling. By the linearity of $F(\cdot, t),$ we easily get
$$
\tilde{e}^{n+1}-e^n=\Delta t F\left( (e^n+\tilde{e}^{n+1}))/2, t^{n+1/2}\right)+r^n-G^n.
$$ Therefore,
 \begin{eqnarray*}
&&\|\tilde{e}^{n+1}\|^2-\|e^n\|^2=\langle \tilde{e}^{n+1}-e^n, \tilde{e}^{n+1}+e^n \rangle \\
&&=
\left \langle \Delta t F\left( (e^n+\tilde{e}^{n+1})/2, t^{n+1/2}\right) +r^n-G^n, \tilde{e}^{n+1}+e^n \right \rangle \\
&& \le \alpha \Delta t/2 \|e^n+\tilde{e}^{n+1}\|^2 +C\Delta t^2 \|e^n+\tilde{e}^{n+1}\|   \\
&& \le \alpha \Delta t (\|e^n\|^2+\|\tilde{e}^{n+1}\|^2) +\Delta t (\|e^n\|^2+\|\tilde{e}^{n+1}\|^2) + C\Delta t^3.
\end{eqnarray*}
When $\Delta t$ is small enough, we have $\|\tilde{e}^{n+1}\|^2 \le (1+C\Delta t) \|e^n\|^2 +C \Delta t^3.$ Now we use $e^{n+1}=\tilde{e}^{n+1}+X^n-\tilde{X}^n.$ Since $\|X^n-\tilde{X}^n\|\le \epsilon_2 =C \Delta t^2,$ we get $\|e^{n+1}\|^2 \le (1+C\Delta t) \|e^n\|^2 +C \Delta t^3.$ This implies \eqref{eqn:errorest} with the assumption on initial condition.
\end{proof}

We note that the theorems above works for general preconditioners, i.e. BUG and also ES preconditioners. It serves as a guideline and assurance on our parameter choice, which will be numerically verified in Section \ref{sec:numerical}.

\subsection{Hybrid preconditioner}

We  now propose a hybrid preconditioner of the ES and the BUG preconditioner.
The ES preconditioner is a well known preconditioner for low rank iterative solvers. It has a simple and direct construction, which has shown effectiveness in many high dimensional applications for tensor equations \cite{bachmayr2023low}.  The detail of this preconditioner is reported in the Appendix.  The definition of the hybrid preconditioner is straightforward. It alternates between the two preconditioners (e.g. restart=1) which is reported in Algorithm \ref{Algorithm: restarted MGlrGMRES hybrid}. We have also considered other restart parameter, but do not observe significant difference. The main motivation for the hybrid preconditioner is to use the ES to provide a better initial guess for the BUG update and vice versa. We note that this is different from the flexible GMRES framework,   because the BUG preconditioner is nonlinear.

\begin{algorithm}
	\caption{Restarted lrGMRES with hybrid preconditioner}\label{Algorithm: restarted MGlrGMRES hybrid}
	\begin{algorithmic}[1] 
        \Require{linear map $\mathcal{A}$, $b$ in SVD form $b=U_bS_bV_b^T$, initial guess in SVD form $x_0 = U_{x0}S_{x0}V_{x0}^T$, rounding tolerance $\epsilon$, GMRES  tolerance $\delta$, maximal number of iterations $m$, number of restart $\mbox{maxit}$, preconditioner $\mathcal{M}$}
         \Ensure{$[U_{x}, S_{x}, V_{x}, \mbox{hasConverged}] = \mbox{rplrGMRES}(\mathcal{A},\mathcal{M}_{\rm hybrid},U_{x0}, S_{x0}, V_{x0},U_{b}, S_{b}, V_{b},\epsilon,\delta,m,\mbox{maxit})$}
         \State hasConverged = False, $i =1$, $[U_{x}, S_{x}, V_{x}] = [U_{x0}, S_{x0}, V_{x0}]$
         \While{hasConverged = False $\&$ $i\le\mbox{maxit}$}
            \If {$~i == 1~(\mbox{mod } 2)$}
            \State $[U_{x}, S_{x}, V_{x}, \mbox{hasConverged}] = \mbox{plrGMRES}(\mathcal{A},\mathcal{M}_{\rm ES},U_{x}, S_{x}, V_{x},U_{b}, S_{b}, V_{b},\epsilon,\delta,m)$
            \Else {$~i==0~(\mbox{mod } 2)$}
            \State
            $[U_{x}, S_{x}, V_{x}, \mbox{hasConverged}] = \mbox{plrGMRES}(\mathcal{A},\mathcal{M}_{\rm BUG}^{\mathcal{A}, U_x, S_x, V_x},U_{x}, S_{x}, V_{x},U_{b}, S_{b}, V_{b},\epsilon,\delta,m)$
            \EndIf
            \State $i=i+1$
         \EndWhile
	\end{algorithmic} 
\end{algorithm}

Finally, we remark that it is also possible to use the rank adaptive version of BUG method \cite{ceruti2022rank} as a preconditioner. The essential design of the schemes are very similar. We have tried this in numerical tests, and did not find significant advantages. Thus the results are not reported in this paper.

\section{Numerical results}
\label{sec:numerical}
In this section, we test the proposed numerical methods by considering the diffusion equation with variable coefficients 
\begin{eqnarray}  
\label{eq:mastereq}
 \frac{\partial X}{\partial t}  
&=& b_1(y)\frac{\partial}{\partial x}\left(a_1(x)\frac{\partial X}{\partial x}\right) + b_2(y)  \frac{\partial^2 (a_2(x)X)}{\partial x \partial y} \notag \\
&&+ a_3(x)  \frac{\partial^2 (b_3(y)X)}{\partial x \partial y} + a_4(x)\frac{\partial}{\partial y}\left(b_4(y)\frac{\partial X}{\partial y}\right) + G(x, y, t)  
\label{2d variable}
\end{eqnarray}
on the  square domain $[-1,1]^2$ subject to zero Dirichlet boundary conditions and a given initial condition. We discretize the equation by standard finite difference in space and  implicit time integrators with various order. The spatially discretized scheme of \eqref{eq:mastereq} is then an example of \eqref{eqn:mode}.
Let the time grids be
$$
t_k = k\Delta t, \qquad k=0,1,\ldots, n_t,
$$
 for some integer $n_t>0$,
 and the space grids be
 $$
(x_i,y_j) = (-1+i\Delta x, -1+j \Delta y), \qquad  i=1,\ldots, n_x,\quad j=1,\ldots, n_y,
 $$ 
  for some integers $n_x>0$ and $n_y>0$. The unknown is represented by a time dependent matrix $X^n_{i,j} \approx X(x_i,y_j,t_n).$
\subsection{Parameter study for the  BUG preconditioner }
\label{sec:parameterstudy}
In this subsection, we  verify the optimal parameter for the  BUG preconditioner    by numerical experiments.  
We consider \eqref{eq:mastereq} with
\begin{align} \label{coefficient ex__parameter_restart}
a_1(x) &= 1+0.1\sin(\pi x),  &b_1(y) &= 1+0.1\cos(\pi y), \nonumber \\
a_2(x) &= 0.15+0.1 \sin(\pi x), &b_2(y) &= 0.15+0.1 \cos(\pi y),\nonumber \\
a_3(x) &= 0.15+0.1 \cos(\pi x), & b_3(y) &= 0.15+0.1 \sin(\pi y),\nonumber \\
a_4(x)  &= 1+0.1\sin(\pi x), & b_4(y) &= 1 + 0.1 \cos(\pi y),
\end{align}
and the rank 1 manufactured solution given by  
\begin{eqnarray}\label{mms ex__parameter_restart}
X(x,y,t) = 0.1\exp(-x^2/0.15^2)\exp(-y^2/0.15^2)\exp(-t),
\end{eqnarray} 
which will implicit define the forcing term $G(x,y,t).$ The final time is $t_{\rm end} = 0.1 \pi.$
We carry out the computation on four different grids with $n_x = n_y \in \{ 2^5-1, 2^6-1, 2^7-1, 2^{8}-1\}$. The timestep is chosen according to $n_t = \floor{t_{\rm end}/h_x}$ and $ \Delta t = t_{\rm end}/n_t$ with final time $t_{\rm end} = 0.1 \pi$.   
We test with the second order spatial discretization and implicit midpoint method in time with time step choice $\Delta t=O(h).$

The parameters we test include the following:  restart parameter, rounding tolerance $\epsilon$ and stopping criteria.
For the restart parameter,
to control the rank growth, we implement the restarted-lrGMRES where the lrGMRES restart every three iterations  with a maximal total number of iterations as $90$.    We will illustrate by an example    that frequent restarting will actually help bring the iteration numbers for the  preconditioner down. This seems to be a   feature of this nonlinear preconditioner. 

As for the parameter $\epsilon,$ according to the discussions in Section \ref{sec:bugp}, we use  $\epsilon=h^2$  in the lrGMRES and $\epsilon_2=O(h^2)$.
As for the stopping criteria, we follow \cite{coulaud2022robust} and use 
$$
\eta_{\mathcal{A},b}(x_k) = \frac{\|\mathcal{A}x_k - b\|}{\|\mathcal{A}\|_2\|x_k\|+\|b\|} \le \delta,
$$
where $\|\mathcal{A}\|_2$ is estimated as follows. Choose a set $\mathrm{W}$ of normalized matrices $w$ generated randomly from a
normal {and uniform} distribution, we estimate
$$
\|\mathcal{A}\|_2 = \max_{w \in \mathrm{W}} \|\mathcal{A} w\|.
$$
In our algorithm, we choose {$10$  normally distributed matrices and $10$  uniformly distributed matrices}. Such a stopping criteria may ensure a backward stable solution in contrast to using $\eta_{b}(x_k)=\|\mathcal{A}x_k - b\|/\|b\|$ as was pointed out by \cite{coulaud2022robust}.  The stopping parameter of lrGMRES is chosen as $\delta = \epsilon$ which was recommended in \cite{coulaud2022robust}.

The discussions above are the baseline for the ``optimal" parameter choice.
Below, we provide the numerical results based on this choice first. Then we will vary each individual parameter and observe the difference in the numerical performance.

\begin{example}[Optimal parameter] \label{Optimal}
\textnormal{
We   solve \eqref{eq:mastereq} with variable coefficients \eqref{coefficient ex__parameter_restart}  and manufactured solution \eqref{mms ex__parameter_restart}. 
We now provide the numerical results using the optimal parameters mentioned above, namely restart parameter being three, stopping criteria $\eta_{\mathcal{A},b}(x_k) \le \delta, \delta=\epsilon=h^3.$  The second order convergence can be clearly observed in Table \ref{table ex__parameter_restart}. Figure \ref{figure ex__parameter_restart3} displays, at each time step, the solution error, $\eta_{\mathcal{A},b}$, the solution rank, the maximal Krylov rank, and the iteration number.  It is evident that the  preconditioner is effective. The number of iteration is mostly 1 except the initial time step.  The numerical solution is of low rank, and more importantly, the maximal Krylov rank remains low.
}

\textnormal{
We compare lrGMRES with BUG preconditioner and lrGMRES without any preconditioner in Figure \ref{figure ex__parameter_restart3_noprecon}, where we display the history of solution error, iteration number, and maximal Krylov rank. It is evident that the preconditioner is the key to feasible computational cost.
}
\begin{table}[htbp]  
 \begin{center} 
\begin{tabular}{| c | c | c |c|c | c | c | } 
 \hline
 $h$ & error  & order  \\
 \hline
 \hline
   3.12(-2) &  1.06(-4) &  --  \\
 1.56(-2)  &  2.71(-5) &   1.97  \\
  7.81(-3) &  6.78(-6) &  2.00 \\
  3.90(-3) &  1.77(-6) &  1.93 \\
 \hline
 \end{tabular}
 \caption{Example \ref{Optimal}. Solving diffusion equation with variable coefficients \eqref{coefficient ex__parameter_restart}  and manufactured solution \eqref{mms ex__parameter_restart} with BUG preconditioner.  Restart every $3$ iterations. Stopping criteria $\eta_{\mathcal{A},b}(x_k) \le \delta.$ $\delta=\epsilon=h^3.$  For $h = h_x = h_y  \in \{3.12(-2), 1.56(-2), 7.81(-3),  3.90(-3)\}$, this table displays the solution error at the final time and order of convergence. }
 \label{table ex__parameter_restart}
 \end{center} 
 \end{table}

\begin{figure}[htbp] 
\centering
\includegraphics[width=0.32\linewidth]{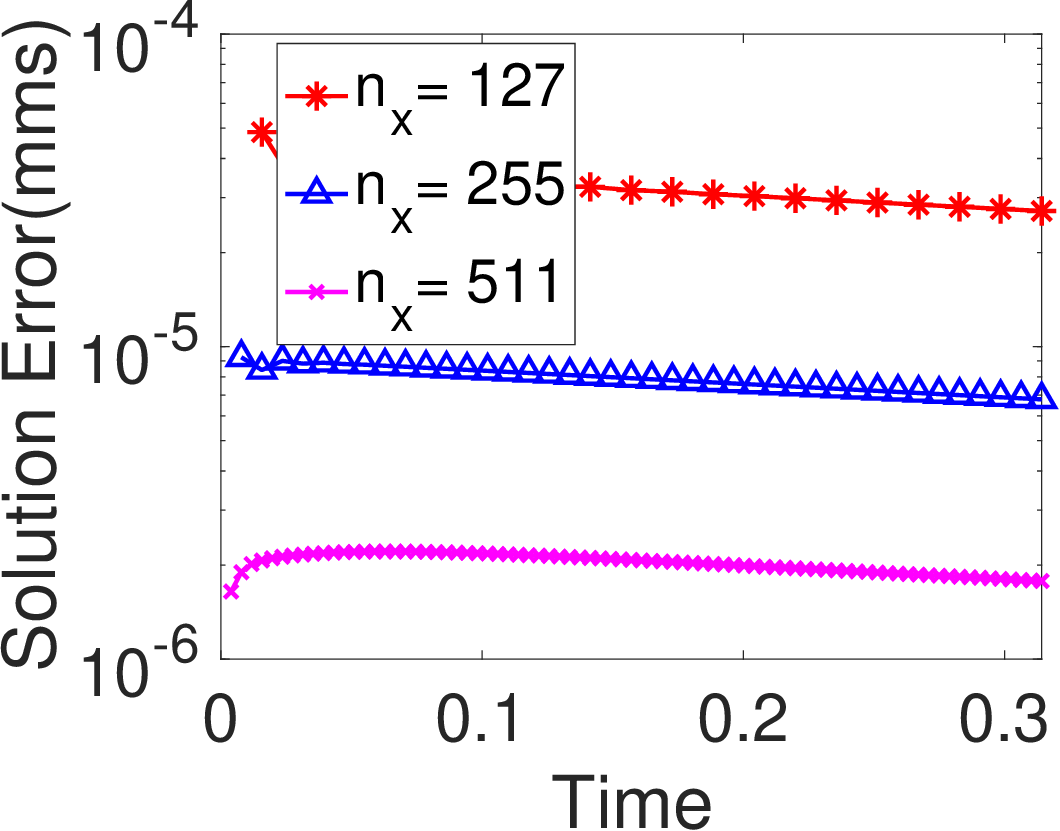} 
\includegraphics[width=0.32\linewidth]{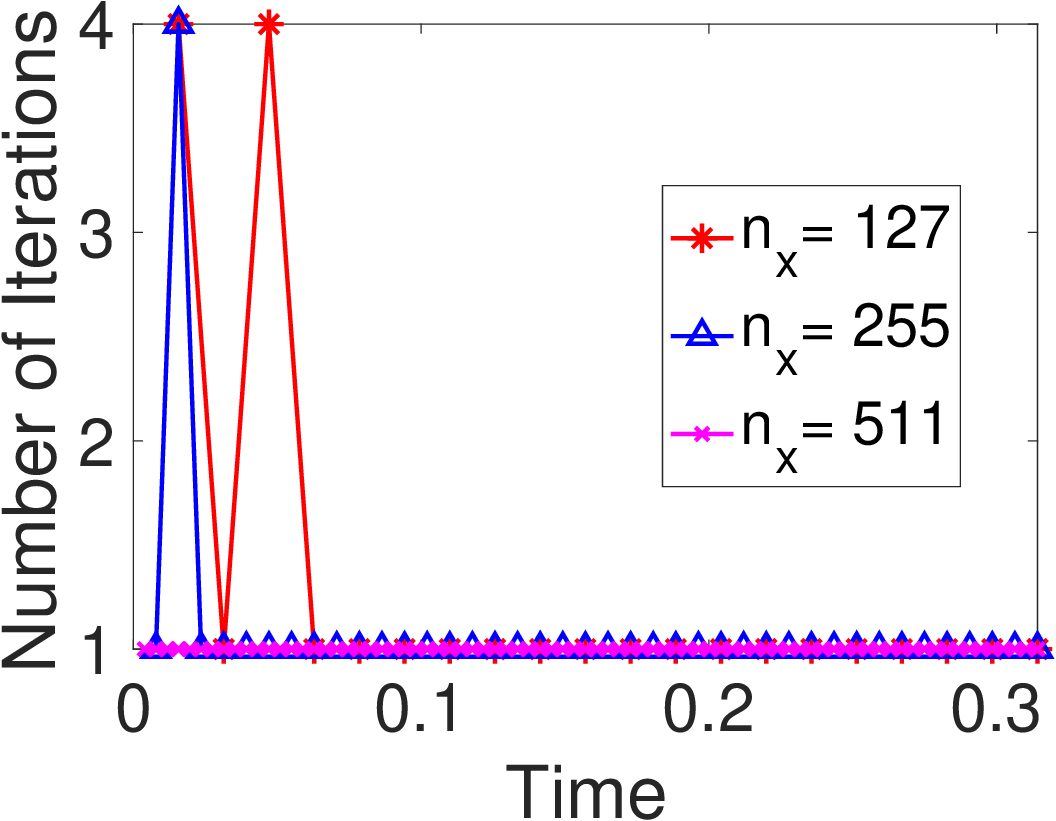} 
\\
\includegraphics[width=0.32\linewidth]{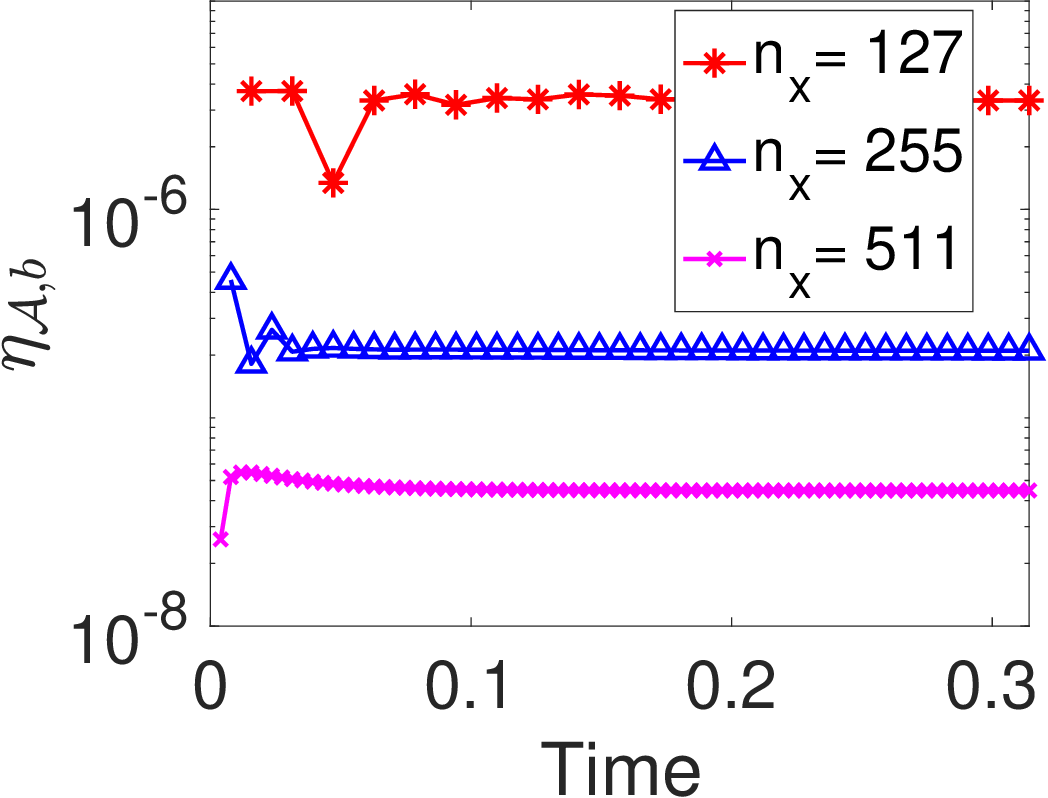} 
\includegraphics[width=0.32\linewidth]{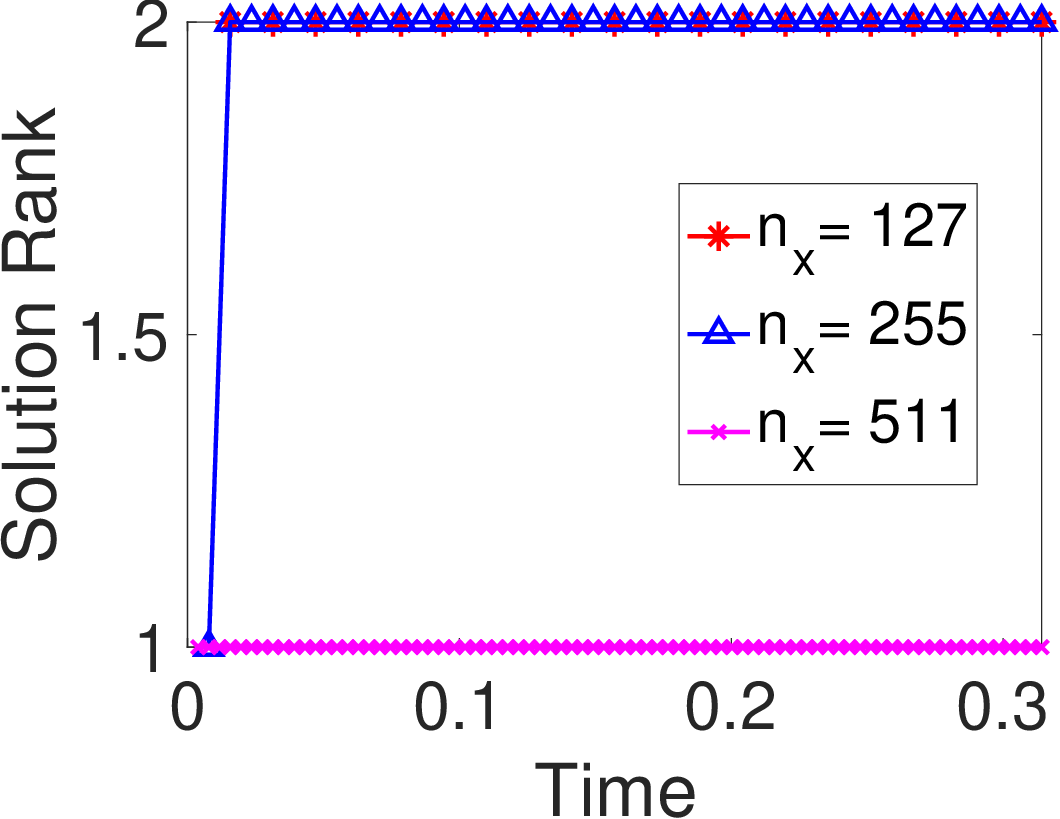} 
\includegraphics[width=0.32\linewidth]{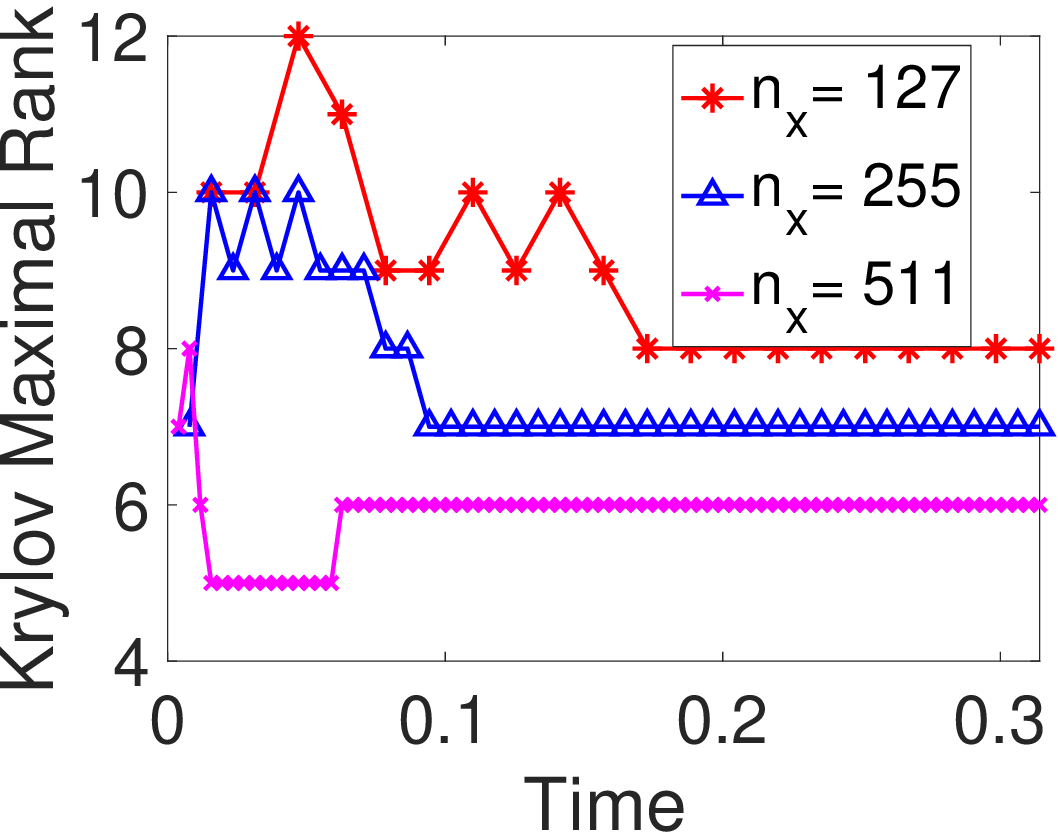} 
\caption{Example \ref{Optimal}. Solving diffusion equation with variable coefficients \eqref{coefficient ex__parameter_restart}  and manufactured solution \eqref{mms ex__parameter_restart} with BUG preconditioner. Restart every $3$ iterations. Stopping criteria $\eta_{\mathcal{A},b}(x_k) \le \delta.$ $\delta=\epsilon=h^3.$ For $h=h_x = h_y  \in \{1.56(-2), 7.81(-3),  3.90(-3)\}$, this figure displays the history of solution error, iteration number, $\eta_{\mathcal{A},b}$, solution rank, and maximal Krylov rank. 
} 
\label{figure ex__parameter_restart3}
\end{figure}

\begin{figure}[htbp] 
\centering
\includegraphics[width=0.32\linewidth]{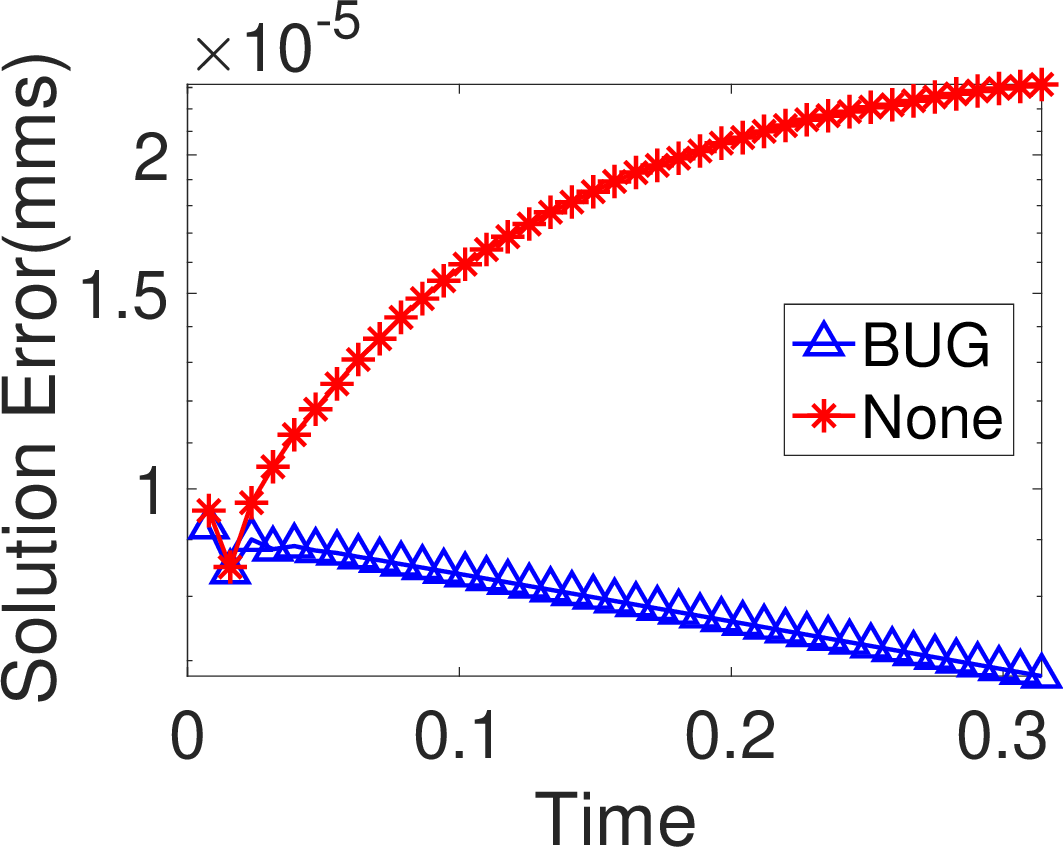} 
\includegraphics[width=0.32\linewidth]{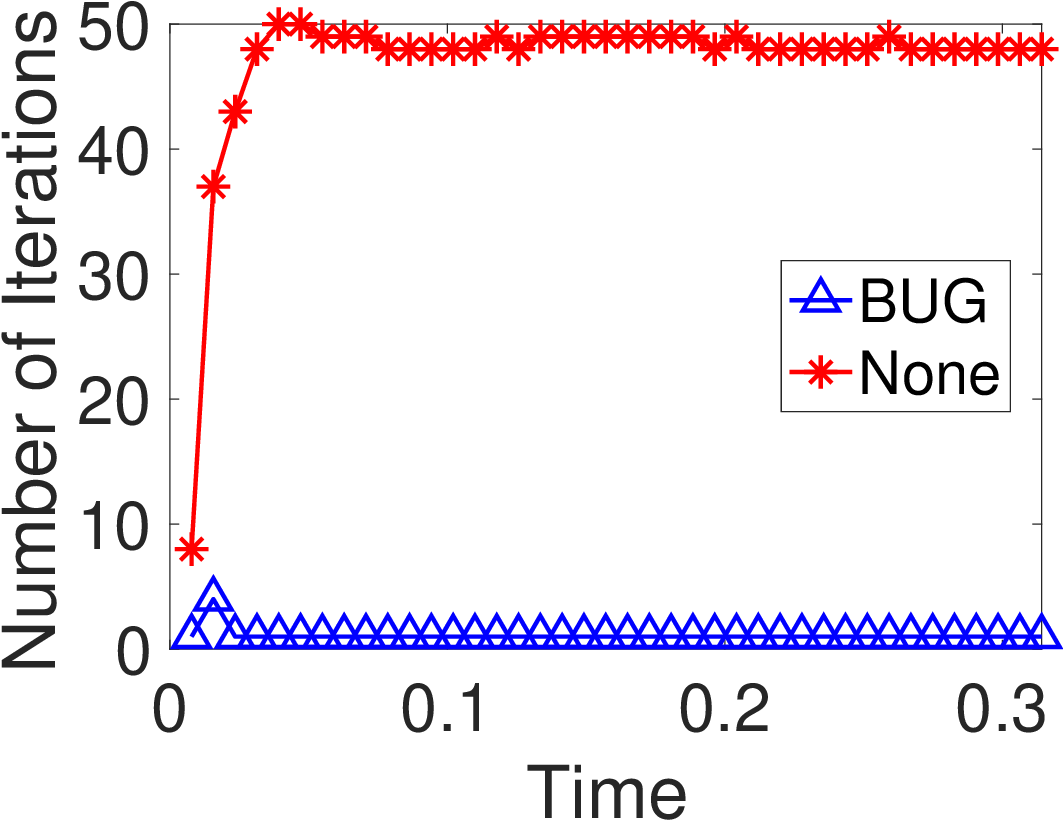} 
\includegraphics[width=0.32\linewidth]{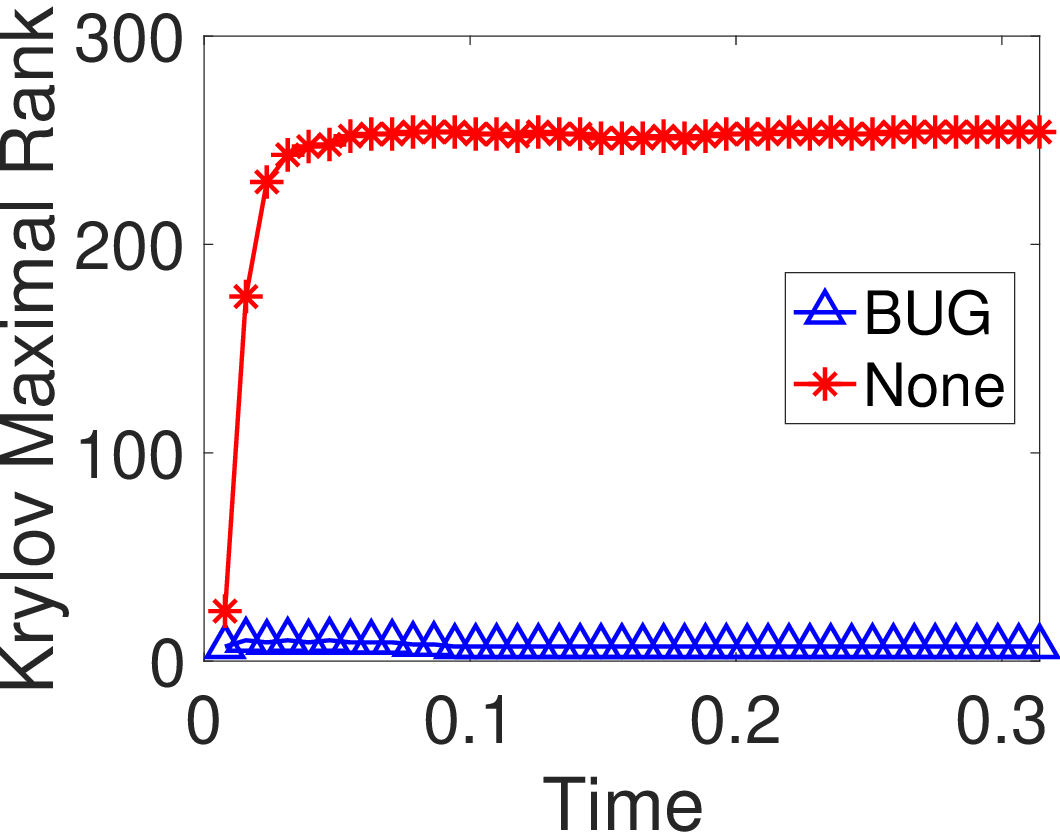} 
\caption{Example \ref{Optimal}. Solving diffusion equation with variable coefficients \eqref{coefficient ex__parameter_restart}  and manufactured solution \eqref{mms ex__parameter_restart} with BUG preconditioner (which restarts every $3$ iterations) and no preconditioner (which restarts every $3$ iterations ``None-3'' and every $25$ iterations ``None-25''). Stopping criteria $\eta_{\mathcal{A},b}(x_k) \le \delta.$ $\delta=\epsilon=h^3.$ For $h=h_x = h_y  \in \{ 7.81(-3)\}$, this figure displays the history of solution error, iteration number, and maximal Krylov rank. 
} 
\label{figure ex__parameter_restart3_noprecon}
\end{figure}

 \end{example}

\begin{example}[Varying the restart parameter] \label{Example restart}
\textnormal{
In this example, we enlarge the restart parameter to 25 while keeping all other parameters unchanged from Example \ref{Optimal}. We solve \eqref{eq:mastereq} with variable coefficients \eqref{coefficient ex__parameter_restart}  and manufactured solution \eqref{mms ex__parameter_restart}. Figure \ref{figure ex__parameter_restart25} displays the corresponding result. %
The iteration numbers increase dramatically for the first time step, and the maximal Krylov ranks also become larger.
This shows that BUG preconditioning benefits from frequent restarting.}

\begin{figure}[htbp] 
\centering
\includegraphics[width=0.32\linewidth]{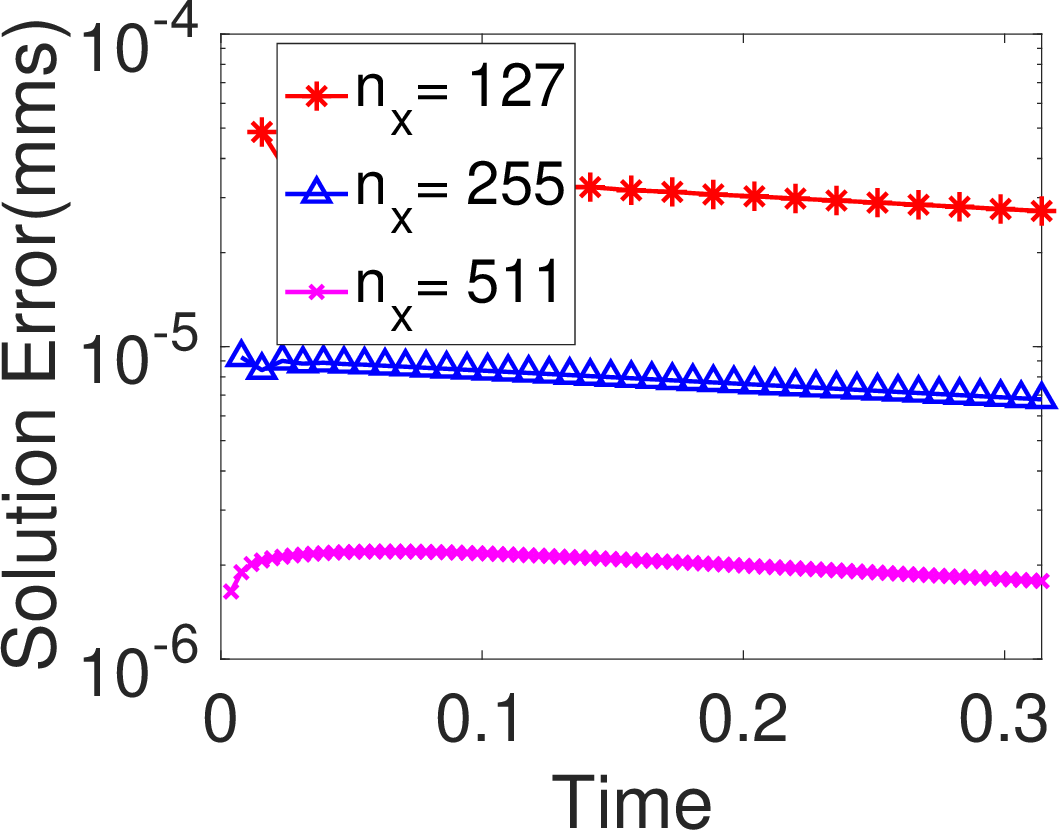} 
\includegraphics[width=0.32\linewidth]{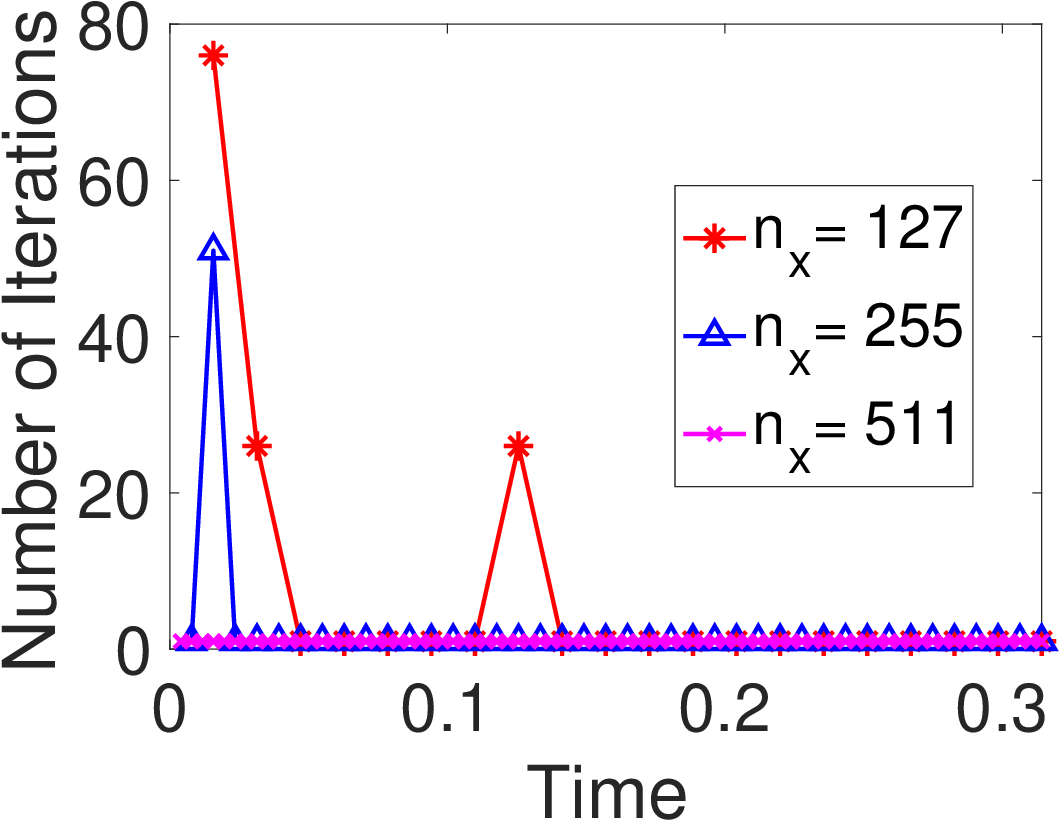} 
\\
\includegraphics[width=0.32\linewidth]{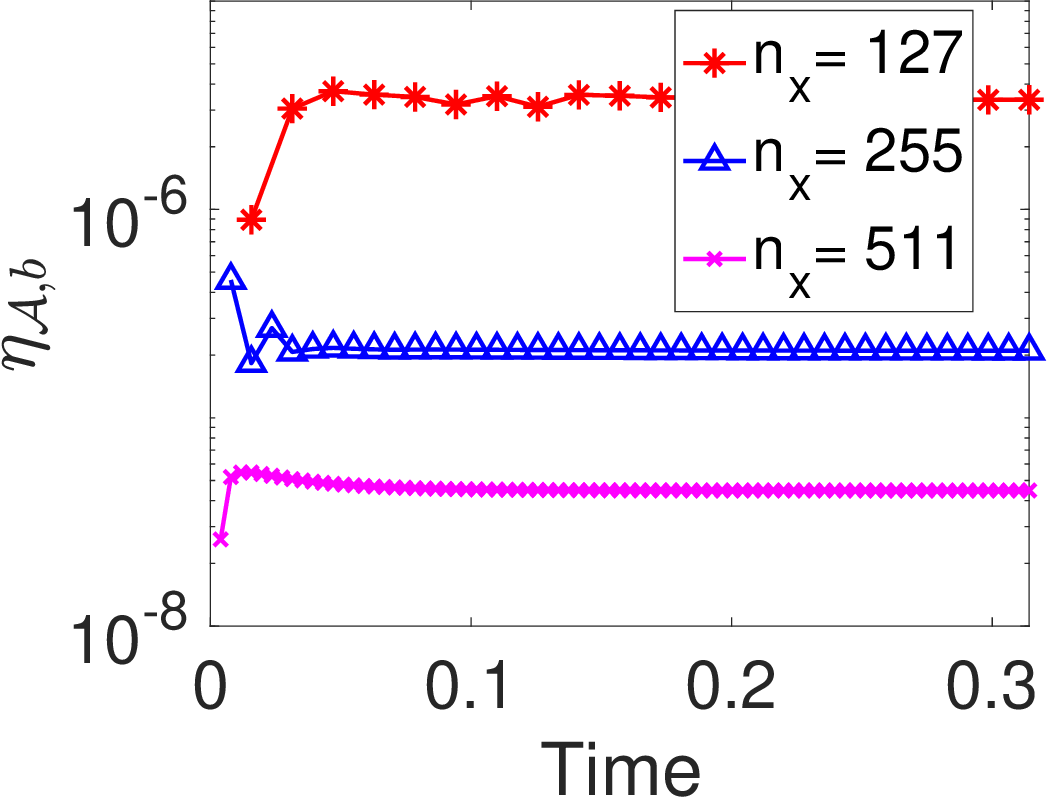} 
\includegraphics[width=0.32\linewidth]{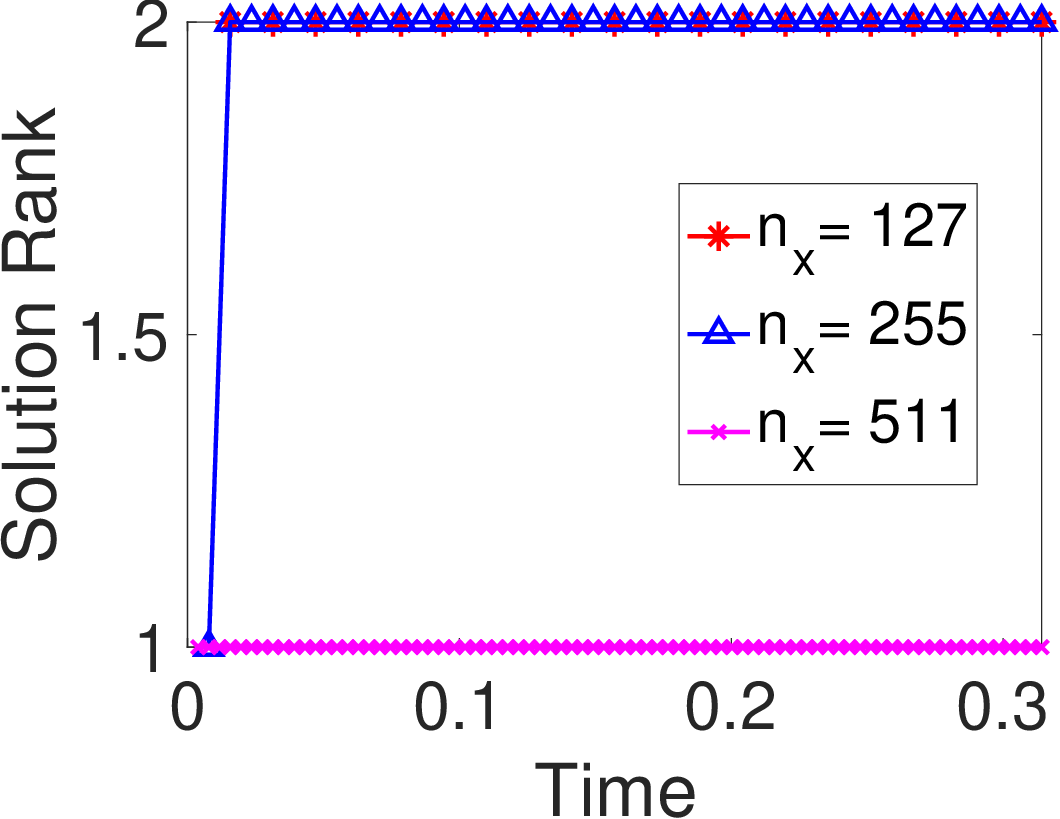} 
\includegraphics[width=0.32\linewidth]{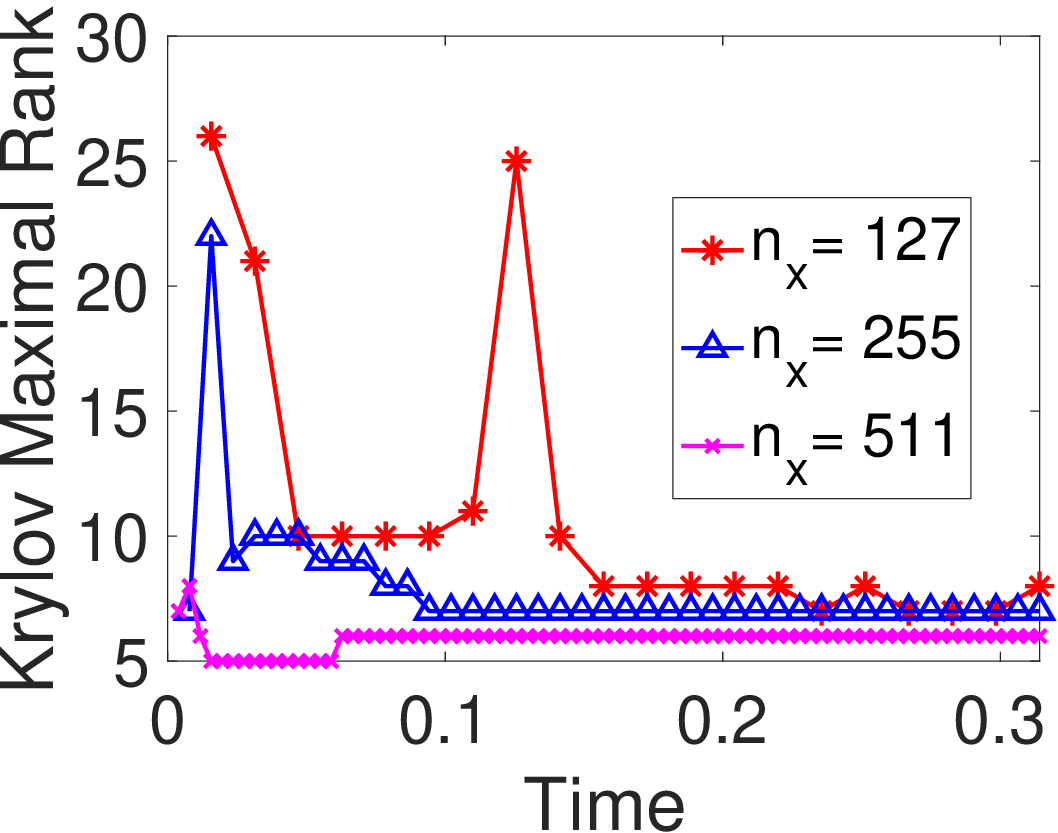} 
\caption{Example \ref{Example restart}. Solving diffusion equation with variable coefficients \eqref{coefficient ex__parameter_restart}  and manufactured solution \eqref{mms ex__parameter_restart}. Restart every $25$ iterations. Stopping criteria $\eta_{\mathcal{A},b}(x_k) \le \delta.$ $\delta=\epsilon=h^3.$ For $h=h_x = h_y  \in \{ 1.56(-2), 7.81(-3),  3.90(-3)\}$, this figure displays the history of solution error, iteration number, $\eta_{\mathcal{A},b}$, solution rank, and maximal Krylov rank. 
} 
\label{figure ex__parameter_restart25}
\end{figure}

\end{example}

\begin{example}[Varying the stopping criteria]\label{exa:stop}
\textnormal{
  In this example, we  switch the stopping criteria to $\eta_b$ and keep all other parameters the same as in Example \ref{Optimal}. The results for \eqref{eq:mastereq} with variable coefficients \eqref{coefficient ex__parameter_restart}  and manufactured solution \eqref{mms ex__parameter_restart} are reported in Figure \ref{figure ex__parameter_residual} when the mesh size $h=7.81(-3)$. In this case, the iteration numbers are drastically larger, i.e., the stopping criteria $\eta_b\le\delta$ does not guarantee that $\eta_b$ stagnates below the stopping parameter before the maximal number of iteration is reached. 
  Similar behavior has been reported in    \cite{coulaud2022robust}.}

\begin{figure}[htbp] 
\centering
\includegraphics[width=0.32\linewidth]{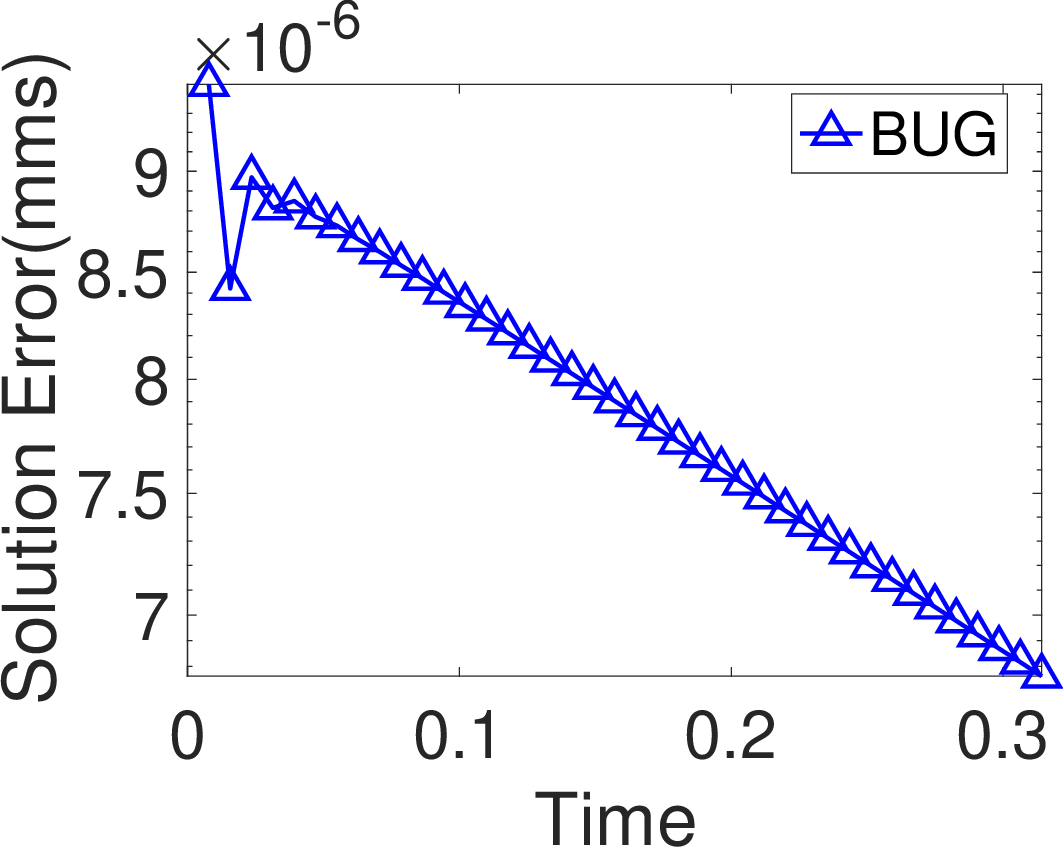} \includegraphics[width=0.32\linewidth]{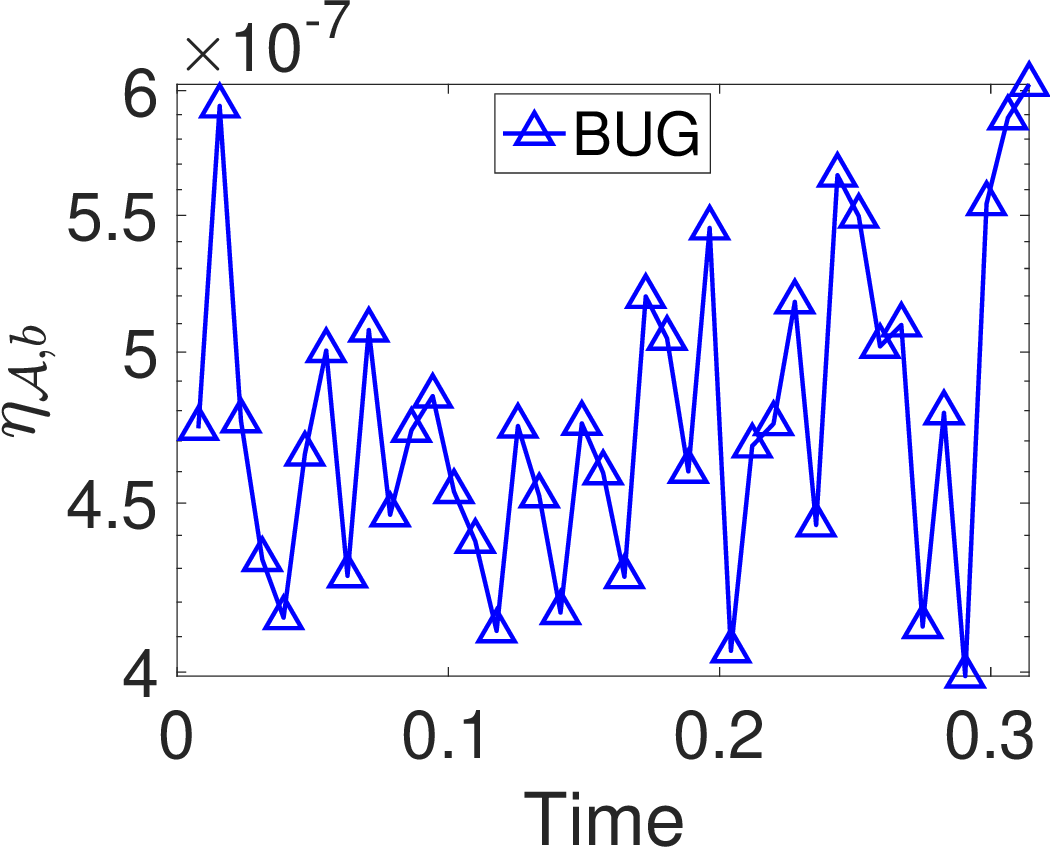} \includegraphics[width=0.32\linewidth]{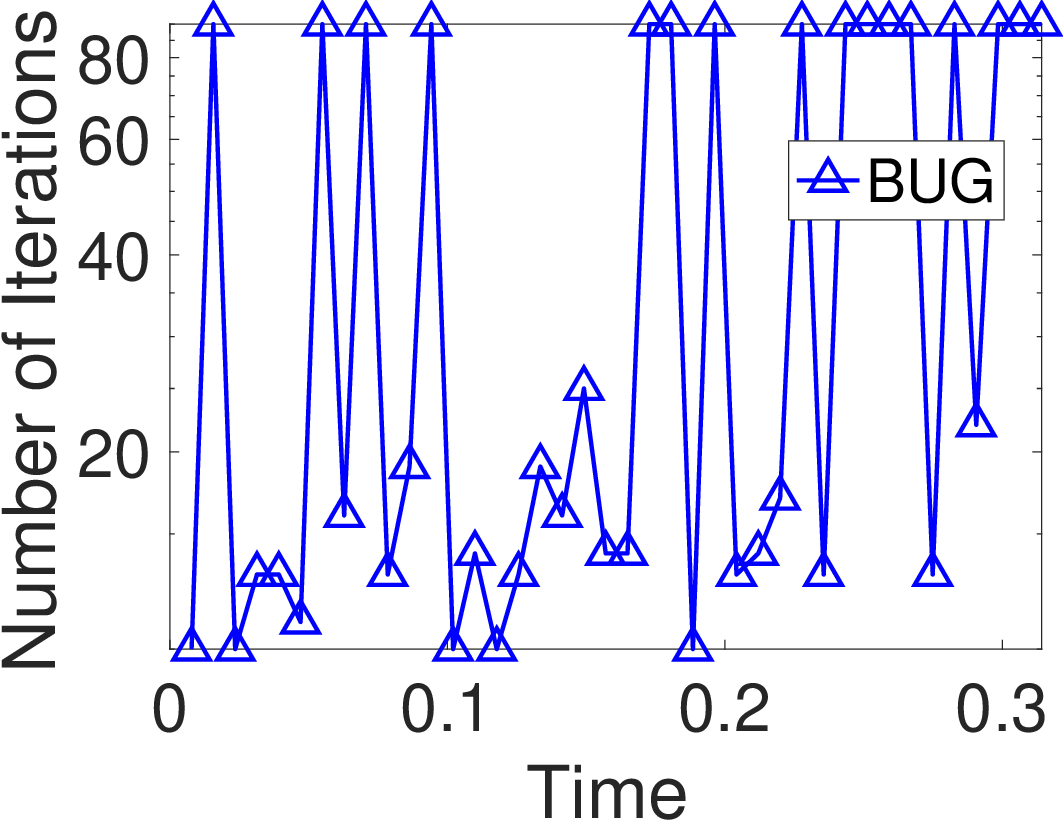} \\
\caption{Example \ref{exa:stop}. Solving diffusion equation with variable coefficients \eqref{coefficient ex__parameter_restart} and manufactured solution \eqref{mms ex__parameter_restart}. Restart every $3$ iterations using  BUG preconditioner. Stopping criteria $\eta_b \le \delta$.  $\delta=\epsilon=h^3.$
For $h= h_x = h_y  =7.81(-3)$, we report the history of solution error, $\eta_{\mathcal{A},b}$, and  iteration number. 
} 
\label{figure ex__parameter_residual}
\end{figure}

\end{example}

\begin{example}[Varying the rounding tolerance] \label{exa:round} 
\textnormal{
 In this example, we solve  \eqref{eq:mastereq} with variable coefficients \eqref{coefficient ex__parameter_restart}  and   manufactured solutions given by \eqref{mms ex__parameter_restart}.
To test whether second-order convergence may be observed by a larger rounding tolerance, we chose  a larger rounding tolerance $\epsilon = h^2$, while keeping all other parameters same as Example \ref{Optimal}. The   convergence rate is displayed in Table \ref{table hx2<hx3} for both ES and BUG preconditioner. We observe clearly deteriorated convergence rate (of first order) for ES preconditioners. However, the BUG preconditioner  retains its second order convergence rate. The BUG preconditioner has more forgiveness for larger tolerance and this phenomenon is reported in Tables \ref{table highcontrast_variable} and \ref{table ex_BDF_const convergence} in later parts of the paper.}
\begin{table}[htbp]  
 \begin{center} 
\begin{tabular}{| c | c | c |c|c | c|c|c|c| c | c | } 
 \hline
 $h$   & error (ES) & order (ES) &   error (BUG) & order (BUG)   \\
 \hline
 \hline
 3.12(-2) & 1.18(-4) &  --  &  1.11(-4) &  --  \\
 1.56(-2)  & 5.16(-5) &  1.19   &  2.84(-5) &  1.96   \\
 7.81(-3) & 3.09(-5) &  0.74  &  7.11(-6) &  2.00  \\
 3.90(-3) & 1.65(-5) &  0.91  &  1.77(-6) &  2.00  \\
 \hline
 \end{tabular}
  \caption{Example \ref{exa:round}. Solving diffusion equation with variable coefficients \eqref{coefficient ex__parameter_restart} with  manufactured solution \eqref{mms ex__parameter_restart}.  BUG and ES preconditioner. $\delta=\epsilon=h^2$. For $h=h_x = h_y  \in \{3.12(-2), 1.56(-2), 7.81(-3),  3.90(-3)\}$, this table displays the error at the final time and order of convergence. 
 \label{table hx2<hx3}}  
\end{center} 
 \end{table} 
\end{example}

In summary, from examples tested in this subsection, we confirm that the optimal parameter shall be taken as the following: stopping criteria $\eta_{\mathcal{A},b}(x_k) \le \delta.$ $\delta=\epsilon$ with  $\epsilon$ chosen according to local truncation error analysis and with frequent restarting (restart=3). In the remaining part of the paper, we will use this choice for the BUG preconditioner.

\subsection{Comparison of BUG, ES, and hybrid preconditioners}
In this subsection, we conduct detailed comparison of the performance of the ES, BUG and the hybrid preconditioner. Like the previous subsection, we consider second order in space and time schemes with mesh $\Delta t=O(h).$  The BUG preconditioner use the parameter choice indicated in previous subsection. The ES preconditioner use parameters as indicated in the Appendix with the same rounding tolerance as the BUG preconditioner. 

The advantage of the ES preconditioner is its computational efficiency based on a simple sum of Kronecker products. However, the accuracy of ES is based on a particular form of operator structure which may not hold in many cases. We note that the BUG preconditioner does have cost associated with solving K-, L- and the Galerkin steps. Ideally, for matrix equations of different size and structures \cite{simoncini2016computational}, there are various techniques to be used to optimize the computation. In this paper, we did not optimize those solvers, and from our numerical experiments, the result is still satisfactory.

\begin{example} \label{Example compare_es_le_}
\textnormal{
In this example we solve equation \eqref{2d variable} with the following  coefficients 
\begin{eqnarray} \label{coefficient ex_comparision_es<}
a_1(x)& =& 1 = b_1(y), \quad a_2(x) = 0.8 = b_3(y),\nonumber \\
b_2(y) &=& 1 = a_3(x), \quad a_4(x) = 1 = b_4(y),
\end{eqnarray}
 and manufactured solution
\begin{equation}
\label{mms ex_comparision_es<}
X(x,y,t) = \exp(-(x-0.1\sin(t))^2/0.12^2)\exp(-(y+0.1\cos(t))^2/0.12^2)\exp(-t).
\end{equation}
}
 
\textnormal{
 Figure \ref{figure:ex_comparision_es<} displays the simulation results. The three preconditioners perform similarly with regard to the solution error and rank. The  BUG and hybrid preconditioner give smaller values for $\eta_{\mathcal{A},b}$. The iteration numbers for the ES preconditioner are constant (=5) independent of the time step. For the BUG  preconditioner, the iteration numbers are 1 except the first time step. We notice that this observation is consistent with results reported in the previous subsection. The  solution has a rapid transition during the initial layer, and this causes a larger iteration number (={10}) initially. 
}

\textnormal{  
We further test the time stepping with  an initial condition where we set $t=0$ in \eqref{mms ex_comparision_es<} to get
\begin{equation}\label{ic compare}
   X(x,y,0) = \exp(-x^2/0.12^2)\exp(-(y+0.1)^2/0.12^2)
\end{equation}
and no forcing term, i.e. $G(x,y,t)=0$. 
We display in Figure \ref{figure:ex_ic_comparision} the history of $\eta_{\mathcal{A},b}$, the solution rank, the maximal Krylov rank, and  the iteration number. In the figure, we also display the solution rank when rounded by the same rounding constant and a full rank standard finite difference scheme in comparison. We can see that the exact solution has a rank that grows at the initial stage and the rank starts to decay. All three preconditioners  accurately track this trend in rank. The   hybrid preconditioner  yields a lower number of iteration and has the smallest maximal Krylov rank.
}
\begin{figure}[h] 
\centering
\includegraphics[width=0.32\linewidth]{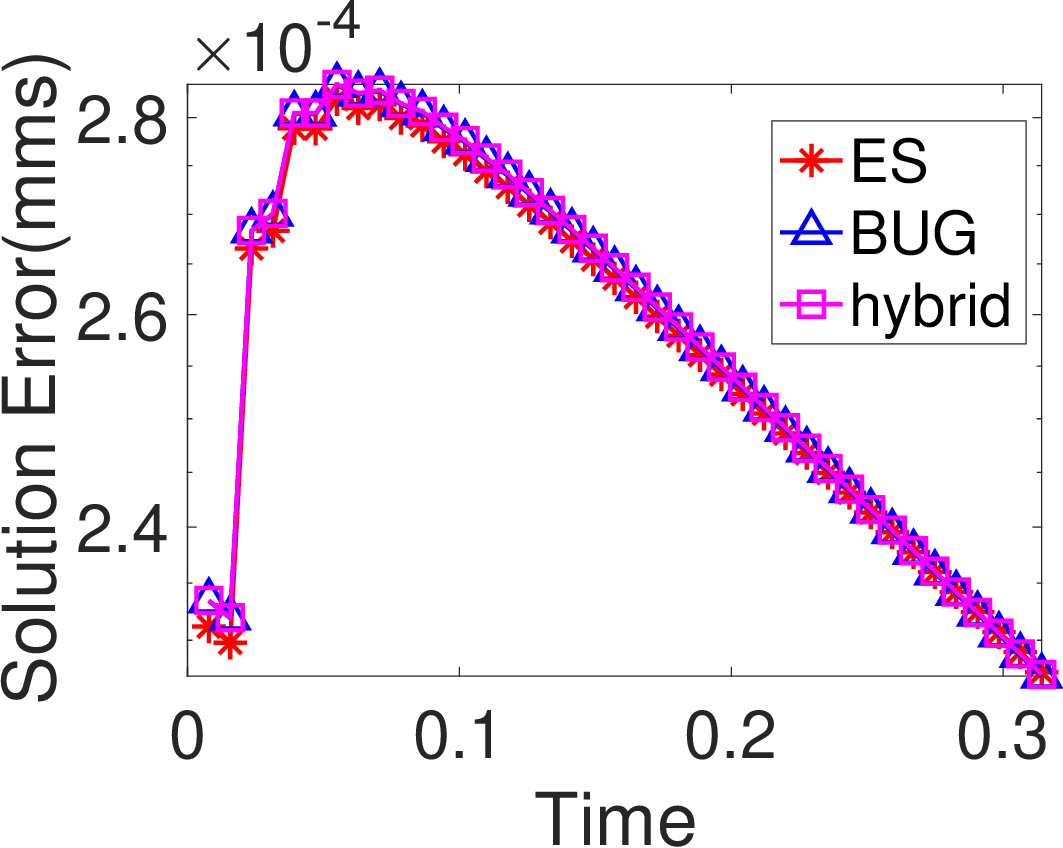} 
\hspace{0.1\linewidth}
\includegraphics[width=0.32\linewidth]{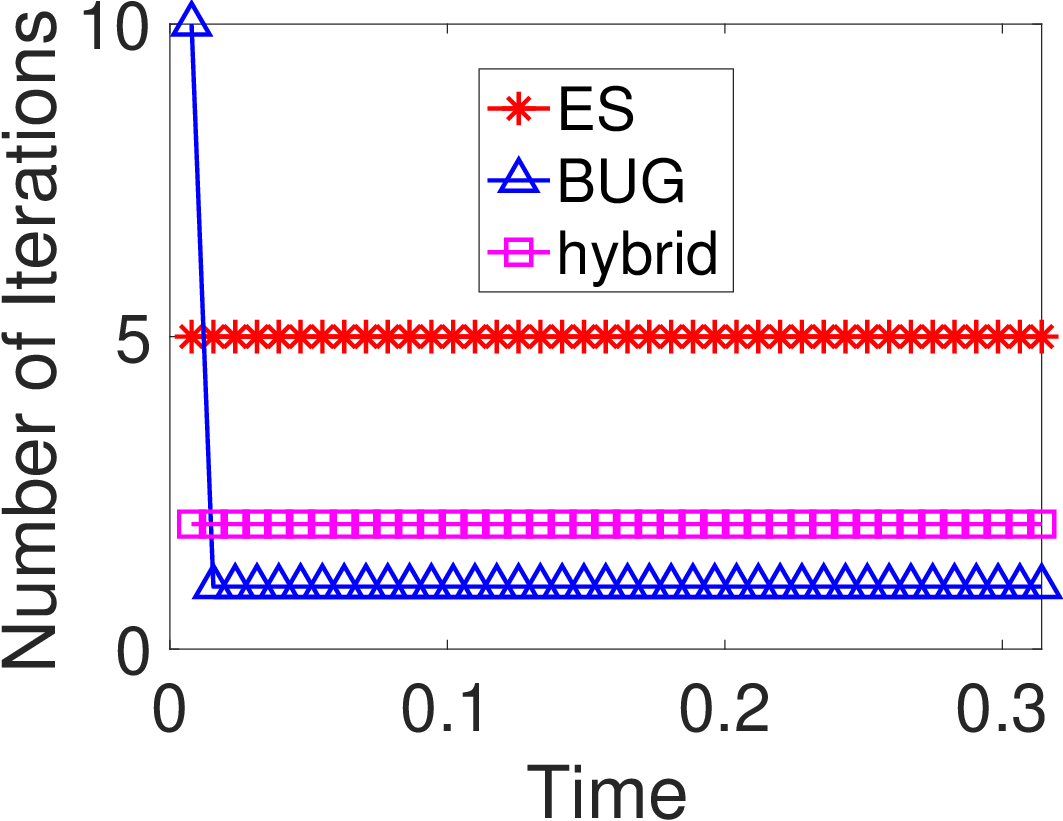} 
\\
\includegraphics[width=0.32\linewidth]{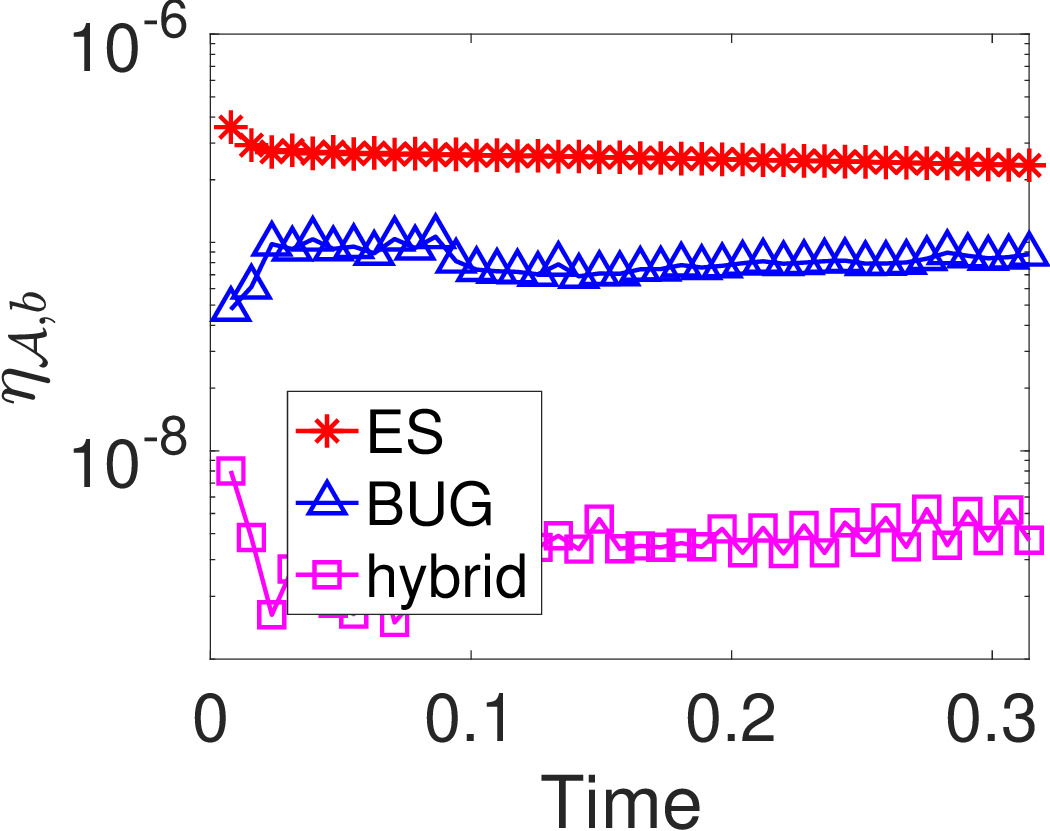} 
\includegraphics[width=0.32\linewidth]{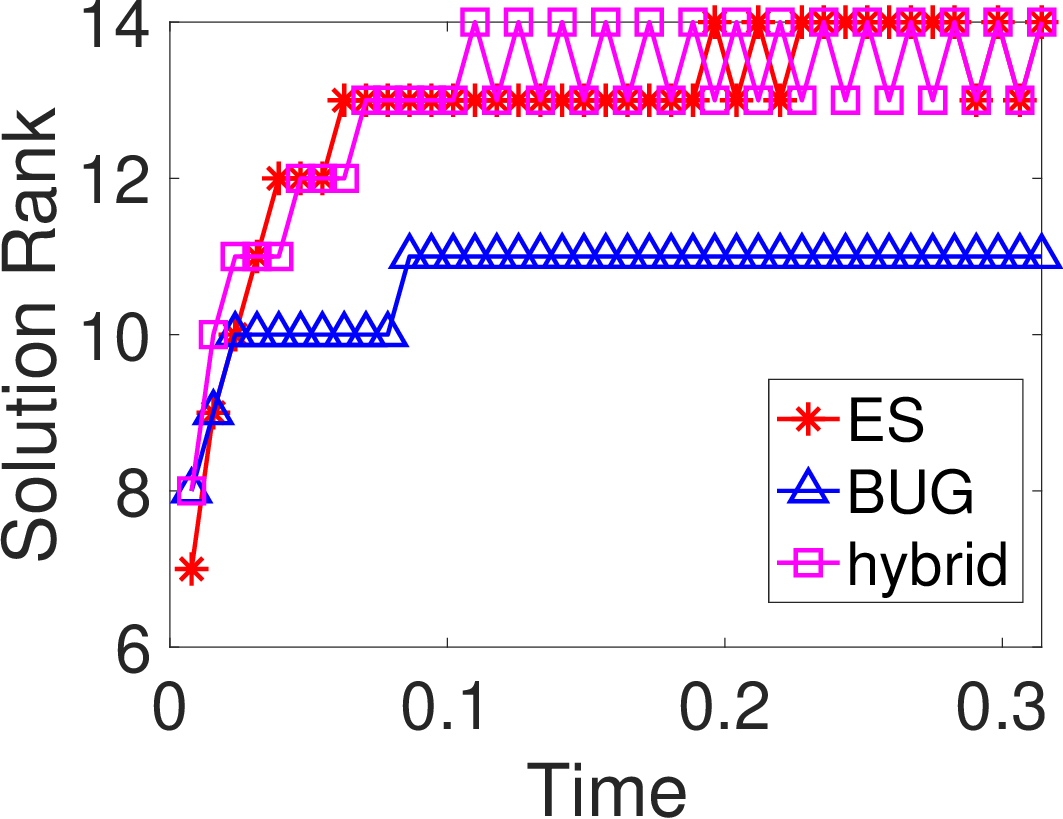} 
\includegraphics[width=0.32\linewidth]{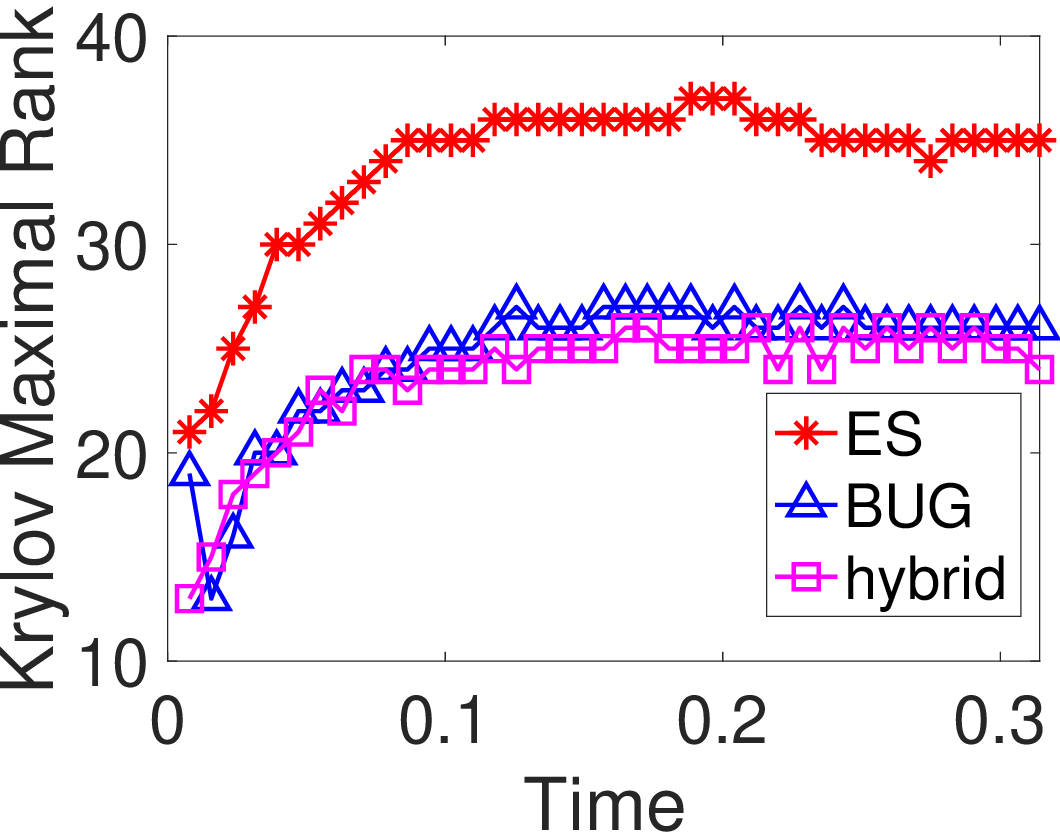} 
\caption{Example \ref{Example compare_es_le_}. Solving diffusion equation with constant coefficients \eqref{coefficient ex_comparision_es<} and manufactured solution \eqref{mms ex_comparision_es<}. Preconditioner: ES, BUG, and hybrid. Rounding tolerance $\epsilon = h^3$. For $h=h_x = h_y  =7.81(-3)$, this figure displays the history of solution error, iteration number, $\eta_{\mathcal{A},b}$, solution rank, and maximal Krylov rank. 
\label{figure:ex_comparision_es<}} 
\end{figure}

\begin{figure}[h] 
\centering
\includegraphics[width=0.32\linewidth]{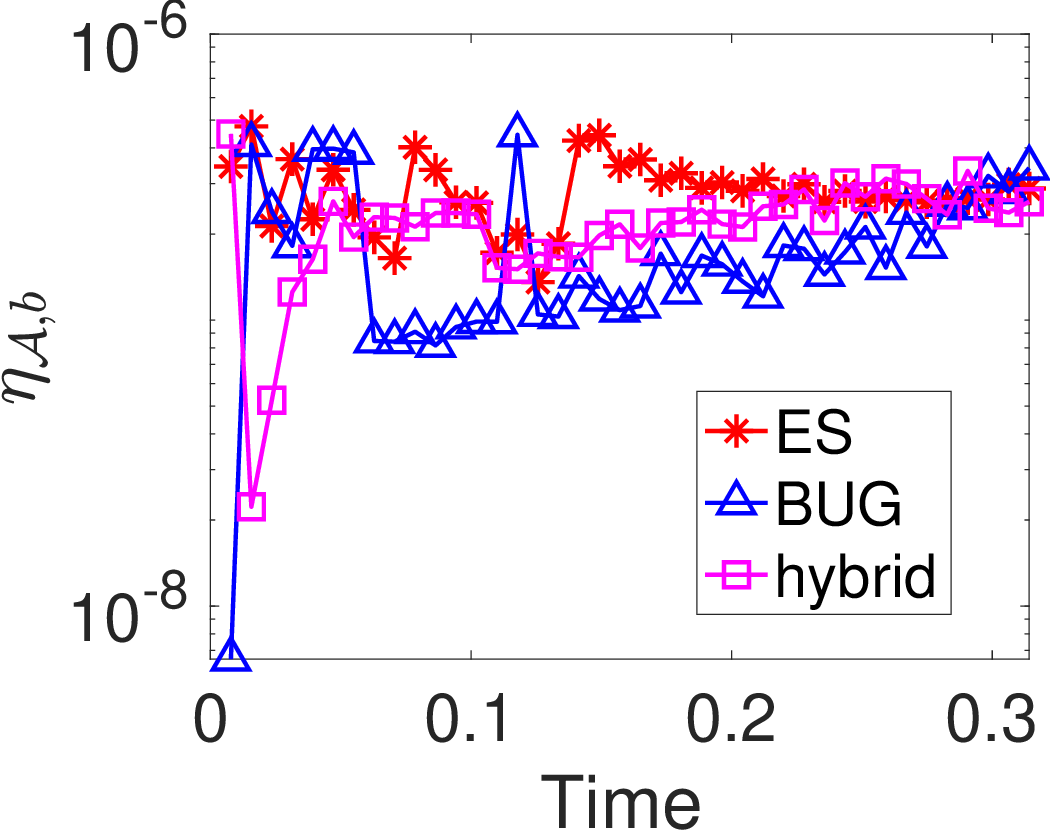} 
\hspace{0.1\linewidth}
\includegraphics[width=0.32\linewidth]{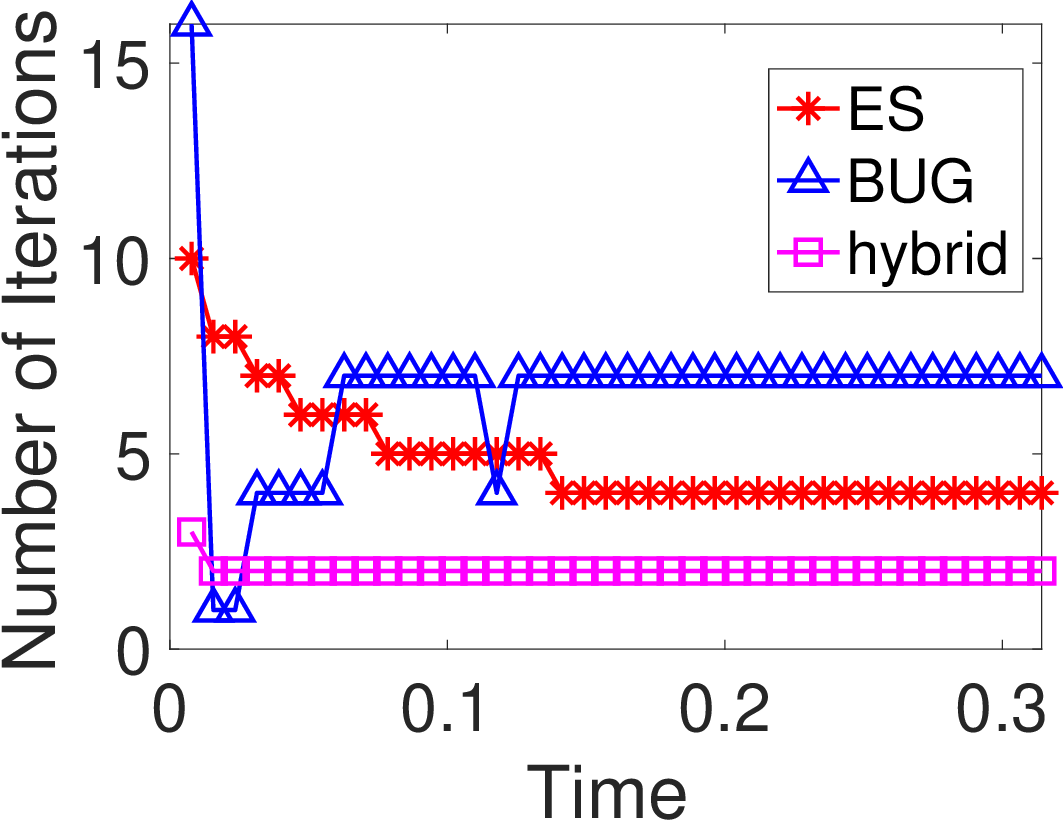} 
\\
\includegraphics[width=0.32\linewidth]{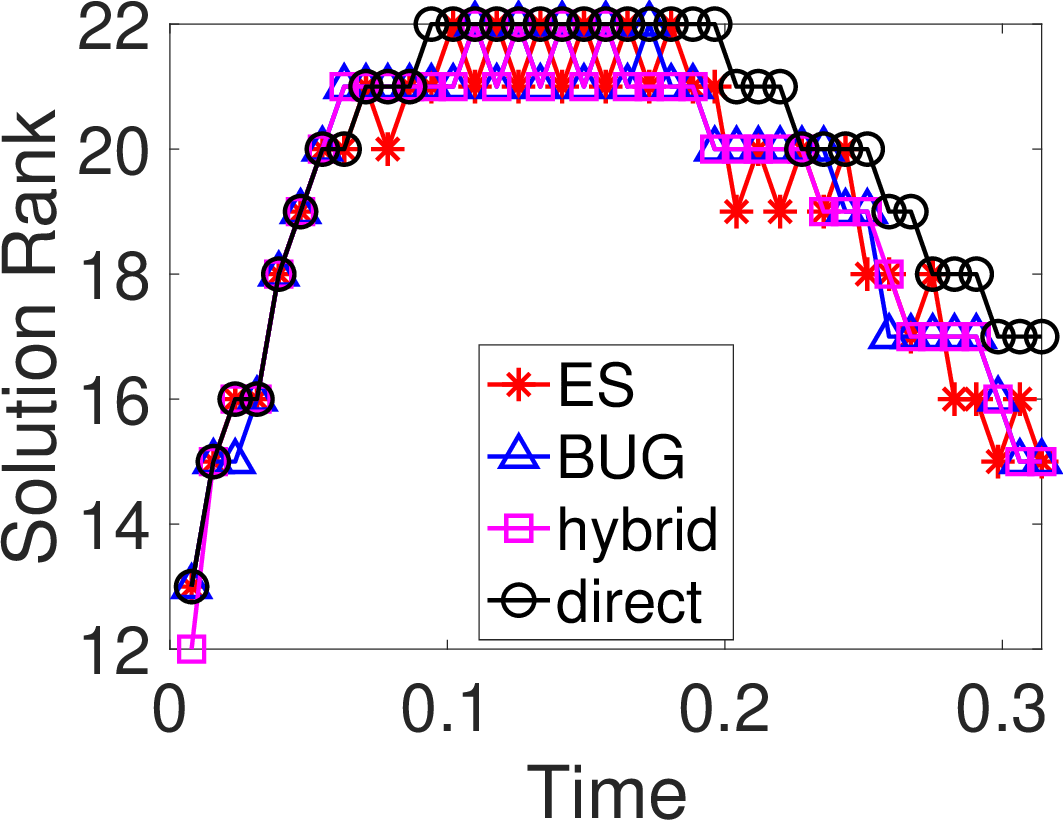} 
\hspace{0.1\linewidth}
\includegraphics[width=0.32\linewidth]{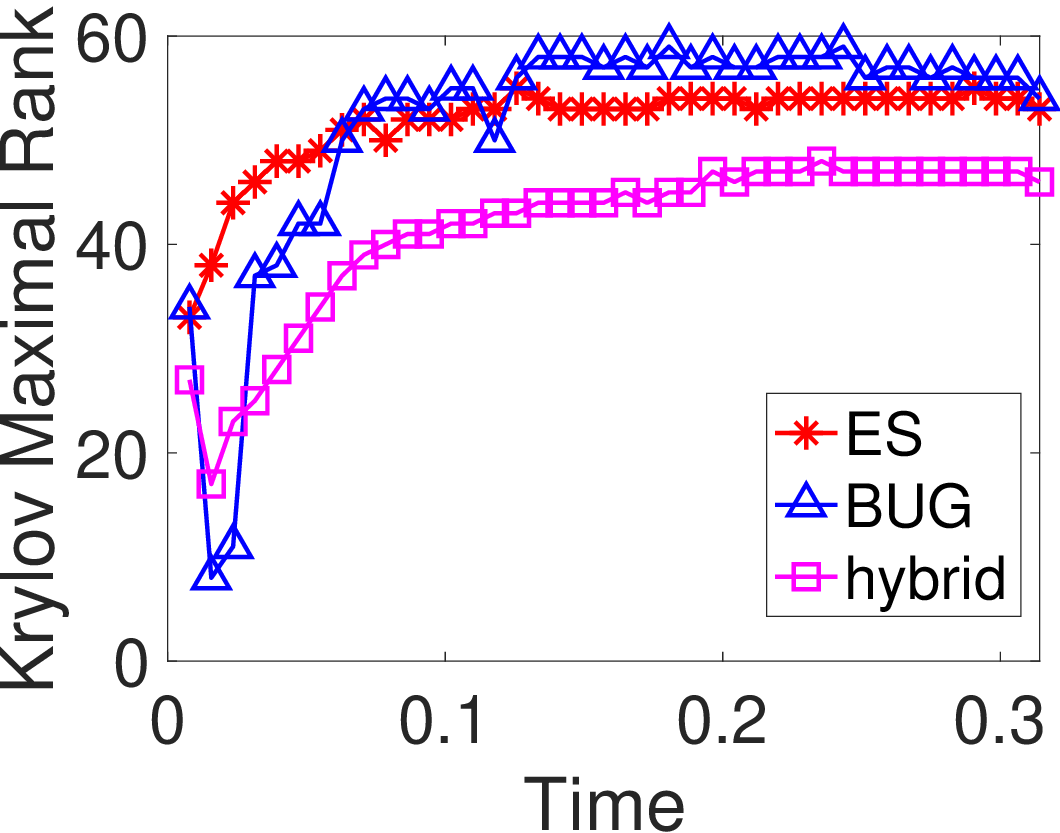} 
\caption{Example \ref{Example compare_es_le_}. Solving diffusion equation with  coefficients \eqref{coefficient ex_comparision_es<} and initial condition \eqref{ic compare}. Preconditioner: ES, mBUG, and hybrid. Rounding tolerance $\epsilon = h^3$. For $h=h_x = h_y  =7.81(-3)$, this figure displays the history of iteration number, $\eta_{\mathcal{A},b}$, solution rank, and maximal Krylov rank.}
\label{figure:ex_ic_comparision}
\end{figure}
\end{example}

\begin{example}[High contrast variable coefficient example]\label{ex:highc}
\textnormal{
We proceed with a challenging example, where we have high contrast variable coefficients give by
\begin{eqnarray} \label{coefficient highcontrast_variable}
a_1(x)& = 1, \quad &b_1(y) = 1+0.1\sin(\pi y), \nonumber\\
a_2(x)&  = 1, \quad & b_2(y) = 1/\eta(1 + 0.1 \sin(\pi y)), \nonumber \\
a_3(x)&  = 1, \quad & b_3(y) = 1/\eta(1 + 0.1 \sin(\pi y)), \nonumber \\
a_4(x)& = 1, \quad &b_4(y) = 1/\eta^2(1 + 0.1 \sin(\pi y)),
\end{eqnarray}
with $\eta =1/10$. We  test with the manufactured solution given by 
\begin{eqnarray}  
\label{coefficient high contrast tobler}
X(x,y,t) =  (1+\sin(\pi t/2))(1-x^2)(1-y^2)\exp(x)\exp(y).
\end{eqnarray}
}

 \begin{table}[]  
 \begin{center} 
\begin{tabular}{| c | c | c| c| c | c | c | c | c  |} 
 \hline
 & error  & order  & error  & order  & error  & order  & error  & order  
\\
 \hline
 $h$ & $\epsilon=h^3$(ES)  & --  &  $\epsilon=\eta h^3$(ES)  &  --  & $\epsilon=\eta^2 h^3$(ES) &  -- &  $ \epsilon=h^3$(BUG)  &  --\\
 \hline
 \hline
   3.12(-2)  &  1.42(-2) & -- &  1.17(-3) & -- & 9.67(-4) & -- &  9.13(-4) & -- \\
 1.56(-2)  &   4.77(-3) &  1.57 &  3.10(-4) &  1.92 &  2.35(-4) &  2.03 &  2.40(-4) &  1.92 \\
  7.81(-3) &  4.71(-3) &  0.017 &  9.71(-5) &  1.67 &  6.06(-5) &  1.95 &  6.03(-5) &  1.99 \\
 \hline
 \end{tabular}
 \caption{Example \ref{ex:highc}. Solving diffusion equation with high contrast variable coefficients \eqref{coefficient highcontrast_variable} and  manufactured solution \eqref{coefficient high contrast tobler}. Preconditioner: ES and BUG. For $h = h_x = h_y  \in \{3.12(-2), 1.56(-2), 7.81(-3)\}$, this table displays the solution error at the final time for different rounding tolerances  and order of convergence. 
 }  
 \label{table highcontrast_variable}
 \end{center} 
 \end{table} 
\textnormal{
In Table \ref{table highcontrast_variable}, we test the convergence using rounding tolerance $\epsilon \in\{ h^3, \eta h^3, \eta^2 h^3\}$ for the ES preconditioner and $\epsilon = h^3$ for BUG preconditioner. For the ES preconditioner, the second order convergence is achieved when $\epsilon = \eta^2 h^3$ according to the tolerance number chosen according to problem scaling. However, the ES preconditioner fails to converge   when $\epsilon = h^3.$  The BUG preconditioner, on the other hand,  gives the second order convergence rate with $\epsilon=h^3.$
This table is similar in spirit to Table \ref{table hx2<hx3}, which shows BUG preconditioner is more forgiving, and allows a larger   choice of $\epsilon.$ %
In Figure \ref{figure highcontrast_variable h3}, we compare ES and BUG  preconditioners with $\epsilon=\eta^2 h^3$ and $\epsilon=h^3$ respectively. In those cases, the two preconditioners give similar errors as indicated by Table \ref{table highcontrast_variable}. We can see the  BUG preconditioner outperfrom ES preconditioner in terms of number of iterations and maximal Krylov rank, and solution rank in this case.
}

\begin{figure}[h] 
\centering
\includegraphics[width=0.32\linewidth]{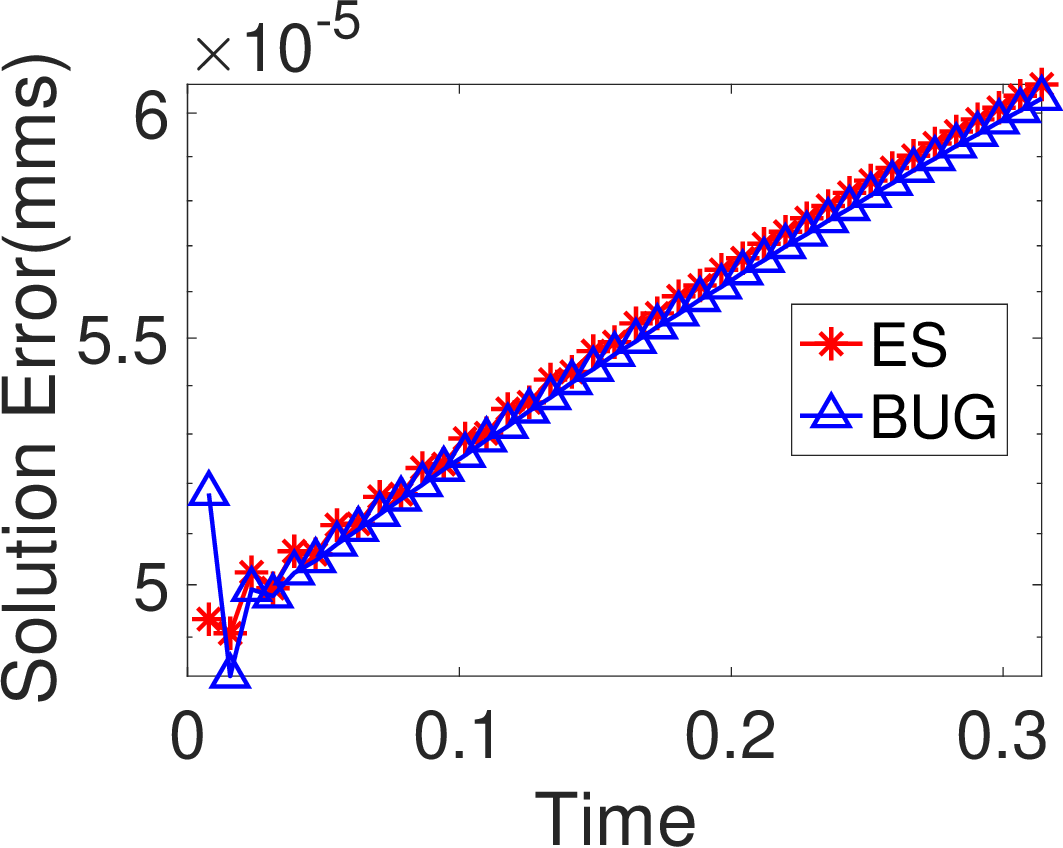} 
\hspace{0.1\linewidth}
\includegraphics[width=0.32\linewidth]{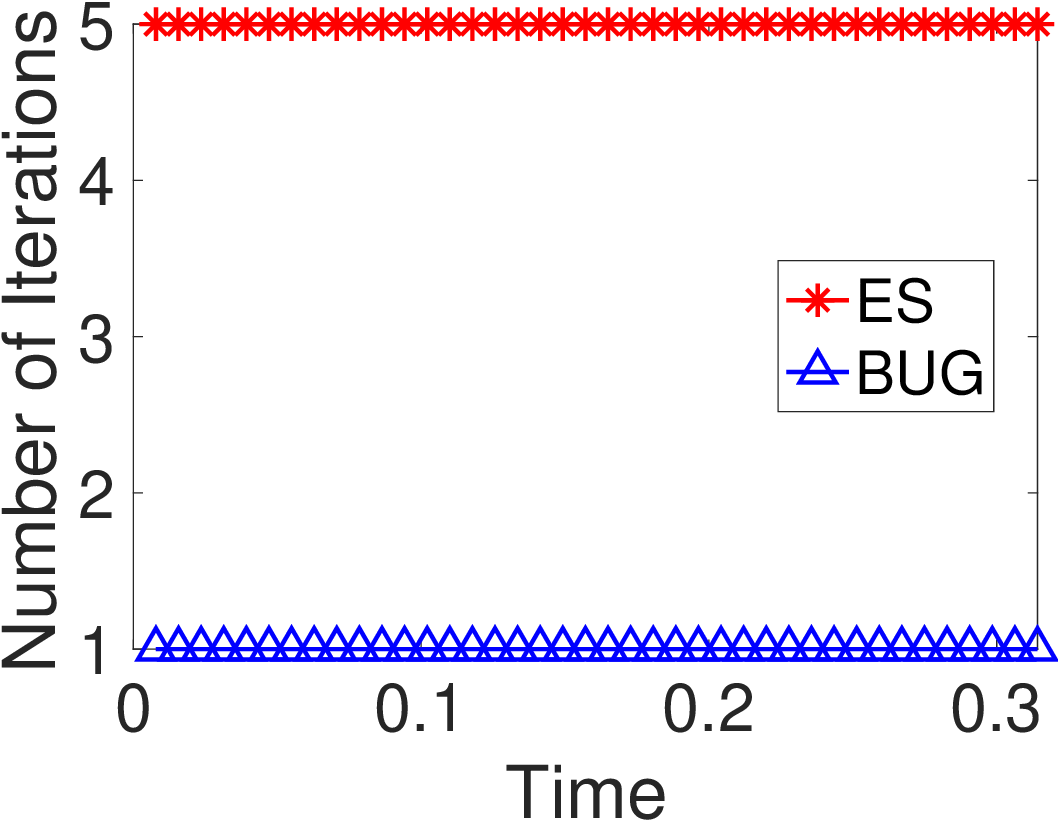} 
\\
\includegraphics[width=0.32\linewidth]{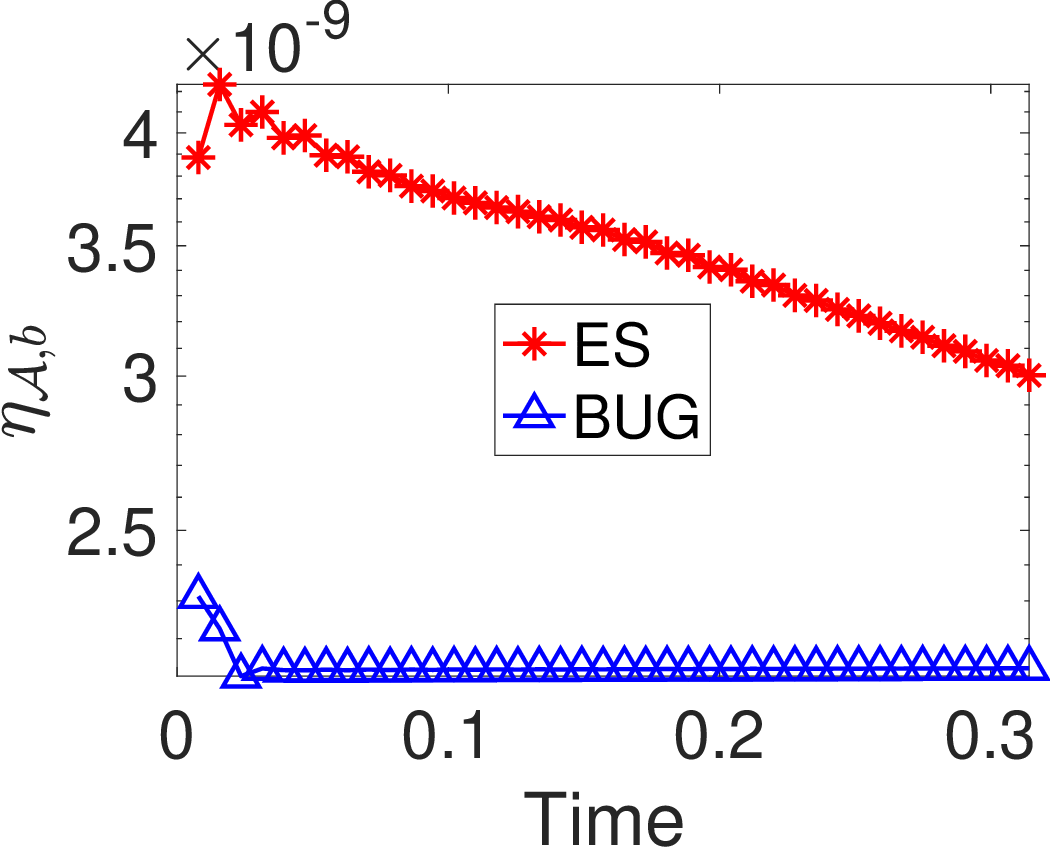} 
\includegraphics[width=0.32\linewidth]{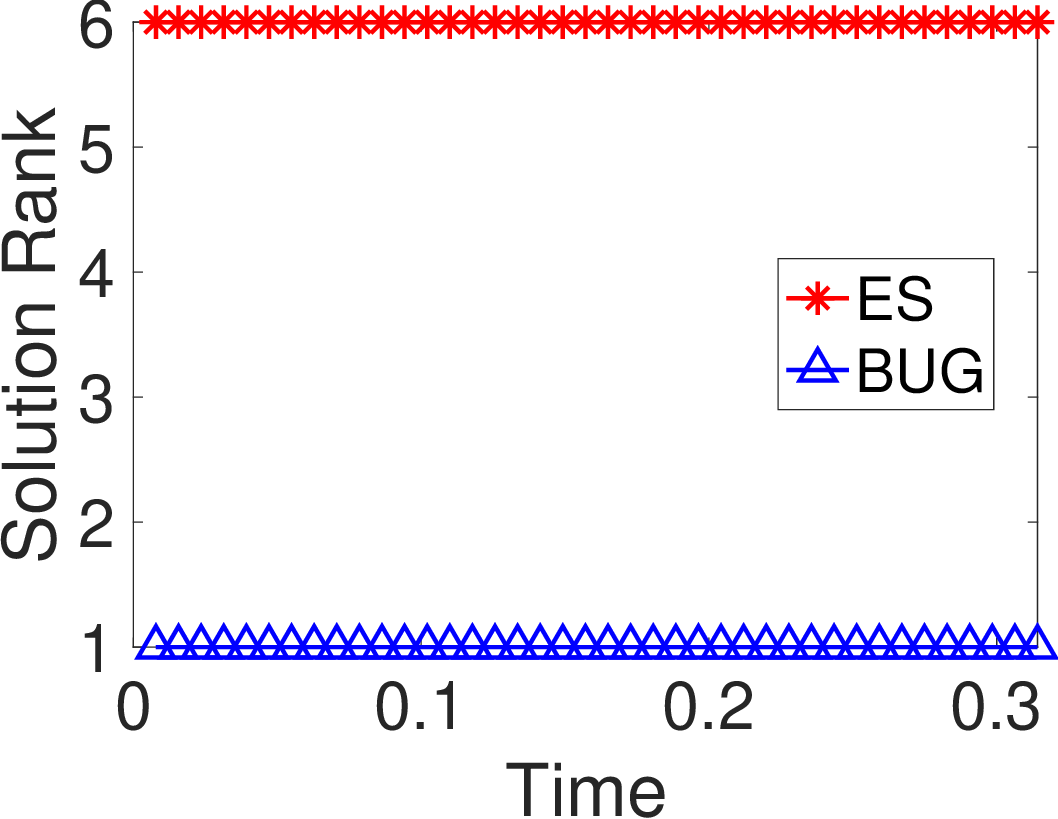} 
\includegraphics[width=0.32\linewidth]{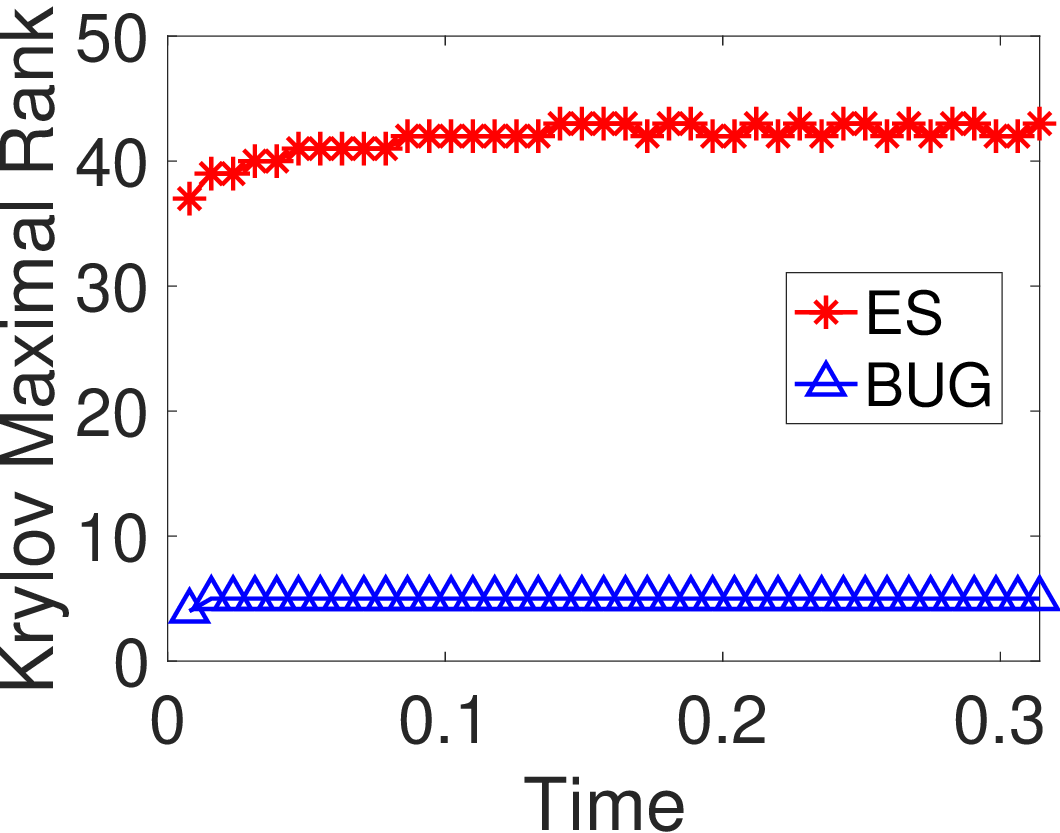} 
\caption{Example \ref{ex:highc}. Solving diffusion equation with high contrast variable coefficients \eqref{coefficient highcontrast_variable} and  manufactured solution \eqref{coefficient high contrast tobler}. Preconditioner: ES and BUG.  Rounding tolerance $\epsilon = \eta^2h^3$ for ES and $\epsilon = h^3$ for BUG. For $h = h_x = h_y  =7.81(-3)$, this figure displays the history of solution error, iteration number, $\eta_{\mathcal{A},b}$, solution rank, and maximal Krylov rank. 
} 
\label{figure highcontrast_variable h3}
\end{figure}

\end{example}

\subsection{Higher order schemes}
In this subsection, we consider higher order schemes. We use a fourth order finite difference spatial discretization with high order single step or multistep method:  Crouzeix’s three stage, fourth-order DIRK (diagonally implicit Runge-Kutta) method and fourth order backward differentiation formula (BDF) with mesh $\Delta t=O(h)$. In this case, the local truncation error is on the order of $\Delta t (\Delta t^4 + h^4) = O(h^5)$. For the BDF4 scheme and the inner stage solver for the DIRK4 method, the rounding tolerance $\epsilon$ is to make sure that its induced error is on the same order of the local truncation error, we need
$
h \epsilon (1+ \Delta t/h^2) = O(h^5),
$
which leads to $\epsilon = O(h^5)$. For the truncation tolerance, we take $\epsilon_2=O(h^4)$ to match the accuracy. To avoid very small rounding tolerance $\epsilon$, we test the algorithms at most to $h=h_x=h_y = 2/2^7$ which already leads to $h^5 =O(10^{-10})$. All other parameters are chosen as recommended in Section \ref{sec:parameterstudy}. The final time is $t_{\rm end} = 0.4\pi$ in this subsection. 

We note that to design high order BUG integrators,  is highly challenging \cite{Nakao:2023aa,Ceruti2024}. Typically one need to add many bases to the K- or L-step to ensure high order accuracy, which will increase the rank and slow down the computation. However, in this framework, this issue is not present any more, which is a major advantage of using the BUG  as a preconditioner rather than a direct solver.

\begin{example}[BDF4]\label{ex:bdf}
\textnormal{
    We  compute with the fourth order BDF scheme and and fourth order spatial discretization for \eqref{eq:mastereq} with the following variable coefficients
    \begin{align}\label{coefficient BDF_1}
a_1(x)& = 1+0.15\sin(\pi x),  & b_1(y) &= 1+0.1\cos(\pi y), \nonumber\\
a_2(x)&  = 0.15 = b_3(y),  & b_2(y) &= 1 = a_3(x), \nonumber\\
a_4(x)& = 1,  & b_4(y) &= 1 + 0.1 \cos(\pi y),
\end{align}
and manufactured solution
\begin{eqnarray}\label{mms BDF}
X(x,y,t) = \exp(-(x-0.1\sin(t))^2/0.12^2)\exp(-(y+0.1\cos(t))^2/0.12^2)\exp(-t).
\end{eqnarray} 
}

 \begin{table}[]  
 \begin{center} 
\begin{tabular}{| c | c | c | c|  c | c | c | c  |} 
\hline
{($\epsilon = h^5$)} & ES & ES & BUG  & BUG & hybrid & hybrid \\
  $h$ & error  & order  & error   & order & error & order \\
 \hline
 \hline
   1.25(-1) &   9.49(-3) & -- &   9.49(-3) & -- &  9.49(-3) & --  \\
 6.25(-2) &    4.03(-4) &  4.55 &  4.03(-4) &  4.55 &  4.03(-4) &  4.55 \\
   3.12(-2) &   2.80(-5) &  3.84 &  2.80(-5) &  3.84 &  2.80(-5) &  3.84 \\
 1.56(-2)  &   1.81(-6) &  3.94 &  1.81(-6) &  3.94 &  1.81(-6) &  3.94 \\
 \hline
\hline
{
 ($\epsilon = h^3$)} & ES & ES & BUG  & BUG & hybrid & hybrid \\
  $h$ & error  & order  & error   & order & error & order \\
 \hline
 \hline
   1.25(-1) &   9.33(-3) & -- &   9.31(-3) & -- &  9.07(-3) & --  \\
 6.25(-2) &    4.03(-4) &  4.53 &  4.03(-4) &  4.52 &  4.07(-4) &  4.47 \\
   3.12(-2) &   4.14(-5) &  3.28 &  2.81(-5) &  3.84 &  2.82(-5) &  3.84 \\
 1.56(-2)  &   4.97(-5) &  -0.26 &  1.81(-6) &  3.95 &  1.81(-6) &  3.96 \\
 \hline
 \end{tabular}
 \caption{Example \ref{ex:bdf}. Solving diffusion equation with variable coefficients \eqref{coefficient BDF_1} and manufactured solution \eqref{mms BDF}. BDF4    and forth order finite difference in space. Preconditioner: ES, BUG, and hybrid. For $h = h_x = h_y  \in \{1.25(-1), 6.25(-2), 3.12(-2), 1.56(-2)\}$, this table displays the solution error at the final time.   Rounding tolerance $\epsilon = h^5$ (top table) and $\epsilon = h^3$ (bottom table).}
 \label{table ex_BDF_const convergence}
 \end{center} 
 \end{table} 

\textnormal{
Table \ref{table ex_BDF_const convergence} displays the solution error at the final time for the ES, BUG, and hybrid preconditioners. When   rounding tolerance is $\epsilon = h^5$, the fourth-order convergence can be observed for all three preconditioners. However when $\epsilon = h^3,$ The ES preconditioner does not converge, while the BUG and hybrid preconditioner still converge with fourth order accuracy.
Figure \ref{figure ex_BDF_const} displays the history of solution error, $\eta_{\mathcal{A},b}$, solution rank, maximal Krylov rank, and  iteration number for three preconditioners. The Krylov ranks are not small for the ES preconditioner, while the  Krylov ranks are smaller for the BUG and hybrid preconditioners.}
\begin{figure}[h] 
\centering
\includegraphics[width=0.32\linewidth]{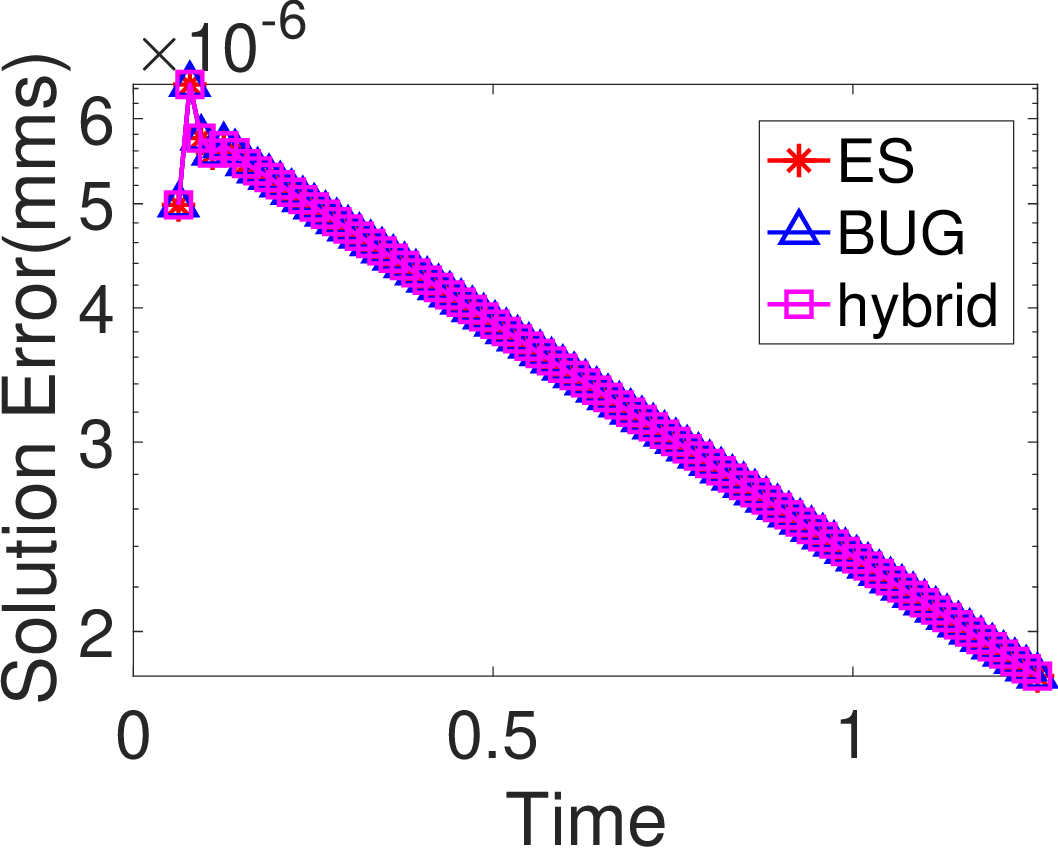} 
\hspace{0.1\linewidth}
\includegraphics[width=0.32\linewidth]{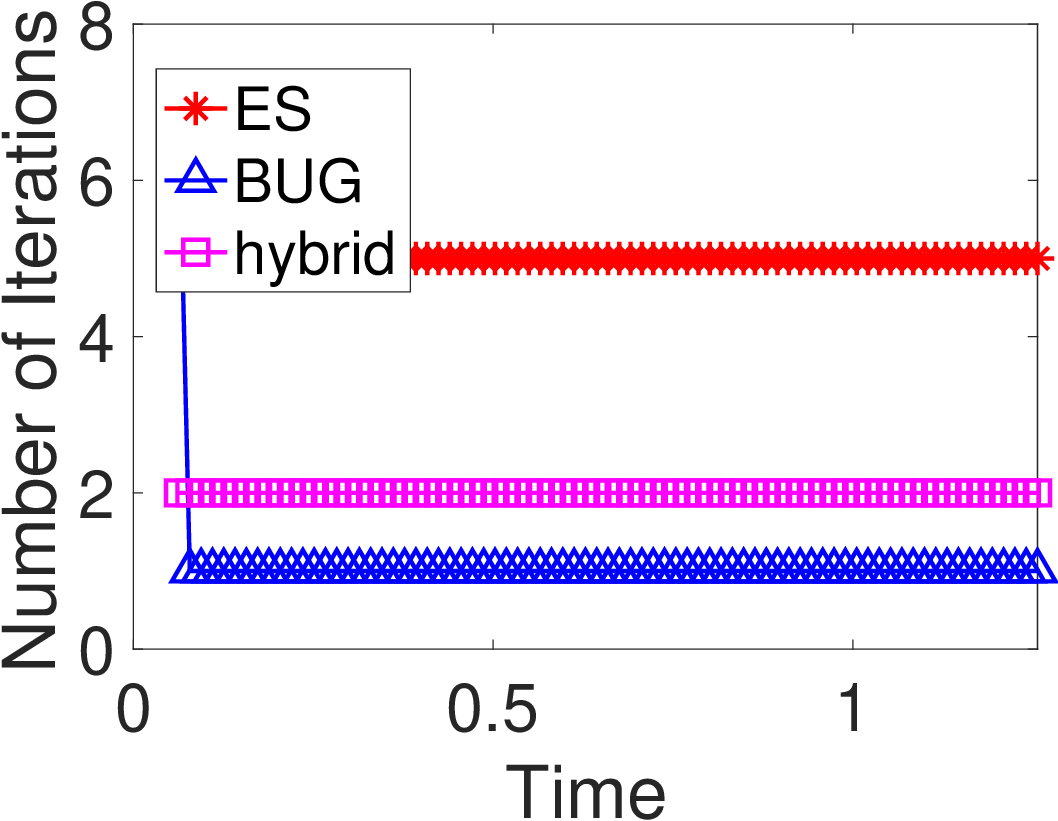} 
\\
\includegraphics[width=0.32\linewidth]{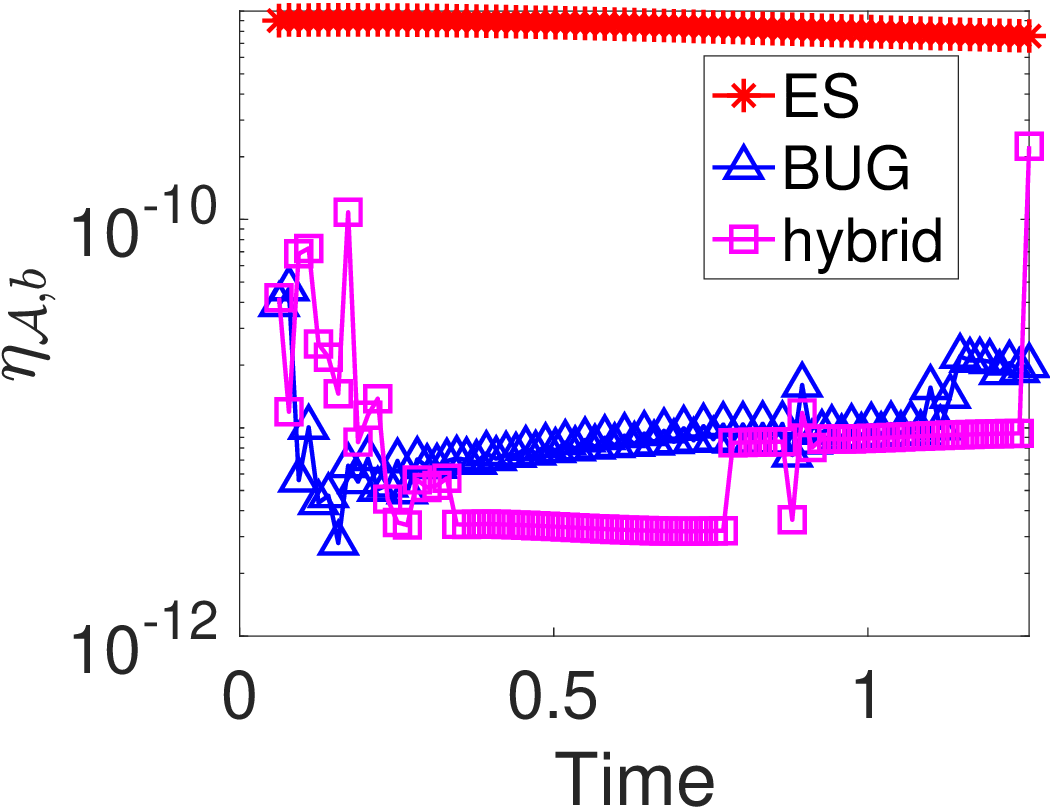} 
\includegraphics[width=0.32\linewidth]{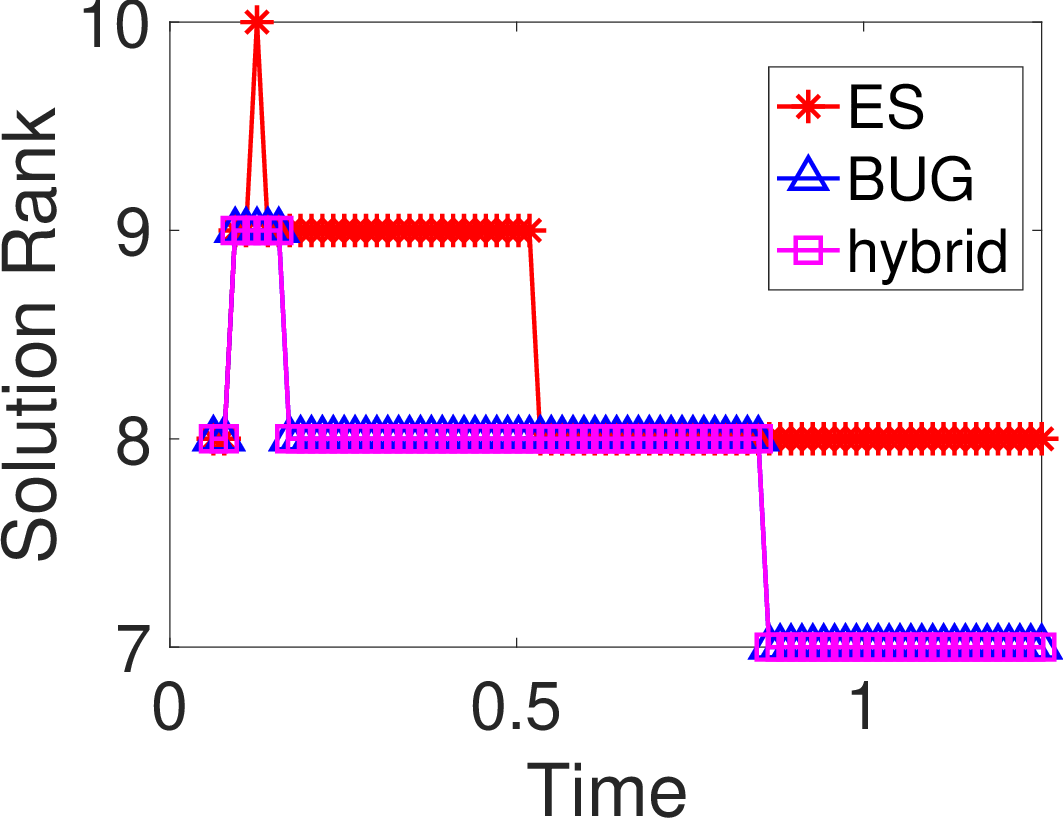} 
\includegraphics[width=0.32\linewidth]{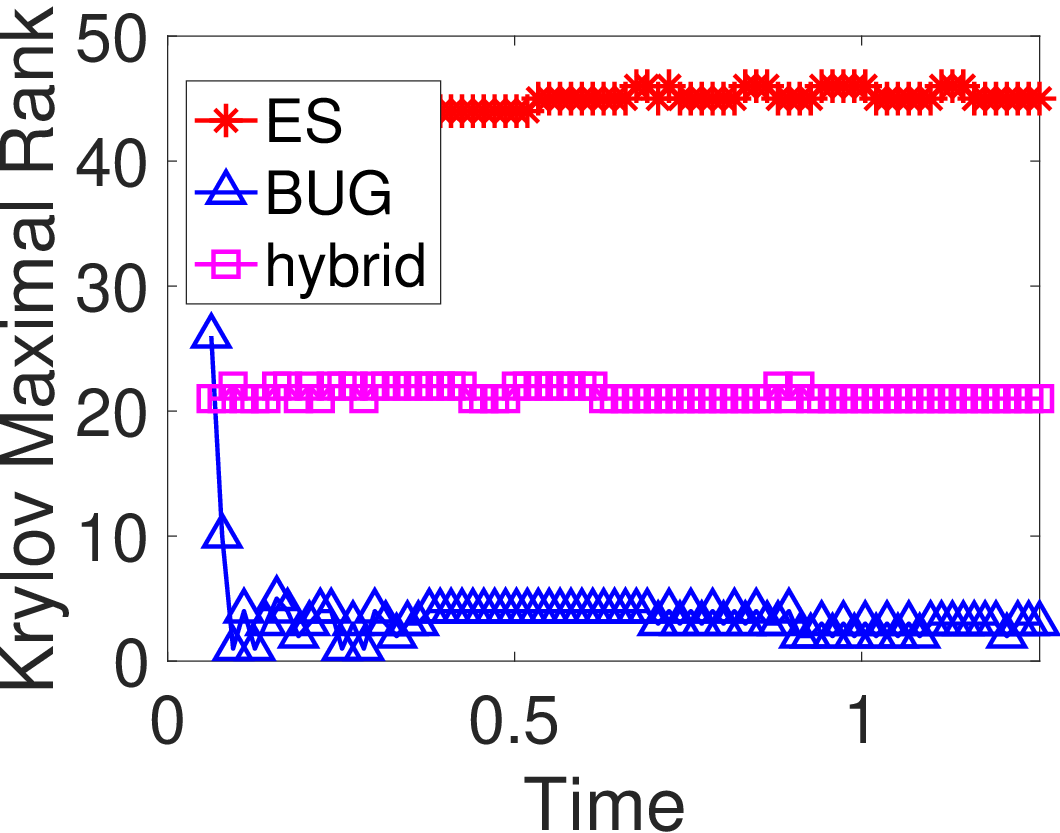} 
\caption{Example \ref{ex:bdf}. Solving diffusion equation with variable coefficient \eqref{coefficient BDF_1} and manufactured solution \eqref{mms BDF}.  BDF4    and forth order finite difference in space.  Preconditioner: ES, BUG, and hybrid. Fourth-order scheme with BDF. Rounding tolerance $\epsilon = h^5$. For $h=h_x = h_y  =1.56(-2)$, this figure displays the history of solution error, iteration number, $\eta_{\mathcal{A},b}$, solution rank, and maximal Krylov rank. 
} 
\label{figure ex_BDF_const}
\end{figure}
\end{example}

\begin{example}[Fourth order DIRK]\label{ex:dirk}
\textnormal{
We consider   the same problem as above, i.e. coefficients \eqref{coefficient BDF_1} and manufactured solution \eqref{mms BDF}
with fourth order DIRK scheme and fourth order spatial discretization. %
We compare the performance of the scheme with two different choice of $U,S,V.$ Those appear in the parameters in the BUG preconditioner, and also as initial guess for the all three preconditioners. 
Figure \ref{figure ex_DIRK X} and \ref{figure ex_DIRK X guess} display the history of solution error, $\eta_{\mathcal{A},b}$, solution rank, maximal Krylov rank, and  iteration number for three preconditioners. For the iteration number, we sum up the iteration numbers in the three inner stages of DIRK, while the maximal Krylov rank is taken as the maximal for the three stages.   In Figure \ref{figure ex_DIRK X}, the initial guess for the implicit {solvers} at DIRK inner stages is taken as the numerical solution $X^n,$ while in Figure \ref{figure ex_DIRK X guess}, it is taken as the inner stage solution at the previous time step, i.e. we use $X^{n-1, (j)}$ which is the $j-$th inner stage at the previous time step for the current $j-$th inner stage. We can see a visible difference between the two approaches.  With $X^n$ as the initial guess, all three preconditioners give excessively large solution rank, thus large maximal Krylov rank. With initial guess taken as previous inner stage values, the solution rank is at a reasonable value. This example indicates the advantage of using $X^{n-1, (j)}$ as the initial guess for the inner stage of Runge-Kutta methods. For this example, the hybrid conditioner offers the best overall quality in solution rank and maximal Krylov rank.
Table \ref{table ex_DIRK convergence} displays the solution error at the final time for the ES, BUG, and hybrid preconditioners with initial guess $X^{n-1, (j)}$. All three preconditioners have errors that are similar to a full rank solver. 
}

\begin{figure}[h] 
\centering
\includegraphics[width=0.32\linewidth]{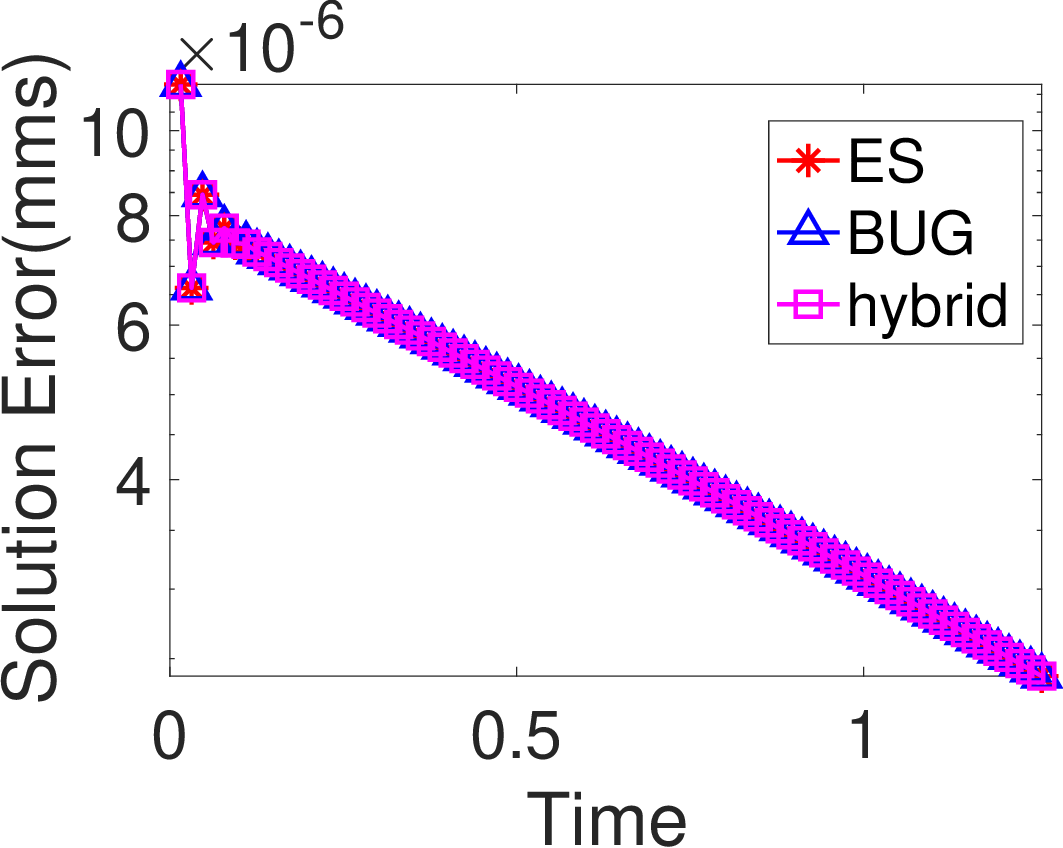} 
\hspace{0.1\linewidth}
\includegraphics[width=0.32\linewidth]{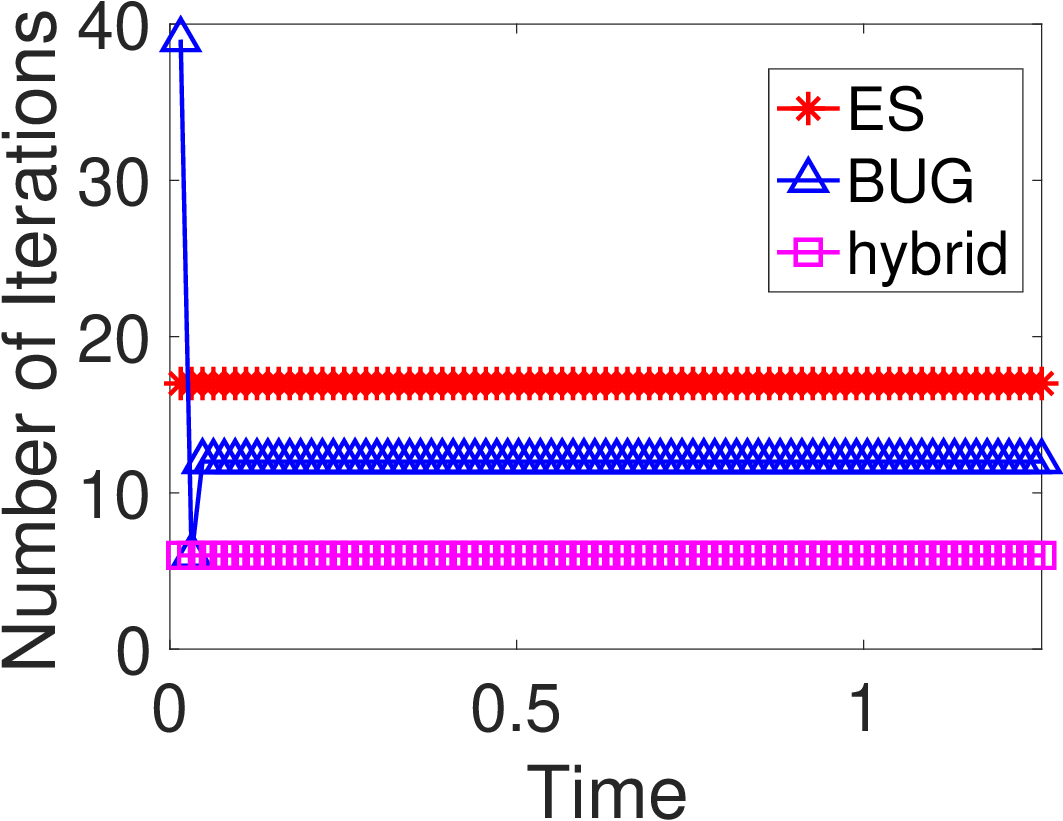} 
\\
\includegraphics[width=0.32\linewidth]{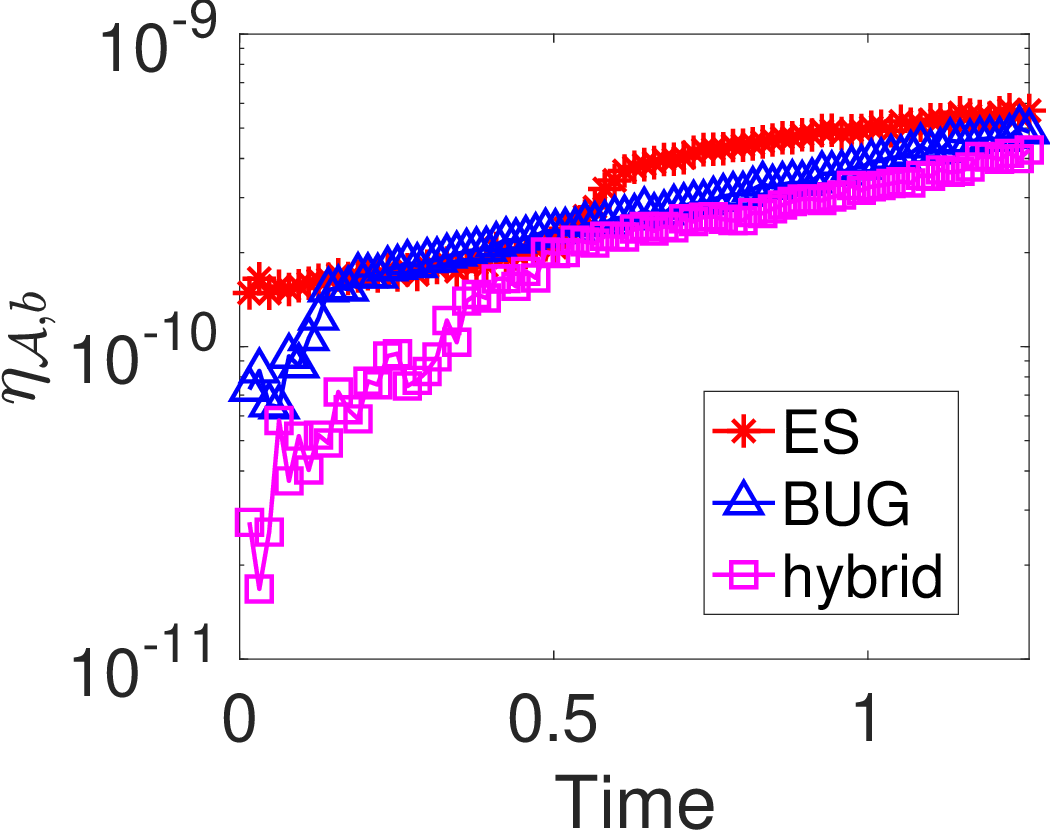} 
\includegraphics[width=0.32\linewidth]{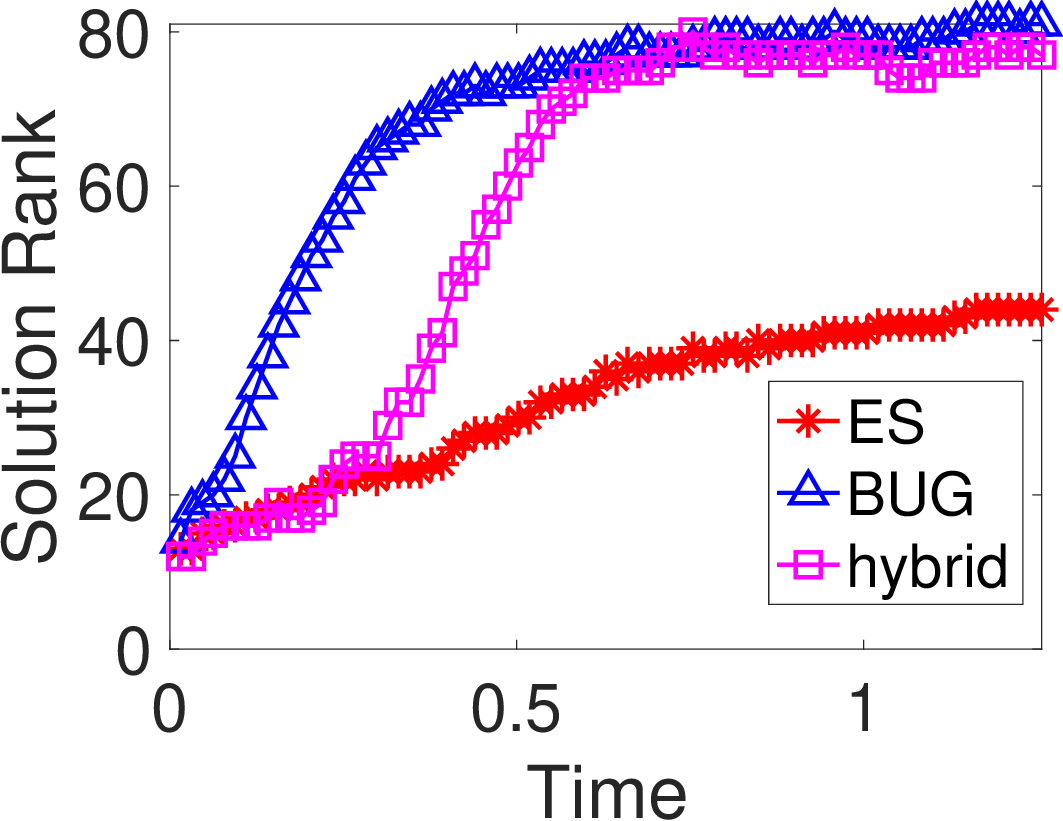} 
\includegraphics[width=0.32\linewidth]{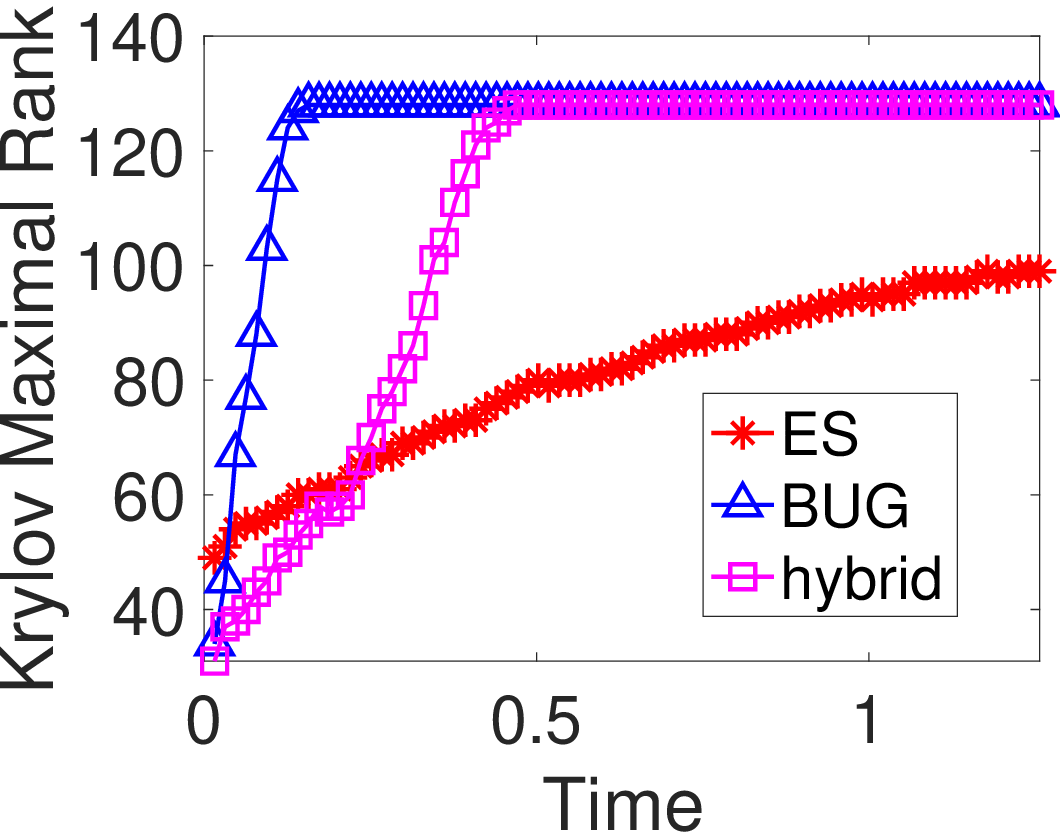} 
\caption{Example \ref{ex:dirk}. Solving diffusion equation with coefficients \eqref{coefficient BDF_1} and manufactured solution \eqref{mms BDF}.
Initial guess: $USV^T=X^{n}$. Fourth-order DIRK  and forth order finite difference in space. Preconditioner: ES, BUG, and hybrid. Rounding tolerance $\epsilon = h^5$ in implicit solver and $h^4$ in last step update $X^{n} \to X^{n+1}$. For $h = h_x = h_y  =1.56(-2)$, this figure displays the history of solution error, iteration number, $\eta_{\mathcal{A},b}$, solution rank, and maximal Krylov rank. For the iteration number, we sum up the iteration numbers in the three inner stages of DIRK, while the maximal Krylov rank is taken as the maximal for the three stages.
} 
\label{figure ex_DIRK X}
\end{figure}
 
\begin{figure}[h] 
\centering
\includegraphics[width=0.32\linewidth]{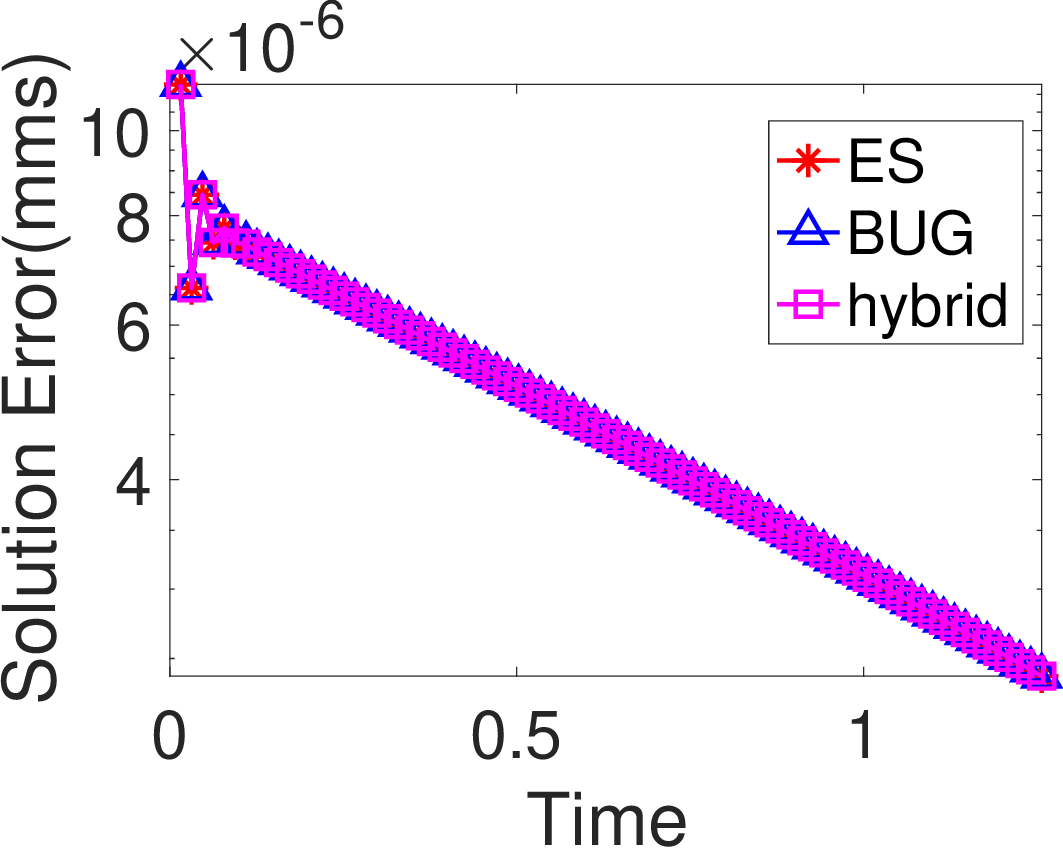} 
\hspace{0.1\linewidth}
\includegraphics[width=0.32\linewidth]{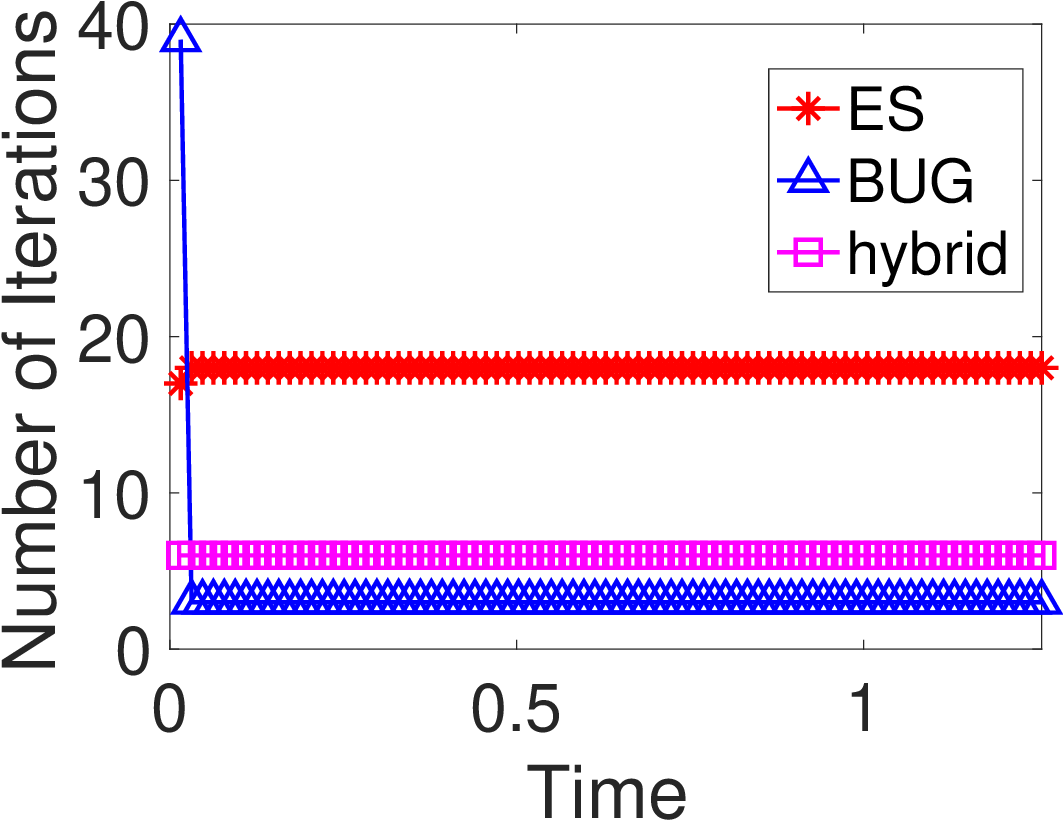} 
\\
\includegraphics[width=0.32\linewidth]{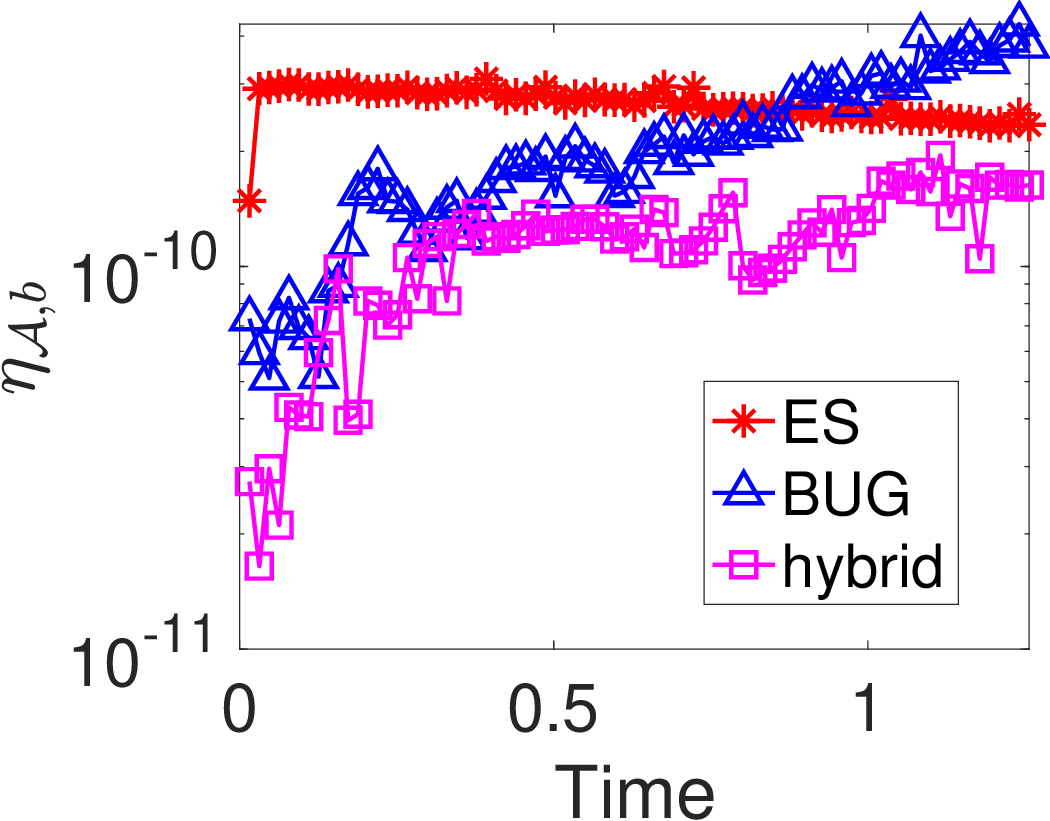} 
\includegraphics[width=0.32\linewidth]{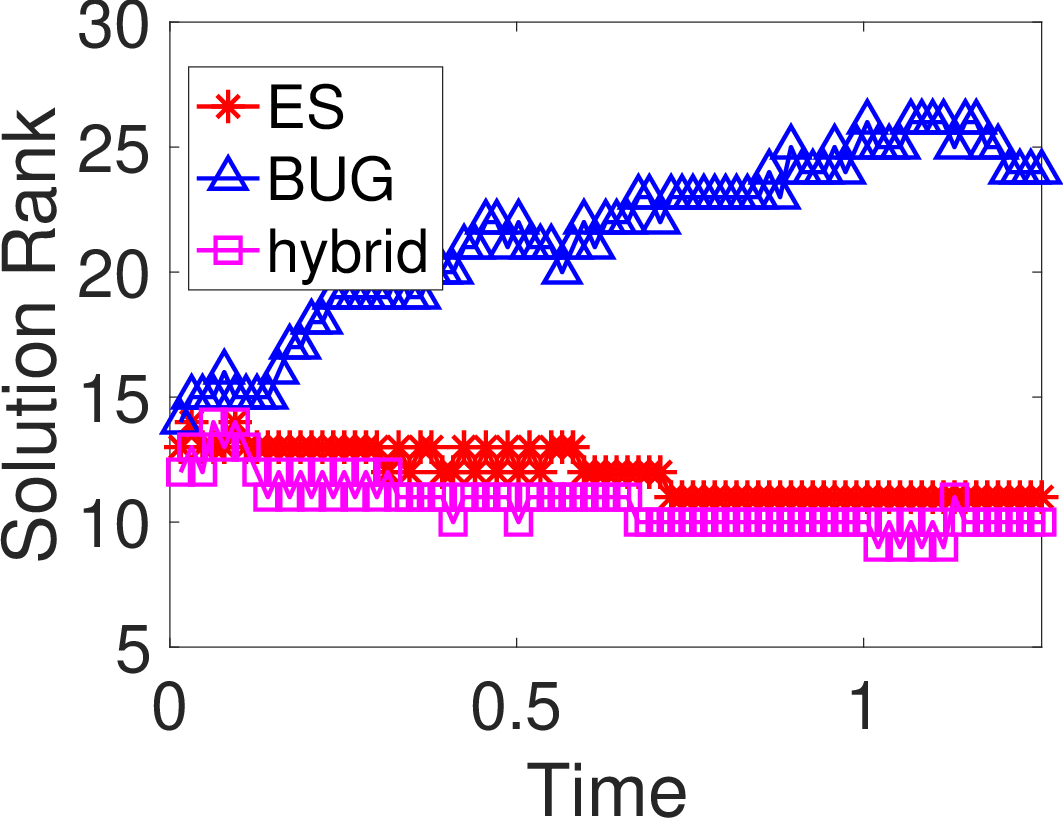} 
\includegraphics[width=0.32\linewidth]{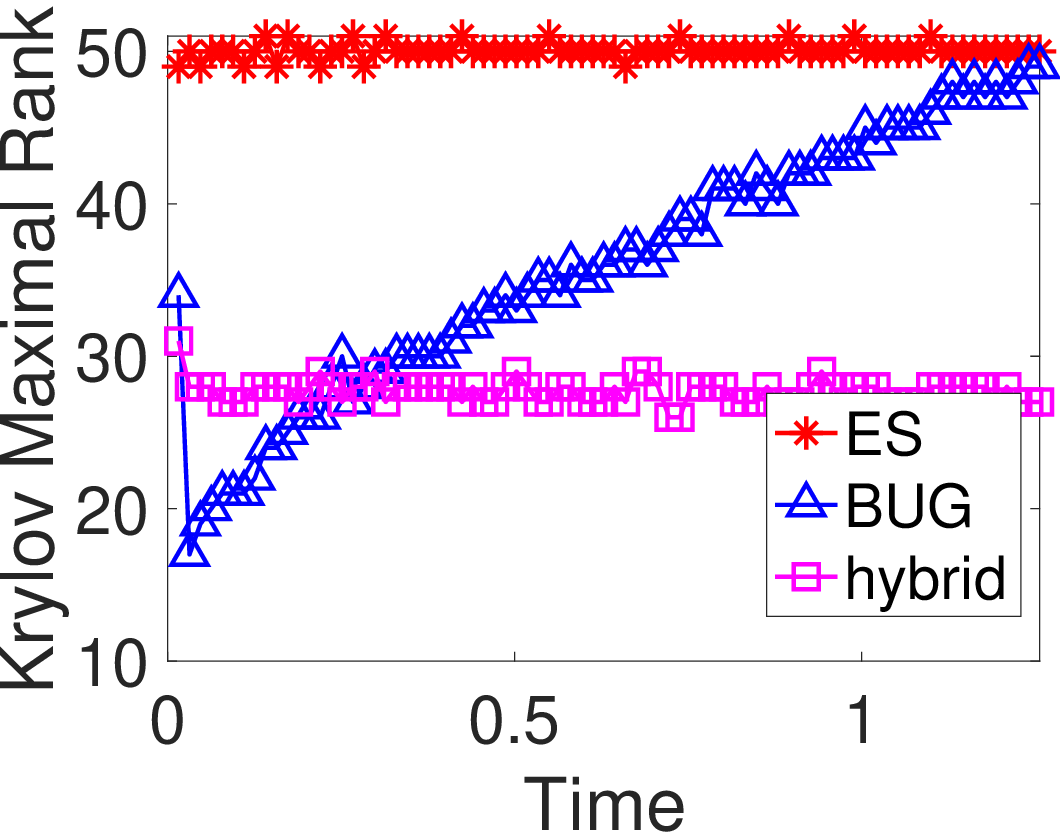} 
\caption{Example \ref{ex:dirk}. Solving diffusion equation with coefficients \eqref{coefficient BDF_1} and manufactured solution \eqref{mms BDF}.
Initial guess: $USV^T=X^{n-1, (j)}$ for updating implicit solver at the $j-$th inner stage in DIRK. Fourth-order DIRK  and forth order finite difference in space. Preconditioner: ES, BUG, and hybrid. Rounding tolerance $\epsilon = h^5$ in implicit solver and $h^4$ in last step update $X^{(n)} \to X^{(n+1)}$. For $h = h_x = h_y  =1.56(-2)$, this figure displays the history of solution error, iteration number, $\eta_{\mathcal{A},b}$, solution rank, and maximal Krylov rank. For the iteration number, we sum up the iteration numbers in the three inner stages of DIRK, while the maximal Krylov rank is taken as the maximal for the three stages.
} 
\label{figure ex_DIRK X guess}
\end{figure}

\end{example}

 \begin{table}[]  
 \begin{center} 
\begin{tabular}{| c | c | c | c| c|   c | c | c | c  |} 
\hline
  & ES & ES & BUG & BUG &  hybrid & hybrid & full rank & full rank \\
 $h$ & error   & order   & error   & order   & error   & order & error & order   \\
 \hline
 \hline
   1.25(-1) &    8.13(-3) & -- &    8.19(-3) & -- &   8.15(-3) & -- & 8.15(-3) & -- \\
 6.25(-2) &   4.10(-4) &  4.31 &  4.10(-4) &  4.32 &  4.10(-4) &  4.31 &  4.10(-4) &  4.31 \\
   3.12(-2) &   2.98(-5) &  3.78 &  2.98(-5) &  3.78 &  2.98(-5) &  3.78 &  2.98(-5) &  3.78 \\
 1.56(-2)  &   2.38(-6) &  3.64 &  2.38(-6) &  3.64 &  2.38(-6) &  3.64 &  2.38(-6) &  3.64 \\
 \hline
 \end{tabular}
 \caption{Example \ref{ex:dirk}. Solving diffusion equation with variable coefficients \eqref{coefficient BDF_1} and manufactured solution \eqref{mms BDF}.
 Fourth-order DIRK time stepping and forth order finite difference in space. Initial guess: $X^{n-1, (j)}$ for updating implicit solver at the $j-$th inner stage in DIRK.  Preconditioner: ES, BUG, and hybrid, in comparison with a full rank solver.   This table displays the solution error at the final time and the associated order.  
 }  
 \label{table ex_DIRK convergence}
 \end{center} 
 \end{table}

\section{Conclusions and future work}
\label{sec:conclude}
In this work, we propose   new preconditioners for the low rank GMRES schemes for implicit time discretization of the matrix differential equations. The preconditioner is based on the BUG method, which is a DLRA of the matrix differential equation on a fixed low rank manifold. The BUG method as a direct solver  is subject to the modeling error and does not have convergence guarantee for general problems.  On the other hand, when used as a preconditioner, it offers robust performance.  Moreover, the BUG preconditioner framework can easily accommodate arbitrary high order time stepping without the need to augment/enlarge the basis as done in variants of BUG \cite{Nakao:2023aa,Ceruti2024}. 

The BUG preconditioner is a nonlinear preconditioner. We perform a detailed study on   its parameters, including tolerance parameter, restart parameter and stopping criteria. We conclude a frequent restart with tolerance selected based on local truncation error of the scheme is effective. This is justified by stability and convergence analysis. We further propose a hybrid version where we alternate between BUG and ES preconditioner.
 The numerical evidence has suggested that BUG preconditioner and the hybrid version generally require a very 
 low iteration number except for the few initial time steps, and it can control the maximal Krylov rank well.  Moreover, compared to the standard ES preconditioner, it is less sensitive to larger rounding tolerance and is blind to the operator being inverted. We show that for 
 the challenging high contrast anisotropic cross diffusion example, the BUG preconditioner is more effective than the ES preconditioner.
Future  work includes higher dimensional extensions and applications.

\subsection*{Acknowledgment and Disclamer}
This material is based upon work supported by the U.S. Department of Energy, Office of Science, Advanced Scientific Computing Research (ASCR), under Award Number DE-SC0025424, and DOE Mathematical Multifaceted Integrated Capability Centers (MMICC) center grant DE-SC0023164.

This report was prepared as an account of work sponsored by an agency of the United States Government. Neither the United States Government nor any agency thereof, nor any of their employees, makes any warranty, express or implied, or assumes any legal liability or responsibility for the accuracy, completeness, or usefulness of any information, apparatus, product, or process disclosed, or represents that its use would not infringe privately owned rights. Reference herein to any specific commercial product, process, or service by trade name, trademark, manufacturer, or otherwise does not necessarily constitute or imply its endorsement, recommendation, or favoring by the United States Government or any agency thereof. The views and opinions of authors expressed herein do not necessarily state or reflect those of the United States Government or any agency thereof.
\bibliographystyle{abbrv}
\bibliography{ref.bib}

\appendix

\section{ES preconditioner}
We follow \cite{bachmayr2023low} to define the ES preconditioner for operators of the form
\begin{eqnarray*}
  A =   \sum_{i=1}^2 \bigotimes_{j=1}^2 B_{i,j}, \qquad B_{i,j} = \left\{\begin{array}{cc}
        A_{i}, & i=j, \\
        I, & i\not=j,
    \end{array} \right.
\end{eqnarray*}
with self-adjoint and positive definite matrices $A_i$. The construction of the ES preconditioner that approximates $A^{-1}$ is based on the approximation of  
$$
\frac{1}{t} = \int_{0}^\infty e^{-st} {\,\rm d} s, \quad t>0
$$
by Sinc quadrature and passing $t$ to $A$ in appropriate sense.

The following Lemma given by \cite[Corollary 4.3]{bachmayr2023low} states the approximation of $t^{-1}$ by ESs.
\begin{lemma}
    Let $\delta^* = \delta_0 + \eta < 1$ with $\delta_0,\eta \in (0,1)$, and $T>1$. With 
    \begin{eqnarray*}
        \alpha &=& \frac{2\pi}{\ln3 + |\ln (\cos 1)| + |\ln(\delta_0/2)|}, \\
        m &=& \lceil \alpha^{-1} \ln |\ln(\delta_0/2)| \rceil,\quad n = \lceil \alpha^{-1} (|\ln(\eta/2)| + \ln T) \rceil,
    \end{eqnarray*}
    define the ES preconditioner
    $$
    S_{1,\delta_0,\eta,T} = \alpha \sum_{-n}^m e^{k \alpha} \exp(-e^{k\alpha} t).
    $$
Then 
$$
|t^{-1} - S_{1,\delta_0,\eta,T}| \le \delta t^{-1}, \qquad \mbox{for all } t \in [1, T].
$$
\end{lemma}
Passing from $t^{-1}$ to $A^{-1}$ is based on the following Lemma given by \cite[Proposition 4.1]{bachmayr2023low}.
\begin{lemma}
Let the matrix $A$ be given by
\begin{eqnarray*}
  A =   \sum_{i=1}^2 \bigotimes_{j=1}^2 B_{i,j}, \qquad B_{i,j} = \left\{\begin{array}{cc}
        A_{i}, & i=j, \\
        I, & i\not=j,
    \end{array} \right.
\end{eqnarray*}
with self-adjoint and positive definite matrices $A_i \in \mathbb{R}^{N\times N}$ and assume that $\sigma(A) \subset [1, \mbox{cond}(A)]$. Given a continuous function $\tilde{f}: \mathbb{R}^+ \to \mathbb{R}^+$,   assume that for some $\delta^* \in (0,1)$ we have $t^{-1}-\tilde{f}(t)\le \delta t^{-1}$ for all $t \in \sigma(A) \subset [1, \mbox{cond}(A)]$. Then
$$
(1-\delta^*) \langle A^{-1} v,v \rangle \le \langle \tilde{f}(A) v,v \rangle \le (1+\delta^*) \langle A^{-1} v,v \rangle \mbox{ for all } v \in \mathbb{R}^N.
$$
\end{lemma}
With the above two lemmas, we are able to estimate the condition number using the ES preconditioner.
Let $A$, $S_{1,\delta_0,\eta,T}$ be given in the above two lemmas with their corresponding assumptions. Then it is expected from   \cite[Proposition 4.7]{bachmayr2023low} that
$
\mbox{cond} (A \mathcal{M}_{\delta,T})$ can be estimated by $ \frac{1+\delta^*}{1-\delta^*}
$, 
where $T \ge \mbox{cond} (A)$, and 
    $$
\mathcal{M}_{\delta^*,T} = S_{1,\delta^*/2,\delta^*/2,T} (A).
$$

In this paper, to define the ES preconditioner for equation \eqref{2d variable}, we consider the diagonal part of the diffusion coefficients, and use their domain averages. Namely, the operator to be used for the ES preconditioner is the approximate inverse corresponding to the scheme for $$u_t=\overline{a_1} \overline{b_1} u_{xx} +\overline{a_4} \overline{b_4} u_{yy},$$ where $\overline{\cdot}$ denotes the domain average of the function. This   will yield the following type of operator in the time stepping
$$
A = I\bigotimes I - \tau ( I \bigotimes C +  D \bigotimes I ),
$$
for a given $\tau$ (which is related to the time stepping). %
We simply  rewrite the operator as
$$
 I \bigotimes (0.5I-\tau   C) + (0.5I-\tau   D)\bigotimes I .
$$
Let $A_1 = (0.5I-\tau    C)$ and $A_2 = (0.5I-\tau   D)$, the exponential sum preconditioner reads
$$
\mathcal{M}_{\delta^*,T} = \alpha \sum_{-n}^m e^{k \alpha} \exp(-e^{k\alpha} A_2) \bigotimes  \exp(-e^{k \alpha} A_1).
$$
The choice of the parameter $T$ depends on the condition number of $\mathcal{A}$. In our applications,  the condition number of $\mathcal{A}$ is on the order of $\mathcal{O}(\tau \eta^{-2} h^{-2})$ and we chose $T = 4\tau \eta^{-2} h^{-2}$, where $\eta =1/10$ for high contrast and $\eta=1$ otherwise. 
We chose empirically $\delta^*=0.2$ in the numerical examples.
For every matrix  in SVD form $b= U_b S_b V_b^T$, we 
round $\mathcal{M}_{\delta^*,T} b$ by the truncation sum algorithm with a given tolerance and a maximal rank. 
We chose this rounding tolerance the same as the lrGMRES rounding tolerance $\epsilon$. 

\end{document}